\documentclass[12pt,leqno]{article}% ,draft

\usepackage{amsfonts}
\usepackage{amsmath}
\usepackage{amssymb}
\usepackage{amscd}
\usepackage{mathtools}
\usepackage{here}
\newcommand{\rarrow}[1]{{\buildrel #1 \over \longrightarrow}}

\makeatletter
  \@addtoreset{equation}{section}
\makeatother
\makeatletter

\newcommand{\Z}{\mathbb{Z}}

\newcommand{\Ker}{\operatorname{Ker}}

\renewcommand{\Im}{\operatorname{Im}}
\newcommand{\mtoda}[5]{\left\{
\begin{array}{c}%
{#1} \\
{#4}
\end{array}\hspace{-1.5mm},\hspace{0.5mm}%
\begin{array}{c}%
{#2} \\
{#5}
\end{array}\hspace{-1.5mm},\hspace{0.5mm}%
{#3}
\right\}}

\newtheorem{thm}{Theorem}[section]

\newtheorem{prop}[thm]{Proposition}

\newtheorem{lem}[thm]{Lemma}

\newtheorem{conj0}{Conjecture}
\newtheorem{rem}[thm]{Remark}

\newtheorem{exam}[thm]{Example}

\newenvironment{proof}{\noindent{\bf Proof.}}
{\noindent \ \hfill$\Box$\par}

\begin{document}

\title{The $23$-rd and $24$-th homotopy groups of the $n$-th rotation group
\\
}

\author{Y. Hirato, J. Lee, 
T. Miyauchi, J. Mukai, M. Ohara\\}
\date{}
\maketitle
\noindent
{\bf Abstract}
We denote by $\pi_k(R_n)$ the $k$-th homotopy group of the $n$-th rotation group $R_n$
and  $\pi_k(R_n:2)$ the 2-primary components of it.
We determine the group structures of
$\pi_k(R_n:2)$ for $k = 23$ and $24$ by use of the fibration $R_{n+1}\rarrow{R_n}S^n$.
The method is based on Toda's composition methods. 
\section{Introduction and statements of results}
\hspace{2ex}

Steenrod \cite{St}, Sugawara \cite{Su}, Kervaire \cite{Ke} and Mimura \cite{Mi3}
had studied the $k$-th homotopy groups $\pi_k(R_n)$ of the $n$-th rotation group $R_n$.
The group structures
of $\pi_{k}(R_n)$ for $k \le 22$ had been determined in this period ~\cite{St}, \cite{Su}, \cite{Ke}, \cite{Enc}, \cite{Lu}, \cite{Ka}, 
\cite{KM1}, \cite{KM2} and \cite{HKM}.

According to \cite{MT2} and \cite{Mi3}, the group structures of 
$\pi_{23}(R_n)$ for
$n \leq 9$ are known up to extensions in $7 \le n \le 9$ \cite{Mi3}.
By inspecting \cite{Lu}, the group structures of $\pi_{24}(R_n)$ 
for $3\leq n\leq 6$ are obtained. 

In this paper, we determine the $23, 24$-th homotopy groups $\pi_{23}(R_n)$ and $\pi_{24}(R_n)$, respectively. 

Let $\pi_{n+k}(S^n)$ stand for the $(n+k)$-th the homotopy groups of \ 
$S^n$. Our method is the composition methods \cite{T}.
We freely use generators and relations in the $2$-primary components 
$\pi^n_{n+k}$ of $\pi_{n+k}(S^n)$ for $k \leq 24$ \cite{T}, \cite{MT1}, \cite{Mi1}, \cite{Od1}. 

We denote by $\iota_n$ the homotopy class of the the identity map of $S^n$ and by $\iota_{R_n}$ the homotopy class of the identity map of $R_n$. For elements $\alpha_i\in\pi_k(R_n)$\ for\ $i=1, 2$\ and\ $\beta\in\pi_h(S^k)$, we have 
$$
(\alpha_1+\alpha_2)\beta = \alpha_1\beta + \alpha_2\beta.
$$

In particular, for an integer $t$, we have 
$$
t\iota_{R_n}\circ\alpha = t\alpha = \alpha\circ t\iota_k\ \mbox{for} \ \alpha\in\pi_k(R_n).
$$

We denote by $R^n_k$ 
the $2$-primary components $\pi_k(R_n :2)$. 
Our main tool is to use the following exact sequence induced from the fibration $p_{n+1} : R_{n+1} \to S^{n+1}$ : 
\begin{equation}\label{seqk_n}\tag*{$(k)_{n}$}
\pi^n_{k+1}\rarrow{\Delta}R^n_k
\rarrow{i_*}R^{n+1}_k\rarrow{p_*}\pi^n_k
\rarrow{\Delta}R^n_{k-1},
\end{equation}
where $i=i_{n+1}: R_n \to R_{n+1}$ is the inclusion, $p=p_{n+1}: R_{n+1} \to S^n$ is the projection
and $\Delta$ is the connecting map.

For an element $\alpha \in \pi^n_k$, we denote by
$[\alpha] \in R^{n+1}_k$ an element satisfying $p_*[\alpha] =\alpha$.
Though $[\alpha]$ is only determined modulo $\Im\ i_* ={i_{n+1}}_*(R^n_{k})$, we will sometimes give restrictions on $[\alpha]$ to fix it more concretely.
For the inclusion $i_{n+1,m}:R_{n+1} \to R_m$ for $n+1 \leq m$, we set $[\alpha]_m={i_{n+1,m}}_*[\alpha] \in R^m_k$. 

In this paper, we will be able to improve some relations in $\pi^n_{k+n}$ for $k = 23$  and $24$. For example, we obtain the relation 
$\Delta(\nu^2_{19}) = [\sigma^2_9]_{19}\eta_{23}$ (Proposition 6.14).
This implies $P(\nu^2_{39}) = \psi_{19}\eta_{42} 
\equiv \bar{\sigma}'_{19} + \delta_{19}\ \bmod\ \sigma_{19}\bar{\mu}_{26}$ which improves \cite [(5.39)]{MMO}.

To see that calculation is affected by the choice of Toda brackets where the generator lives, let us raise a generator $[\varepsilon_{15}]$ as an example. 
We recall $\varepsilon_{15}\in\{\eta_{15},2\iota_{16},\nu^2_{16}\}_1\ 
\bmod\ \eta_{15}\sigma_{16}$ \cite[(6.1)]{T}. 
We can take $[\varepsilon_{15}]$ in $\{[\eta_{15}], 2\iota_{16}, \nu^2_{16}\}_1$. By using this choice, we have $\Delta(\nu^2_{18}) 
= [\varepsilon_{15}]_{18}$ (Proposition \ref{R1923}).

For any element $[\varepsilon_{15}]\in R^{16}_{23}$, we have 
$\Delta(\nu^2_{18})\equiv [\varepsilon_{15}]_{18} \bmod\ 
[\sigma^2_9]_{18}$. 
At best we will be able to obtain the relation $[\varepsilon_{15}]_{22} = 0$

We also reconsider the choice of $[\bar\varepsilon_9]\in R^{10}_{24}$. We take this element in a certain matrix Toda bracket as in the proof of Proposition \ref{10}, 
which has also been developed responding to these calculations.  

\bigskip

\bigskip

The group structures of $\pi_{23+k}(R_n)$ for $k+19 \le n \le k+24$ with $0\leq k\leq 1$ are obtained by \cite{Ke}.

We recall the splitting \cite{BM} for $k\leq 2n - 3$ and $n\geq 13$:
\[
\pi_k(R_n)\cong \pi_k(R_{\infty})\oplus\pi_{k+1}(V_{2n, n}),
\]
where $V_{m,r} = R_m/R_{m-r}$ for $m\geq r$ is the Stiefel manifold. By use of
this splitting, \cite{Bo} and the group structures of
$\pi_{k+1}(V_{2n, n})$ \cite{HoM}, we obtain the group
structures of $\pi_{k+23}(R_n)$ for $n\geq k+13$ with $k=0$ or $1$.

For the element $\kappa_7$, we adopt the renamed one %, that is,
which satisfies the equation\ $2\kappa_7=\bar{\nu}_7\nu^2_{15}$ \cite[(15.5)]{MT1}. Our result is stated as follows.

\begin{thm} \label{thm1}
\begin{enumerate}
\item $\pi_{23}(R_3:2)=\Z_2\{[\eta_2]\nu'\bar{\mu}_6\}
\oplus\Z_2\{[\eta_2]\nu'\eta_6\mu_7\sigma_{16}\}$;

$\pi_{23}(R_4:2)=\Z_2\{[\eta_2]_4\nu'\bar{\mu}_6\}
\oplus\Z_2\{[\eta_2]_4\nu'\eta_6\mu_7\sigma_{16}\}
\oplus\Z_2\{[\iota_3]\nu'\bar{\mu}_6\}\\ 
\oplus\Z_2\{[\iota_3]\nu'\eta_6\mu_7\sigma_{16}\}$;

$\pi_{23}(R_5:2)=\Z_2\{[\nu_4\sigma'\mu_{14}]\}
\oplus\Z_2\{[\nu_4(E\zeta')]\}
\oplus\Z_2\{[\nu_4\eta_7\bar{\varepsilon}_8]\}$;

$\pi_{23}(R_6:2)=\Z_8\{[\zeta_5]\sigma_{16}\}
\oplus\Z_2\{[\nu_5]\bar{\varepsilon}_8\}
\oplus\Z_2\{[\nu_4\sigma'\mu_{14}]_6\}
\oplus\Z_2\{[\nu_4(E\zeta')]_6\}\\ 
\oplus\Z_2\{[\nu_4\eta_7\bar{\varepsilon}_8]_6\}$;

\item$\pi_{23}(R_7:2)=\Z_2\{[\zeta_5]_7\sigma_{16}\}
\oplus\Z_2\{[\nu_4\sigma'\mu_{14}]_7\}
\oplus\Z_2\{[\nu_4(E\zeta')]_7\}
\oplus\Z_2\{[\nu_4\eta_7\bar{\varepsilon}_8]_7\}\\
\oplus\Z_2\{[\eta_6]\mu_7\sigma_{16}\}
\oplus\Z_4\{[P(E\theta)+\nu_6\kappa_9]\}$,
where $2[P(E\theta)+\nu_6\kappa_9]\equiv[\nu_5]_7\bar{\varepsilon}_8\ 
\bmod\ 
[\nu_4(E\zeta')]_7, [\nu_4\eta_7\bar{\varepsilon}_8]_7$;

$\pi_{23}(R_8:2)=\Z_2\{[\zeta_5]_8\sigma_{16}\}
\oplus\Z_2\{[\nu_4\sigma'\mu_{14}]_8\}\oplus\Z_2\{[\nu_4(E\zeta')]_8\}
\oplus\Z_2\{[\nu_4\eta_7\bar{\varepsilon}_8]_8\}\\
\oplus\Z_2\{[\eta_6]_8\mu_7\sigma_{16}\}
\oplus \Z_4\{[P(E\theta)+\nu_6\kappa_9]_8\}
\oplus\Z_2\{[\iota_7]\sigma'\mu_{14}\}\\
\oplus\Z_2\{[\iota_7](E\zeta')\}
\oplus\Z_2\{[\iota_7]\eta_7\bar{\varepsilon}_8\}
\oplus\Z_2\{[\iota_7]\mu_7\sigma_{16}\}$;

$\pi_{23}(R_9:2)=
\Z_2\{[\zeta_5]_9\sigma_{16}\}
\oplus\Z_2\{[\iota_7]_9\mu_7\sigma_{16}\}
\oplus\Z_4\{[P(E\theta)+\nu_6\kappa_9]_9\}$, where 
$2[P(E\theta)+\nu_6\kappa_9]_9=[\nu_5]_9\bar{\varepsilon}_8$\ and\ 
$\Delta(\rho')\equiv[\zeta_5]_9\sigma_{16}\ \bmod\ [\nu_5]_9\bar{\varepsilon}_8$.

\item
$\pi_{23}(R_{10}:2)=\Z_{32}\{[\sigma^2_9]\}
\oplus\Z_2\{[\bar{\nu}_9\nu^2_{17}]\}
\oplus\Z_2\{[P(E\theta)+\nu_6\kappa_9]_{10}\}$,\ where\ 
$\Delta(\sigma^2_{10}) = 2x[\sigma^2_9]$ for odd $x$, 
$8\Delta(\sigma^2_{10})=16[\sigma^2_9]=[\iota_7]_{10}\mu_7\sigma_{16}$\ 
and\ 
$\Delta\kappa_{10}
\equiv [\bar{\nu}_9\nu^2_{17}] + [P(E\theta)+\nu_6\kappa_9]_{10}
\bmod\ 16[\sigma^2_9]$;

$\pi_{23}(R_{11}:2)=\Z_2\{[\sigma^2_9]_{11}\}
\oplus\Z_2\{[P(E\theta)+\nu_6\kappa_9]_{11}\}$ 
where $[\bar{\nu}_9\nu^2_{17}]_{11}=[P(E\theta)+\nu_6\kappa_9]_{11}
=[[\iota_{10},\nu_{10}]]\eta_{22}$;

$\pi_{23}(R_{12}:2)=\Z_2\{[\theta']\}
\oplus\Z_2\{[\sigma^2_9]_{12}\}
\oplus\Z_2\{[P(E\theta)+\nu_6\kappa_9]_{12}\}$, where $\Delta(E\theta')
=[P(E\theta)+\nu_6\kappa_9]_{12}$;

$\pi_{23}(R_{13}:2)=\Z\{[32[\iota_{12}, \iota_{12}]]\}
\oplus\Z_2\{[\sigma^2_9]_{13}\}$;

$\pi_{23}(R_{14}:2)=\Z\{[32[\iota_{12}, \iota_{12}]]_{14}\}
\oplus\Z_2\{[\sigma^2_9]_{14}\}$;

\item
$\pi_{23}(R_{15}:2)=\Z\{[32[\iota_{12}, \iota_{12}]]_{15}\}
\oplus\Z_2\{[\sigma^2_9]_{15}\}\oplus\Z_2\{[\eta_{14}\varepsilon_{15}]\}
\oplus\Z_2\{[\eta_{14}\sigma_{15}]\eta_{22}\}$, where\  
$[\eta_{14}\varepsilon_{15}]\in\{[\eta^2_{14}],2\iota_{16},\nu^2_{16}\}_{11}$;

$\pi_{23}(R_{16}:2)=\Z\{[32[\iota_{12}, \iota_{12}]]_{16}\}
\oplus\Z_2\{[\sigma^2_9]_{16}\}\oplus\Z_2\{[\eta_{14}\varepsilon_{15}]_{16}\}\\ 
\oplus\Z_2\{[\eta_{14}\sigma_{15}]_{16}\eta_{22}\}
\oplus\Z_2\{[\eta_{15}]\sigma_{16}\}
\oplus\Z_2\{[\varepsilon_{15}]\}$, where $[\varepsilon_{15}]\in\{[\eta_{15}],2\iota_{16},\nu^2_{16}\}_1$,\ $\Delta(\varepsilon_{16})
= [\eta_{14}\varepsilon_{15}]_{16}$\ and\ $\Delta(\eta_{16}\sigma_{17}) = [\eta_{14}\sigma_{15}]_{16}\eta_{22}$;

$\pi_{23}(R_{17}:2)=\Z\{[32[\iota_{12}, \iota_{12}]]_{17}\}
\oplus\Z_2\{[\sigma^2_9]_{17}\}\oplus\Z_2\{[\eta_{15}]_{17}\sigma_{16}\}
\oplus\Z_2\{[\varepsilon_{15}]_{17}\}$;

$\pi_{23}(R_{18}:2)=\Z\{[32[\iota_{12}, \iota_{12}]]_{18}\}
\oplus\Z_2\{[\sigma^2_9]_{18}\}\oplus\Z_2\{[\varepsilon_{15}]_{18}\}$, where 
$\Delta(\nu^2_{18}) = [\varepsilon_{15}]_{18}$;

$\pi_{23}(R_n:2)=\Z\{[32[\iota_{12}, \iota_{12}]]_n\}
\oplus\Z_2\{[\sigma^2_9]_n\}$ for $n = 19, 20, 21$, where $\Delta\nu_{21}=[\sigma^2_9]_{21}$;

$\pi_{23}(R_{24}:2)=\Z\{\Delta\iota_{24}\}
\oplus\Z\{[32[\iota_{12},\iota_{12}]]_{24}\}$;

$\pi_{23}(R_{22+n}:2)=\Z\{[32[\iota_{12}, \iota_{12}]]_{22+n}\}$\ for 
$n=0, 1$ and $n\geq 3$.
\end{enumerate}
\end{thm}

\begin{thm} \label{thm2}
\begin{enumerate}
\item $\pi_{24}(R_3:2)=\Z_2\{[\eta_2]\nu'\eta_6\bar{\mu}_7\}$;

$\pi_{24}(R_4:2)=\Z_2\{[\eta_2]_4\nu'\eta_6\bar{\mu}_7\}
\oplus\Z_2\{[\iota_3]\nu'\eta_6\bar{\mu}_7\}$;

$\pi_{24}(R_5:2)=\Z_2\{[\nu^2_4]\kappa_{10}\}
\oplus\Z_2\{[\nu_4\sigma'\eta_{14}]\mu_{15}\}\}$, where $[\nu_4\sigma'\eta_{14}]\mu_{15}=[\nu_4\sigma'\mu_{14}]\eta_{23}$;

$\pi_{24}(R_6:2)=\Z_8\{[\bar{\zeta}_5]\}
\oplus\Z_2\{[\nu_5]\mu_8\sigma_{17}\}
\oplus\Z_2\{[\nu^2_4]_6\kappa_{10}\}
\oplus\Z_2\{[\nu_4\sigma'\eta_{14}]_6\mu_{15}\}$, where 
$\Delta(\bar{\zeta}_6)=2[\bar{\zeta}_5]$; 

\item $\pi_{24}(R_7:2)=\Z_2\{[\bar{\zeta}_5]_7\}
\oplus\Z_2\{[\nu_5]_7\mu_8\sigma_{17}\}\oplus\Z_2\{[\eta_6]\bar{\mu}_7\}
\oplus\Z_2\{[\eta_6]\eta_7\mu_8\sigma_{16}\}
\oplus\Z_2\{[P(E\theta)+\nu_6\kappa_9]\eta_{23}\}
\oplus\Z_2\{[\nu^2_4]_7\kappa_{10}\}
\oplus\Z_2\{[\nu_4\sigma'\eta_{14}]_7\mu_{15}\}$;

$\pi_{24}(R_8:2)=\Z_2\{[\bar{\zeta}_5]_8\}
\oplus\Z_2\{[\eta_6]_8\bar{\mu}_7\}
\oplus\Z_2\{[\eta_6]_8\eta_7\mu_8\sigma_{16}\}
\oplus\Z_2\{[P(E\theta)+\nu_6\kappa_9]_8\eta_{23}\}
\oplus\Z_2\{[\nu_5]_8\mu_8\sigma_{17}\}
\oplus\Z_2\{[\nu^2_4]_8\kappa_{10}\}
\oplus\Z_2\{[\nu_4\sigma'\eta_{14}]_8\mu_{15}\}\\ 
\oplus\Z_2\{[\iota_7]\sigma'\eta_{14}\mu_{15}\}
\oplus\Z_2\{[\iota_7]\nu_7\kappa_{10}\}
\oplus\Z_2\{[\iota_7]\bar{\mu}_7\}
\oplus\Z_2\{[\iota_7]\eta_7\mu_8\sigma_{17}\}$, 
where $\Delta(\nu_8\kappa_{11})=[\nu^2_4]_8\kappa_{10}$;

$\pi_{24}(R_9:2)=\Z_2\{[\bar{\zeta}_5]_9\}
\oplus\Z_2\{[P(E\theta)+\nu_6\kappa_9]_9\eta_{23}\}
\oplus\Z_2\{[\nu_5]_9\mu_8\sigma_{17}\}\\ 
\oplus\Z_2\{[\iota_7]_9\nu_7\kappa_{10}\}
\oplus\Z_2\{[\iota_7]_9\bar{\mu}_7\}
\oplus\Z_2\{[\iota_7]_9\eta_7\mu_8\sigma_{17}\}$, 
where $\Delta(\mu_9\sigma_{18})=[\nu_5]_9\mu_8\sigma_{17}
+[\iota_7]_9\eta_7\mu_8\sigma_{17}$.

\item
$\pi_{24}(R_{10}:2)=\Z_{16}\{\Delta(E\rho')\}
\oplus\Z_2\{[\sigma^2_9]\eta_{23}\}
\oplus\Z_4\{[\bar{\varepsilon}_9]\}
\oplus\Z_2\{[P(E\theta)+\nu_6\kappa_9]_{10}\eta_{23}\}\ 
\oplus\Z_2\{[\iota_7]_{10}\bar{\mu}_7\}$, where $8\Delta(E\rho')=[\bar{\zeta}_5]_{10}+[\iota_7]_{10}\bar{\mu}_7$,
$[P(E\theta)+\nu_6\kappa_9]_{10}\eta_{23}
= [\bar{\nu}_9\nu^2_{17}]\eta_{23}$
 \ and \ 
$2[\bar{\varepsilon}_9]=[\iota_7]_{10}\nu_7\kappa_{10}$;

$\pi_{24}(R_{11}:2)=
\Z_2\{[\sigma^2_9]_{11}\eta_{23}\}
\oplus\Z_4\{[\bar{\varepsilon}_9]_{11}\}
\oplus\Z_2\{[P(E\theta)+\nu_6\kappa_9]_{11}\eta_{23}\}\oplus\Z_2\{[\iota_7]_{11}\bar{\mu}_7\}$;

$\pi_{24}(R_{12}:2)=\Z_2\{[\theta']\eta_{23}\}
\oplus\Z_2\{[\eta_{11}]\theta\}
\oplus\Z_2\{[\sigma^2_9]_{12}\eta_{23}\}
\oplus\Z_2\{[\bar{\varepsilon}_9]_{12}\}\\
\oplus\Z_2\{[P(E\theta)+\nu_6\kappa_9]_{12}\eta_{23}\}
\oplus\Z_2\{[\iota_7]_{12}\bar{\mu}_7\}$,\ where $\Delta((E\theta')\eta_{24})=[P(E\theta)+\nu_6\kappa_9]_{12}\eta_{23}$;

$\pi_{24}(R_{13}:2)=\Z_2\{[\eta_{11}]_{13}\theta\}
\oplus\Z_2\{[\sigma^2_9]_{13}\eta_{23}\}
\oplus\Z_2\{[\bar{\varepsilon}_9]_{1 3}\}
\oplus\Z_2\{[\iota_7]_{13}\bar{\mu}_7\}$, 
where $\Delta(E\theta)=[\eta_{11}]_{13}\theta$;

$\pi_{24}(R_{14}:2)=\Z_8\{[\zeta_{13}]\}
\oplus\Z_2\{[\sigma^2_9]_{14}\eta_{23}\}
\oplus\Z_2\{[\bar{\varepsilon}_9]_{14}\}
\oplus\Z_2\{[\iota_7]_{14}\bar{\mu}_7\}$;

\item
$\pi_{24}(R_{15}:2)=
\Z_2\{[\eta_{14}\mu_{15}]\}
\oplus\Z_2\{[\zeta_{13}]_{15}\}
\oplus\Z_2\{[\sigma^2_9]_{15}\eta_{23}\}
\oplus\Z_2\{[\bar{\varepsilon}_9]_{15}\}
\oplus\Z_2\{[\iota_7]_{15}\bar{\mu}_7\}$;

$\pi_{24}(R_{16}:2)=
\Z_2\{[\mu_{15}]\}
\oplus\Z_2\{[\eta_{15}]\varepsilon_{26}\}
\oplus\Z_2\{[\eta_{15}]\bar{\nu}_{16}\}
\oplus\Z_2\{[\eta_{14}\mu_{15}]_{16}\}
\oplus\Z_2\{[\zeta_{13}]_{16}\}
\oplus\Z_2\{[\sigma^2_9]_{16}\eta_{23}\}
\oplus\Z_2\{[\bar{\varepsilon}_9]_{16}\}
\oplus\Z_2\{[\iota_7]_{16}\bar{\mu}_7\}$, 
where $\Delta(\eta_{16}\varepsilon_{17})=[\zeta_{13}]_{16}$,\ 
$\Delta(\nu^3_{16})\equiv[\bar{\varepsilon}_9]_{16}\ \bmod\ [\iota_7]_{16}\bar{\mu}_7$ and $\Delta\mu_{16}\equiv[\eta_{14}\mu_{15}]_{16}\ \bmod\ 
\Delta(\eta_{16}\varepsilon_{17}), \Delta(\nu^3_{16})$;

$\pi_{24}(R_{17}:2)=
\Z_2\{[\mu_{15}]_{17}\}
\oplus\Z_2\{[\eta_{15}]_{17}\varepsilon_{16}\}
\oplus\Z_2\{[\eta_{15}]_{17}\bar{\nu}_{17}\}
\oplus\Z_2\{[\sigma^2_9]_{17}\eta_{23}\}
\oplus\Z_2\{[\iota_7]_{17}\bar{\mu}_7\}$, where 
$\Delta(\varepsilon_{17}) = [\eta_{15}]_{17}\varepsilon_{16}$ and 
$\Delta(\bar{\nu}_{17}) = [\eta_{15}]_{17}\bar{\nu}_{16}$;

$\pi_{24}(R_{18}:2)=\Z_{16}\{\Delta\sigma_{18}\}
\oplus\Z_2\{[\sigma^2_9]_{18}\eta_{23}\}
\oplus\Z_2\{[\iota_7]_{18}\bar{\mu}_7\}$, where $8\Delta\sigma_{18}=[\mu_{15}]_{18}$;

$\pi_{24}(R_{19}:2)=
\Z_2\{[\sigma^2_9]_{19}\eta_{23}\}
\oplus\Z_2\{[\iota_7]_{19}\bar{\mu}_7\}$, where 
$\Delta(\nu^2_{19})=[\sigma^2_9]_{19}\eta_{23}$;

$\pi_{24}(R_{22}:2)=\Z_4\{\Delta\nu_{22}\}
\oplus\Z_2\{[\iota_7]_{22}\bar{\mu}_7\}$;

$\pi_{24}(R_{23}:2)=\Z_2\{[\eta^2_{22}]\}
\oplus\Z_2\{[\iota_7]_{23}\bar{\mu}_7\}$;

$\pi_{24}(R_{24}:2)=\Z_2\{[\eta_{23}]\}\oplus
\Z_2\{[\eta^2_{22}]_{24}\}
\oplus\Z_2\{[\iota_7]_{24}\bar{\mu}_7\}$, where
$\Delta\eta_{24}=[\eta^2_{22}]_{24}$;

$\pi_{24}(R_{25}:2)=\Z_2\{[\eta_{23}]_{25}\}
\oplus\Z_2\{[\iota_7]_{25}\bar{\mu}_7\}$, where $\Delta\iota_{25}=[\eta_{23}]_{25}$;

$\pi_{24}(R_{20+n}:2)=\Z_2\{[\iota_7]_{20+n}\bar{\mu}_7\}$ for $n=0, 1$ and 
$n\geq 6$.
\end{enumerate}
\end{thm}
 
\section{Recollection of fundamental facts}

First of all we recall a formula \cite[(2.2)]{MT2}
\begin{equation} \label{DeSig}
\Delta(\alpha\circ E\beta)=\Delta\alpha\circ\beta.
\end{equation}

Suppose given elements $\alpha\in\pi_k(S^n)$ and $\beta\in\pi_m(S^k)$ such 
that $\Delta\alpha=0$. Since $\Delta(\alpha\beta)= 
\Delta({p_{n+1}}_*([\alpha]\beta))=0$, 
we obtain
\begin{equation} \label{Delta}
\Delta(\alpha\beta)=0\ (\Delta\alpha=0).
\end{equation}

By \cite[Theorem 4]{Wh42}, we have the relation
\begin{equation}\label{pD} 
{p_n}_*\Delta\iota_n=(1+(-1)^n)\iota_{n-1}. 
\end{equation}
So, for an
element $\alpha\in\pi_k(S^n)$, we obtain  a formula
\begin{equation} \label{DelS2}
{p_{n+1}}_*\Delta(E\alpha)=0\ (n: \mbox{even}); \ 
{p_{n+1}}_*\Delta(E\alpha)=2\iota_n\circ\alpha\ (n: \mbox{odd}).
\end{equation}

Let $J: \pi_k(R_n)\to \pi_{k+n}(S^n)$ be the $J$\ homomorphism \cite{Wh}. 
For $\alpha\in\pi_{m}(R_{n})$ and $\beta\in\pi_{k}(S^{m})$, we have a formula
\[
J(\alpha\beta )=J(\alpha )E^n\beta.
\]

We use the following formulas without referring to them explicitly.
\[\begin{split}
EJ(\alpha)&=-J({i}_{n+1}\circ\alpha)
\text{\ for\ }  \alpha\in\pi_k(R_{n})
\text{\ \cite[(2.1)]{JHCWh},}
\\
J\Delta(\gamma)&=[\gamma,\iota_{n}]
\text{\ for\ } \gamma\in\pi_{k}(S^n)
\text{\ \cite[(3.6)]{JHCWh},}
\\
HJ(\beta)&=E^n(p_n\circ\beta) 
\text{\ for\ } \beta\in\pi_{k}(R_n)
\text{\ \cite[Corollary 15.9]{J}.}
\end{split}\]
We have the relation $3[\iota_n,[\iota_n,\iota_n]]=0$ from the Jacobi identity for Whitehead products. Hence, for $\gamma\in\pi_{k}^{n}$,
by the formula \cite[Corollary (7.4)]{BB}, we have the relation
\begin{equation}\label{JD}
J\Delta(\gamma)=(-1)^{kn}[\iota_{n},\gamma]
=(-1)^{kn}[\iota_n,\iota_n]\circ E^{n-1}\gamma \in\pi_{k+n-1}^{n}.
\end{equation}
Concerning the exact sequence \ref{seqk_n}, we obtain the commutative diagram up to
sign:
\begin{equation} \label{EHP}
\begin{CD}
R^{n}_k @>{i_*}>> R^{n+1}_k @>{p_*}>> \pi^{n}_k \ @>{\Delta}>> R^{n}_{k-1}\\
@V{J}VV @V{J}VV @V{E^{n+1}}VV \ @V{J}VV\\
\pi^{n}_{k+n} @>{E}>> \pi^{n+1}_{k+n+1} @>{H}>> \pi^{2n+1}_{k+n+1},
@>{P}>> \pi^{n}_{k+n-1},
\end{CD}
\end{equation}
where the lower sequence is the EHP sequence. This diagram is also a direct 
consequence of \cite[(11.4)]{T}. 

The following formula about Toda brackets is useful \cite{MO}:
\begin{equation} \label{mt1}
\Delta\{\alpha, E^n\beta, E^n\gamma\}_n\subset(-1)^n\{\Delta(\alpha),E^{n-1}\beta,E^{n-1}\gamma\}_{n-1}.
\end{equation}

By \cite[Lemma 5.1]{MiM}, we know ($\alpha\in\pi_{n+k}(S^n)$):
\begin{equation}\label{JTbr}
J\{\alpha,\beta,\gamma\}\subset\{J\alpha,E^n\beta,E^n\gamma\}_n. 
\end{equation}

We need the following:
\begin{gather}
\label{t1}
\eta_3\nu_4=\nu'\eta_6,\ \eta_5\nu_6=0,\ P(\iota_{11})=\nu_5\eta_8,\ \nu_6\eta_9=0 
\quad\text{\cite[(5.9), (5.10)]{T}};\\
%\label{t2}
%[\iota_5,\iota_5]=\nu_5\eta_8 
%\quad\text{\cite[(5.10)]{T}};\\
%\label{2s''}
%2\sigma''=E\sigma''',\ 2\sigma'=E\sigma''
%\quad\text{\cite[Lemma 5.14]{T}};\\
%\label{2s8}
%P(\iota_{17})=\pm(2\sigma_8-E\sigma')
%\quad\text{\cite[(5.16)]{T}};\\
%\label{t4}
%\eta_6\sigma' =4\bar{\nu_6},\ 
%\eta_7\sigma_8=\sigma'\eta_{14}+\bar{\nu}_7+\varepsilon_7
%\quad\text{\cite[(7.4)]{T}};\\
\label{et9s}
\eta_9\sigma_{10}=\bar{\nu}_9+\varepsilon_9,\ 
\eta_9\sigma_{10}+\sigma_9\eta_{16}=P(\iota_{19})
\quad\text{\cite[Lemma 6.4, (7.1)]{T}};\\
%\label{t3}
%\eta_5\bar{\nu}_6=\nu^3_5,\ \bar{\nu}_6\eta_{14}=\nu^3_6
%\quad\text{\cite[Lemma 6.3, (7.3)]{T}};\\
%\label{t5}
%\varepsilon_3\eta_{11}=\eta_3\varepsilon_4
%\quad\text{\cite[(7.5)]{T}};\\
\label{n7s}
\sigma'\nu_{14}
=x\nu_7\sigma_{10}\ (x:\mbox{odd})
\quad\text{\cite[(7.19)]{T}};\\
%\label{2s10n}
%2\sigma_{10}\nu_{17}=P(\eta_{21}),\ 
%\sigma_{11}\nu_{18}=P(\iota_{23})
%\quad\text{\cite[(7.21)]{T}};\\
\label{n11s}
\nu_{11}\sigma_{14}=0,\ \sigma_{12}\nu_{19}=0
\quad\text{\cite[(7.20)]{T}};\\
%\label{etm}
%\mu_3\eta_{12}=\eta_3\mu_4
%\quad\text{\cite[Proposition 2.2(2)]{Og}};\\
\label{t6}
\varepsilon_3\nu_{11}=\nu'\bar{\nu}_6,\ 
\varepsilon_4\nu_{12}=P(\bar\nu_9)
\quad\text{\cite[(7.12), (7.13)]{T}};\\
\label{n6e}
\nu_6\bar\nu_9=\nu_6\varepsilon_9=P(\nu^2_{13})
\quad\text{\cite[(7.17), (7.18)]{T}};\\
\label{bn9n}
\bar{\nu}_9\nu_{17}=P(\nu_{19})
\quad\text{\cite[(7.22)]{T}}.
%\label{kaet}
%\bar{\varepsilon}_6=\eta_6\kappa_7
%\quad\text{\cite[(10.23)]{T}};\\
%\label{epsig1}
%\varepsilon_3\sigma_{11}=0,\ \bar{\nu}_6\sigma_{14}=0.
 %\sigma_{11}\bar{\nu}_{18}=0 \quad  \text{and} \\
%\label{epsig2}
% \sigma_{10}\bar\nu_{17}=\sigma_{10}\varepsilon_{17}=P(\nu^2_{21}) 
%$\quad\text{\cite[Lemma 10.7, (10.18)]{T}};\\
%\label{s'7z}
%\sigma'\zeta_{14}=x\zeta_7\sigma_{18}\ (x: \mbox{odd})
%\quad\text{\cite[Lemma 12.12]{T}}.
%\label{z11s}
%\zeta_{11}\sigma_{22}=0
%\quad\text{\cite[I-Proposition 3.5(7)]{Od1}}.
%;\\
%\label{sigsigmu}
%\sigma^2_{10}\mu_{24} = 0
%\quad\text{\cite[(2.3)]{MMO}}.
\end{gather}

Next we show relations in the homotopy groups of spheres which are necessary for our calculations.

By \cite[Lemma 5.14]{T}, we know
\begin{equation}\label{2s''}
2\sigma''=E\sigma''',\ 2\sigma'=E\sigma''.
\end{equation}

By \cite[(7.10), (7.20)]{T}, we obtain
\begin{equation}\label{et^2e}
\eta^2_9\varepsilon_{11}=4(\nu_9\sigma_{12})=0.
\end{equation}

By \cite[(10.23)]{T}, we know 
\begin{equation}\label{kaet}
\bar{\varepsilon}_6=\eta_6\kappa_7.
\end{equation}

By $\sigma'\nu_{14}=x\nu_7\sigma_{10}$ for odd $x$ \eqref{n7s} and
\eqref{t6}, we have
\begin{equation}\label{n7sigsig}
\nu_7\sigma^2_{10}=0.
\end{equation}

By \cite[Proposition 2.13(7)]{Og} and \cite[Lemma 12.3]{T}% and \eqref{2n4}
, we have
\begin{equation}\label{Ogu0}
\mu_5\varepsilon_{14}=\varepsilon_5\mu_{13}=\eta_5\mu_6\sigma_{15}.
\end{equation}

We recall the relations from (\cite[Table 3]{Ka}, \cite[Lemma 1.1(i)]{KM1}):
\begin{gather}
\label{k1}
\Delta\iota_4=2[\iota_3]-[\eta_2]_4,\\
\label{k2}
\Delta\iota_5=[\iota_3]_5\eta_3,\\
\label{k3}
\Delta\iota_8=2[\iota_7]-[\eta_6]_8,\\
\label{Dn5}
\Delta\nu_5=0,\\
\label{k4}
\Delta\iota_9=[\nu_5]_9+[\iota_7]_9\eta_7,\\
\label{k4.5}
\Delta\iota_{11}=[\iota_7]_{11}\nu_7,\\
\label{k5}
\Delta\nu_6=2[\nu_5],\\
%\intertext{and}
\label{k6}
\Delta\nu_4\equiv [\iota_3]\nu'\ \bmod\ [\eta_2]_4\nu'.
\end{gather}

We take 
\begin{gather}
\label{Ji3}
J[\iota_3]=\nu_4,\\
\label{Jet2}
J[\eta_2]=\nu'
\end{gather}

We also recall the relations from \cite[Lemma 1.1(ii)]{KM1}:
\begin{gather}
\label{Ji7}
J[\iota_7]=\sigma_8, \\
\label{Jn5}
J[\nu_5]=\bar\nu_6+\varepsilon_6,\\
\label{Jet5ep}
J[\eta_5\varepsilon_6]=-\sigma''\sigma_{13}+\bar{\nu}_6\nu^2_{14},\\
\label{Jeta6}
J[\eta_6]=\sigma', \\ 
\label{Jbn6e6}
J[\bar\nu_6+\varepsilon_6]=\sigma'\sigma_{14},\\
\label{let6sg}
[\eta_6]\sigma'=4[\bar\nu_6+\varepsilon_6]+[\nu_5]_7\nu^2_8+[\eta_5\varepsilon_6]_7.
\end{gather}

For a group $G$, let $\alpha\in G$. Denote by $\sharp\alpha$ the order of $\alpha$. 
By \cite[(3.2)]{GM} and the fact that  
$2\pi_{4n+4}(R_{4n+4})=0$ \cite{Ke}, we have
\begin{equation}\label{Ke1}
\Delta(\eta_{4n+3})=0;\ 
\sharp[\eta_{4n+3}]=2 \ (n\ge 2).
\end{equation}

By \cite[(4.2)]{GM}, \eqref{DeSig} and the fact that $2R^{4n+3}_{4n+4}=0$ \cite{Ke}, we have 
\begin{equation}\label{De^2}
\Delta(\eta^2_{4n+2})=0;\ \sharp[\eta^2_{4n+2}]=2.
\end{equation}

By \eqref{Ke1} and \cite[(4.1), (4.4)]{GM}, we obtain 
\begin{equation} \label{Ke2}
\Delta\iota_{4n+1} = [\eta_{4n-1}]_{4n+1}\ \mbox{for}\ n\geq 2.
\end{equation}

By \cite[p. 161, Table]{Ke}, we know $R^{4n+2}_{4n+2}\cong\Z_4$ for $n\ge 2$.
So, we have 
\begin{equation}\label{De4+2}
\Delta(\eta_{4n+2})\in 2R^{4n+2}_{4n+2}\ \mbox{for}\ n\ge 2.
\end{equation}
In particular, by \cite[p.~96]{KM2}, we have
\begin{equation}\label{Det10}
\Delta(\eta_{10}) = 2[\iota_7]_{10}\nu_7.
\end{equation}

By \cite[p. 410, \mbox{the proof of} Lemma 3.5(1)]{GM}, we have 
\begin{equation}\label{nu8n-1}
\Delta(\nu_{8n-1})=0;\ \sharp[\nu_{8n-1}]=8.
\end{equation}

By \cite[(4.1)-(4.5)]{GM}, we have 
\begin{equation}\label{gm1}
\Delta(\eta^2_{8n+1})=4[\nu_{8n-1}]_{8n+1},\
\Delta(\eta_{8n+2})=2[\nu_{8n-1}]_{8n+2},
\end{equation}
\begin{equation}\label{gm2}
\Delta(\iota_{8n+3})=[\nu_{8n-1}]_{8n+3}.
\end{equation}

We recall from \cite[(1.8)]{HKM},
\begin{equation}\label{Ke21}
\Delta\eta_{4n+4}=[\eta^2_{4n+2}]_{4n+4}\ \mbox{for}\ n\geq 2.
\end{equation}

By \cite[Lemma 5.2]{MiM}, we have 
\begin{equation}\label{M-m}
\{p_{n+1}, i_{n+1}, \Delta\iota_n\}\ni \iota_n \bmod 2\iota_n.
\end{equation}
 
We need the following.
\begin{lem}\label{[Ea]}
Suppose that $\alpha\in\pi^{n-2}_{n-2+k}$ satisfies $\Delta(E\alpha)=0$. Then,
$[E\alpha]\in -\{i_n,\Delta\iota_{n-1},\alpha\}\ \bmod\ {i_n}_*R^{n-1}_{n-1+k}+R^n_{n-1}\circ E\alpha$. In particular,\\ 
$\Delta\iota_{2n}\in -\{i_{2n},\Delta\iota_{2n-1},2\iota_{2n-2}\}
\ \bmod\ {i_{2n}}_*R^{2n-1}_{2n} + R^{2n}_{2n-1}\circ 2\iota_{2n-1}$.
\end{lem}
\begin{proof} 
By \eqref{M-m}, we obtain
\[
p_n\circ(-\{i_n,\Delta\iota_{n-1},\alpha\})
=\{p_n,i_n,\Delta\iota_{n-1}\}\circ E\alpha\ni E\alpha \bmod 
p_n\circ R^{n}_{n-1}\circ E\alpha.
\]
This and ${p_n}_*[E\alpha]=E\alpha$ imply the relation
\[
[E\alpha]\in -\{i_n,\Delta\iota_{n-1},\alpha\}\ \bmod\ {i_n}_*R^n_{n-1+k}+R^n_{n-1}\circ E\alpha
\]
by using the exact sequence $(n-1+k)_{n-1}$, and we obtain the assertion.

By considering $[2\iota_{2n-1}]=\Delta\iota_{2n}$, we have the rest.
This completes the proof.
\end{proof}

Suppose that $n\geq 2$. By \cite[Theorem 7.4, Lemma 9.1]{T}, 
we know $\zeta_{4n+1}\in\{2\iota_{4n+1},\eta^2_{4n+1},\varepsilon_{4n+3}\}\ \bmod\ \bar{\nu}_{4n+1}\nu_{4n+9}$. 
By relations 
$\Delta\iota_{4n+2}\circ\eta^2_{4n+1}
=\Delta(\eta^2_{4n+2})=0$ \eqref{De^2} and 
$\eta^2_{4n+1}\circ\varepsilon_{4n+3} =0$ 
\eqref{et^2e}, the Toda bracket\\  
$\{\Delta\iota_{4n+2},\eta^2_{4n+1},\varepsilon_{4n+3}\}$ is defined. 
By \cite[p. 96]{KM2}, we have $\Delta(\zeta_9) = 0$ and 
$\Delta(\bar{\nu}_9\nu_{17})=[\nu_5]_9\bar{\nu}_8\nu_{16}\neq 0$.  
We define 
$$
[\zeta_{4n+1}]\in\{\Delta\iota_{4n+2}, \eta^2_{4n+1}, \varepsilon_{4n+3}\}.
$$

By \cite[(7.14)]{T}, we have
\begin{equation}\label{4z5}
4\zeta_5=\eta^2_5\mu_7.
\end{equation}

We show the following.
\begin{lem} \label{zeta1}
\begin{enumerate}
\item 
$\Delta\zeta_{4n+2} = \pm 2[\zeta_{4n+1}]
\quad\text{and}\quad \sharp[\zeta_{4n+1}]=8$\ for\ $n\ge 2$.
\item 
$[\zeta_{4n+1}]_{4n+4}\equiv\Delta(\eta_{4n+4}\varepsilon_{4n+5})\
 \bmod\ {i_{4n+2,4n+4}}_*R^{4n+2}_{4n+4}\circ\varepsilon_{4n+4}$ \
 for\ $n\ge 2$. In particular, $\Delta(\eta_{8n+4}\varepsilon_{8n+5})
 =[\zeta_{8n+1}]_{8n+4}$\ for\ $n\ge 1$.
\end{enumerate}
\end{lem}
\begin{proof}
We have
\[\begin{split}
- 2[\zeta_{4n+1}]
&\in\{\Delta\iota_{4n+2}, \eta^2_{4n+1}, \varepsilon_{4n+3}\}\circ (-2\iota_{4n+12})\\
&=\Delta\iota_{4n+2}\circ\{\eta^2_{4n+1},\varepsilon_{4n+3}, 2\iota_{4n+11}\}\\
&\ni\Delta\iota_{4n+2}\circ(\zeta_{4n+1})
 = \Delta\zeta_{4n+2}.
\end{split}\]
Since the indeterminacy of this relation is
$\Delta\iota_{4n+2}\circ\pi_{4n+12}^{4n+1}\circ 2\iota_{4n+12}=
\{\Delta\iota_{4n+2}\circ 2\zeta_{4n+1}\}=\{ 2\Delta\zeta_{4n+2}\}$,
we have
$2[\zeta_{4n+1}]\equiv -\Delta\zeta_{4n+2} \bmod 2\Delta\zeta_{4n+2}$.
By relations  
\eqref{4z5} and \eqref{De^2},
we have
$4\Delta\zeta_{4n+2} = \Delta(\eta^2_{4n+2}\mu_{4n+4}) = 0$.
Hence we obtain the first of (1).
We also have
$
8[\zeta_{4n+1}] = 4\Delta\zeta_{4n+2} = 0.
$
On the other hand, by the fact that $\zeta_{4n+1}$ has order 8,
$[\zeta_{4n+1}]$ has order at least 8. 
This leads to the second of (1).

By the definition of $[\zeta_{4n+1}]$ and Lemma \ref{[Ea]}, we have 
$$
[\zeta_{4n+1}]_{4n+3}\in\{i_{4n+3},\Delta\iota_{4n+2},\eta^2_{4n+1}\}\circ\varepsilon_{4n+4}
$$
$$
\ni
[\eta^2_{4n+2}]\varepsilon_{4n+4}\ 
\bmod \ {i_{4n+3}}_*R^{4n+2}_{4n+4}\circ\varepsilon_{4n+4}.
$$
This and \eqref{Ke21} imply the first of (2).

By \cite[p. 161, Table; Theorem 3]{Ke}, $R^{4n+2}_{4n+4}\cong
\Z_{12}$\ for\ $n$ even and $\cong\Z_{12}\oplus\Z_2$\ for\ $n$ odd
and the direct summand $Z_{12}$ is generated by $\Delta(\nu_{4n+2})$, 
where $\nu_{4n+2}$ stands for the generator of $\pi_{4n+5}(S^{4n+2})\cong\Z_{24}$. 
Since $\nu_{4n+2}\circ\varepsilon_{4n+5}=0$ \eqref{n6e}, we have the second of (2). This completes the proof.
\end{proof}

The assertion for $n=2$ is obtained by \cite[Theorem 3.2, Lemma 3.9(1), the last page]{KM2}.

We define $[\zeta_5]\in\{[\nu_5],8\iota_8,E\sigma'\}_1$ \cite[(2.3)]{HKM}. Then, by the fact that $\zeta_6\in\{\nu_6,8\iota_9,2\sigma_9\}_2$, \eqref{mt1} and \eqref{k5}, we have 
$$
\Delta\zeta_6\in\{2[\nu_5],8\iota_8,E\sigma'\}_1
\supset 2\iota_{R_6}\circ\{[\nu_5],8\iota_8,E\sigma'\}_1  
=\{[\nu_5],8\iota_8,E\sigma'\}_1\circ 2\iota_{16}
$$
$$
\ni 2[\zeta_5]\ \bmod\ 2[\nu_5]\circ E\pi^7_{15} + R^6_9\circ 2\sigma_9.
$$
We know $\pi^7_{15}\cong(Z_2)^3$ \cite[Theorem 7.1]{T} and $R^6_9\cong\Z_2$ \cite[p. 162, Table]{Ke}. So, the above indeterminacy is trivial 
and $\Delta\zeta_6=2[\zeta_5]$. 

\begin{conj0}
$\Delta(\eta_{4n+4}\varepsilon_{4n+5}) = [\zeta_{4n+1}]_{4n+4}$\ 
for\ $n\ge 3$.
\end{conj0}
Recall from \cite[Lemma 6.5]{T} the relation 
\begin{equation}\label{mu4}
\mu_n\in\{\eta_n,2\iota_{n+1},E^{n-4}\sigma'''\}_{n-4}+\{\nu^3_n\}
\text{\ \ for\ } n\ge 4.
\end{equation}

By \cite[Lemma 12.10]{T}, we have
\begin{equation}\label{epep}
\varepsilon^2_3=\varepsilon_3\bar{\nu}_{11}=\eta_3\bar{\varepsilon}_4
\quad\text{and}\quad
%\bar{\varepsilon}_3\eta_{18}
\nu_5\sigma_{8}\nu^2_{15}=\eta_5\bar\varepsilon_6.
\end{equation}

By \cite[Lemma 12.12]{T}, we have 
\begin{equation}\label{s'7z}
\sigma'\zeta_{14}=x\zeta_7\sigma_{18}\ (x: \mbox{odd}).
\end{equation}

We show
\begin{lem} \label{zesig}
\begin{enumerate}
\item $\bar{\nu}^2_6=\bar{\nu}_6\varepsilon_{14}=0$.
\item $(\bar{\nu}_6+\varepsilon_6)\bar{\nu}_{14} =
(\bar{\nu}_6+\varepsilon_6)\varepsilon_{14} = \eta_6\bar{\varepsilon}_7$.
\item $\sigma'\eta_{14}\bar{\nu}_ {15} = \sigma'\nu^3_{14}
= \nu_7\sigma_{10}\nu^2_{17} = \eta_7\bar{\varepsilon}_8$
\quad\text{and}\quad $\eta_9\bar{\varepsilon}_{10}=0$.
\item $\nu_5\sigma_8\bar{\nu}_{15} = \nu_5\sigma_8\varepsilon_{15}
= \nu_5\bar{\varepsilon}_8$.
\item $\zeta_5\sigma_{16}\in\{\nu_5\sigma_8, 8\iota_{15}, 2\sigma_{15}\}_1
\ \bmod \nu_5\bar{\varepsilon}_8,\ 2\zeta_5\sigma_{16}$.
\end{enumerate}
\end{lem}
\begin{proof}
(1) is just \cite[Proposition 2.8(2)]{Og}.
By relations (1) and \eqref{epep}, we have (2).
By \eqref{t3}, \eqref{n7s}, %\cite[(7.19)]{T}, 
\eqref{epep} and \eqref{n11s}, we have (3).
The assertion (4) is obtained by \cite[Theorem 2]{Od2}.

By \cite[Theorem 12.8]{T}, we know $\pi^5_{23}=\{\zeta_5\sigma_{16}, \nu_5\bar{\varepsilon}_8, \eta_5\bar{\mu}_6\}\cong\Z_8\oplus(\Z_2)^2$ 
and $E^2: \pi^5_{23}\to\pi^7_{25}$ is a split epimorphism and $\Ker E^2=\{\nu_5\bar{\varepsilon}_8\}$.

By the relations $\sigma'\nu_{14}=x\nu_7\sigma_{10}$ ($x$: odd) 
\eqref{n7s} and $\sigma'\zeta_{14}=y\zeta_7\sigma_{18}$  ($y$: odd) 
\eqref{s'7z}, we have 
\[\begin{split}
E^2\{\nu_5\sigma_8, 8\iota_{15}, 2\sigma_{15}\}_1
&\subset
\{\nu_7\sigma_{10},8\iota_{17}, 2\sigma_{17}\}_1
=\{x\sigma'\nu_{14},8\iota_{17}, 2\sigma_{17}\}_1\\
&\supset
x\sigma'\circ\{\nu_{14},8\iota_{17}, 2\sigma_{17}\}_1
\ni
x\sigma'\zeta_{14}=xy\zeta_7\sigma_{18}\\
&\qquad
\mod\ 
x\sigma'\nu_{14}\circ E\pi^{16}_{24} + \pi^7_{18}\circ 2\sigma_{17}
=\{2\zeta_7\sigma_{18}\},
\end{split}\]
where  $x$ and $y$ are odd.
This completes the proof.
\end{proof}

Next we show
\begin{lem}\label{m3sg}
\begin{enumerate}
\item
$\mu_3\sigma_{12}\in\{\varepsilon_3,\sigma_{11},16\iota_{18}\}_1
=\{\varepsilon_3,2\sigma_{11},8\iota_{18}\}_1\\ 
=\{\varepsilon_3,2\iota_{11},8\sigma_{11}\}_1
\ \bmod\ \eta_3\bar{\varepsilon}_4$.
\item
$\zeta'\in\{\bar{\nu}_6,8\iota_{14},2\sigma_{14}\}_1 \bmod  \eta_6\bar{\varepsilon}_7$\ and\\ 
$\zeta'\in\{\bar{\nu}_6,\sigma_{14},16\iota_{21}\}_1
 =\{\bar{\nu}_6,2\sigma_{14},8\iota_{21}\}_1
\bmod 2\zeta', \eta_6\bar{\varepsilon}_7$. 
\item
$\zeta'+\mu_6\sigma_{15}\in\{\bar{\nu}_6+\varepsilon_6,\sigma_{14},16\iota_{21}\}_1
=\{\bar{\nu}_6+\varepsilon_6,2\sigma_{14},8\iota_{21}\}_1\
 \bmod\ 2\zeta', \eta_6\bar{\varepsilon}_7$,
$E\zeta'+\mu_7\sigma_{16}\in\{\bar{\nu}_7+\varepsilon_7,2\iota_{15},8\sigma_{15}\}_1\ \bmod\ \eta_7\bar{\varepsilon}_8$.
\end{enumerate}
\end{lem}
\begin{proof}
By \cite[I-Proposition 3.2(1), its proof]{Od1}, we have
\[
\mu_3\sigma_{12}\in\{\varepsilon_3,2\iota_{11},8\sigma_{11}\}_1 \bmod \eta_3\bar{\varepsilon}_4.
\]
We have
\[
\{\varepsilon_3,\sigma_{11},16\iota_{18}\}_1
\subset \{\varepsilon_3,2\sigma_{11},8\iota_{18}\}_1
\supset\{\varepsilon_3,2\iota_{11},8\sigma_{11}\}_1
\bmod\ \varepsilon_3\circ E\pi^{10}_{18} + 8\pi^3_{19}.
\]
We know $E\pi^{10}_{18}=\{\bar\nu_{11},\varepsilon_{11}\}$ and $8\pi^3_{19}=0$ \cite[Theorems 7.1 and 12.6]{T}. So, by  \eqref{epep}, we have (1). 

The first of (2) is just \cite[Lemma 5.4(iv)]{Mi3}. 
We have 
$$
\{\bar{\nu}_6,8\iota_{14},2\sigma_{14}\}_1
\subset
\{\bar{\nu}_6,8\sigma_{14},2\iota_{21}\}_1
$$
$$
\supset
\{\bar{\nu}_6,\sigma_{14},16\iota_{21}\}_1
\bmod\ \bar\nu_6\circ E\pi^{13}_{21} + 2\pi^6_{22}.
$$
We know $\bar{\nu}_6\circ E\pi^{13}_{21} = \{\bar\nu^2_6,\bar\nu_6\varepsilon_{14}\}=0$ (\cite[Theorem 7.1]{T}, Lemma \ref{zesig}(1)) and $2\pi^6_{22} = \{2\zeta'\}$ \cite[Theorem 12.6]{T}.
This leads to the second of (2). 

The first of (3) are obtained directly from (1) and (2). 
We have 
\[\begin{split}
E\{\bar{\nu}_6,2\sigma_{14},8\iota_{21}\}_1
&\subset\{\bar{\nu}_7,2\sigma_{15},8\iota_{22}\}_1
\supset\{\bar{\nu}_7,2i\iota_{15},8\sigma_{15}\}_1\\
&\qquad\qquad
\ \bmod\ 
\bar{\nu}_7\circ E\pi^{14}_{22} + 8\pi^7_{23}.
\end{split}\]
We know $\bar{\nu}_7\circ E\pi^{14}_{22} = 0$ (\cite[Theorem 7.1]{T}, Lemma \ref{zesig}(1)) and $8\pi^7_{23} = 0$ \cite[Theorem 12.6]{T}.
This leads the second of (3) and completes the proof.
\end{proof}

By \cite[(7.4)]{T}, we have 
\begin{equation}\label{t4}
\eta_7\sigma_8=\sigma'\eta_{14}+\bar{\nu}_7+\varepsilon_7.
\end{equation}

We recall \cite[(4.1)]{HKM}: 
\begin{equation} \label{Desig9}
\Delta\sigma_9 = [\iota_7]_9(\bar{\nu}_7+\varepsilon_7).
\end{equation}
By this and \cite[Lemma 1.2(v)]{KM1}, we have 
\begin{equation}\label{[nu5]sg8}
[\nu_5]_9\sigma_8 = [\iota_7]_9\sigma'\eta_{14}.
\end{equation}

We know $\Delta(\nu^2_{10})=[\iota_7]_{10}\bar{\nu}_{7}+ x[\nu_5]_{10}\sigma_8$ for some
$x\in\{0,1\}$ \cite[p.~22]{Ka}.
Since $[\nu_5]_{10}=[\iota_7]_{10}\eta_7$ from \eqref{k4}, we have
$\Delta(\nu^2_{10})=[\iota_7]_{10}(\bar{\nu}_{7}+ x\eta_7\sigma_8)$.
By composing this equation with $\eta_{15}$ on the right and using the relation
$\bar{\nu}_7\eta_{15}=\nu^3_7$ \eqref{t3}, we get that
$0=[\iota_7]_{10}(\nu^3_{7}+ x\eta_7\sigma_8\eta_{15})$. We know  
$[\iota_7]_{10}\nu^3_7 \neq 0$ \cite[Proposition 4.1]{Ka}. 
This implies that $x=1$. Hence, by \eqref{t4}, we obtain
\begin{equation}\label{Dnu^2}
\Delta(\nu^2_{10})=[\iota_7]_{10}(\bar{\nu}_{7}+ \eta_7\sigma_8)
 = [\iota_7]_{10} (\varepsilon_7 + \sigma'\eta_{14})
 = [\iota_7]_{10}\bar{\nu}_{7}+ [\nu_5]_{10}\sigma_8.
\end{equation}
From the last relation, we have $[\iota_7]_{11}\bar{\nu}_{7} = [\nu_5]_{11}\sigma_8$. We know $\Delta(\nu_{13})\ne 0$ (\cite[p. 161, Table]{Ke}, \cite[(3.13)]{Ka}), $R^{13}_{15} = \{[\nu_5]_{13}\sigma_8, [\eta^3_{12}]\}\cong\Z_2\oplus\Z$ and $R^{14}_{15} = \{[\eta^3_{12}]_{14}\}\cong\Z$ \cite[Proposition 2.1]{Ka}. 
So, by $(15)_{13}$ and \eqref{Desig9}, we have 

\begin{equation}\label{Dn13}
\Delta(\nu_{13}) = [\nu_5]_{13}\sigma_8 = [\iota_7]_{13}\bar{\nu}_7 = [\iota_7]_{13}\varepsilon_7.
\end{equation}

Finally we need

\begin{lem} \label{ord}
Suppose that $\alpha\in\pi^{4n+2}_k$ for $n\ge 2$ be an element satisfying the relations
$2\alpha=0,\ (E^2\alpha)\eta_{k+2}=0$\ and\ $\eta^2_{4n+2}(E^3\alpha)=0$.
Then, for an element $\beta\in\{\eta_{4n+2},2\iota_{4n+3},E\alpha\}_1\subset\pi^{4n+2}_{k+2}$, 
there exists a lift $[\eta_{4n+2}(E\beta)]\in R^{4n+3}_{k+2}$ of $\eta_{4n+2}(E\beta)$
which is of order $2$.
\end{lem}
\begin{proof}
By \eqref{De4+2}, we have 
\[\begin{split}
\Delta(\eta_{4n+2}(E\beta))&\in R^{4n+2}_{4n+2}\circ 
2\iota_{4n+2}\circ\beta\\
&\subset 
-R^{4n+2}_{4n+2}\circ\{2\iota_{4n+2},\eta_{4n+2},2\iota_{4n+3}\}_1\circ(E^2\alpha)
\\
&=-R^{4n+2}_{4n+2}\circ\eta^2_{4n+2}(E^3\alpha)=0.
\end{split}\]
So, $\eta_{4n+2}(E\beta)$ has a lift. We set
\[
[\eta_{4n+2}(E\beta)]\in\{[\eta^2_{4n+2}],2\iota_{4n+4},E^2\alpha\}_2.
\]
Then, by \eqref{De^2}, we have
\[\begin{split}
2[\eta_{4n+2}(E\beta)]
&\in\{[\eta^2_{4n+2}],2\iota_{4n+4},E^2\alpha\}_2\circ 2\iota_{k+3}\\
&=-[\eta^2_{4n+2}]\circ\{2\iota_{4n+4},E^2\alpha,2\iota_{k+2}\}_2\\
&\ni[\eta^2_{4n+2}](E^2\alpha)\eta_{k+2}=0
\mod\ [\eta^2_{4n+2}]\circ E^2\pi^{4n+2}_{k+1}\circ 2\iota_{k+3}=0.
\end{split}\]
This completes the proof.
\end{proof}

By \cite[Lemma 6.3, (7.3)]{T}, we have
\begin{equation}\label{t3}
\eta_5\bar{\nu}_6=\nu^3_5,\ \bar{\nu}_6\eta_{14}=\nu^3_6.
\end{equation}

We know $\varepsilon_n\in\{\eta_n,2\iota_{n+1},\nu^2_{n+1}\}_1$ for $n\ge 5$ 
\cite[(6.1)]{T}. By the relation $\eta_n\bar{\nu}_{n+1} = \nu^3_n (n \ge 5)$ 
\eqref{t3} and by \eqref{mu4}, we have $\mu_n\in\{\eta_n,2\iota_{n+1},8\sigma_{n+1}\}_1$ for $n \ge 9$. So, by Lemma \ref{ord}, we have 

\begin{exam}\label{shetep}
$\sharp[\eta_{4n+2}\varepsilon_{4n+3}]=\sharp[\eta_{4n+2}\mu_{4n+3}]=2$ for $n \ge 2$. 
\end{exam}

\section{Determination of $\pi_{23}(R_n:2)$\ ($n\leq 9)$}

Let $n\geq 10$. By \cite[Corollary 3.7, Theorem 7.1]{T}, we have
$\{2\iota_n,8\sigma_n,2\iota_{n+7}\}_1\ni(8\sigma_n)\eta_{n+7}=0 \ \bmod\ \ 2\iota_n\circ E\pi^{n-1}_{n+7}+2\pi^n_{n+8}=0$ 
and $\{8\iota_n,2\sigma_n,8\iota_{n+7}\}_1\equiv 0 \ \bmod \ 8\iota_n\circ E\pi^{n-1}_{n+7}+8\pi^n_{n+8}=0$. Since $E\pi^{n-1}_{n+7}=\pi^n_{n+8}$, we have $\{2\iota_n,8\sigma_n,2\iota_{n+7}\}_1=\{2\iota_n,8\sigma_n,2\iota_{n+7}\}$ and $\{8\iota_n,2\sigma_n,8\iota_{n+7}\}_1=\{8\iota_n,2\sigma_n,8\iota_{n+7}\}$. Hence we obtain
\begin{equation}\label{2i8s2i}
\{2\iota_n,8\sigma_n,2\iota_{n+7}\}=\{8\iota_n,2\sigma_n,8\iota_{n+7}\}=0
\text{\quad for } n\geq 10.
\end{equation}

We show 
\begin{lem}\label{236}
\begin{enumerate}
\item If $i_{5 \ast}R_n^4=0$, where $n \ge 4$, then $p_{5 \ast} :
	   R_n^5 \to \pi_n^4$ is a monomorphism. 
\item If $\pi^3_n=\nu' \circ \pi^6_n$ for $n \ge 6$, then $i_{5
 \ast}R^4_n=0$. 
\end{enumerate}
\end{lem}
\begin{proof}
(1) is a direct consequence of $(n)_4$. 

(2) Recall that 
$R^4_n=[\eta_2]_4 \circ \pi^3_n
 \oplus [\iota_3] \circ \pi^3_n$ (\cite[p.~14]{Ka}, \cite[p. 116]{St}). 
We know $\pi_n(R_3)\cong\pi_n(S^3)$, $\pi_n(R_4)\cong\pi_n(R_3)
\oplus\pi_n(S^3)$. 
We know $\pi_6(R_5) \cong \pi_6(Sp(2))=0$. 
Then, $i_{5 \ast}([\eta_2]_4\nu')=[\eta_2]_5\nu'=0$ and $i_{5 \ast}([\iota_3]\nu')=[\iota_3]_5\nu'=0$. 
\end{proof}
By \cite{MT1}, $\pi^3_{23} = \Z_2\{\nu'\bar{\mu}_6\}
\oplus\Z_2\{\nu'\eta_6\mu_7\sigma_{16}\}$. So we obtain
\begin{gather*}
R^3_{23} = \Z_2\{[\eta_2]\nu'\bar{\mu}_6\}
\oplus\Z_2\{[\eta_2]\nu'\eta_6\mu_7\sigma_{16}\}
\shortintertext{and}
R^4_{23} = \Z_2\{[\eta_2]_4\nu'\bar{\mu}_6\}
\oplus\Z_2\{[\eta_2]_4\nu'\eta_6\mu_7\sigma_{16}\}
\oplus\Z_2\{[\iota_3]\nu'\bar{\mu}_6\}
\oplus\Z_2\{[\iota_3]\nu'\eta_6\mu_7\sigma_{16}\}.
\end{gather*}
By applying Lemma~\ref{236} for $n=23$, we have
\begin{equation}\label{i5R423}
i_{5 \ast}(R^4_{23})=0.
\end{equation} 

We recall from \cite[Theorem 5.1]{MT2} that 
$\pi_{23}(Sp(2): 2) = \Z_2\{[\sigma'\mu_{14}]\}
\oplus\Z_2\{[E\zeta']\}\oplus\Z_2\{[\eta_7\bar{\varepsilon}_8]\}$.
So, 
by \cite[Lemma 1.2]{HKM}, 
we obtain
\[
R^5_{23} = \Z_2\{[\nu_4\sigma'\mu_{14}]\}
\oplus\Z_2\{[\nu_4(E\zeta')]\}
\oplus\Z_2\{[\nu_4\eta_7\bar{\varepsilon}_8]\}.
\]
By the fact that $\mu_{14}\in\{\eta_{14},2\iota_{15},8\sigma_{15}\}_1
\bmod \eta_{14}\varepsilon_{15}, \nu^3_{14}$ \eqref{mu4} and 
${i_5}_*R^4_{23}=0$,
we obtain  
\begin{equation}\label{mR524}
[\nu_4\sigma'\mu_{14}] \in \{[\nu_4\sigma'\eta_{14}], 2\iota_{15},8\sigma_{15}\}_1
\ \bmod\ [\nu_4(E\zeta')],\ [\nu_4\eta_7\bar{\varepsilon}_8].
\end{equation}
Since $2[\nu_4 \sigma'\mu_{14}] \in [\nu_4\sigma'\eta_{14}]\circ\{ 2\iota_{15},8\sigma_{15}, 2\iota_{23}\}= [\nu_4\sigma'\eta_{14}]
\circ 0 = 0$ \eqref{2i8s2i}, 
it follows that $\sharp [\nu_4\sigma'\mu_{14}]=2$.  

Making use of the exact sequence $(23)_5$, by the fact that $\pi^5_{24}
= \Z_8\{\bar{\zeta}_5\}\oplus\Z_2\{\nu_5\mu_8\sigma_{17}\}$ and $\pi^6_{25}
= \Z_8\{\zeta_5\sigma_{16}\}\oplus\Z_2\{\nu_5\bar{\varepsilon}_8\}\oplus\Z_2\{\eta_5\bar{\mu}_6\}$ \cite[Theorems 12.8-9]{T}, we obtain 
\begin{equation} \label{R623}
R^6_{23} = \{[\zeta_5]\sigma_{16}, [\nu_5]\bar{\varepsilon}_8, [\nu_4\sigma'\mu_{14}]_6, [\nu_4(E\zeta')]_6, [\nu_4\eta_7\bar{\varepsilon}_8]_6\}\cong\Z_8\oplus(\Z_2)^4.
\end{equation}

We see that $\pi^3_{24}
= \Z_2\{\nu'\eta_6\bar{\mu}_7\}$ \cite{Mi1}. This yields the groups
$R^3_{24}$ and $R^4_{24}$:
$$
R^3_{24}=\{[\eta_2]\nu'\eta_6\bar{\mu}_7\}\cong\Z_2;\ 
R^4_{24}=\{[\eta_2]_4\nu'\eta_6\bar{\mu}_7,[\iota_3]\nu'\eta_6\bar{\mu}_7\}\cong(\Z_2)^2.
$$

By applying Lemma~\ref{236} for $\pi_{24}^3$, we have
\begin{equation}\label{i5R424}
i_{5 \ast}(R^4_{24})=0.
\end{equation}

We show the following.
\begin{equation} \label{524}
R^5_{24} = \Z_2\{[\nu^2_4]\kappa_{10}\}
\oplus\Z_2\{[\nu_4\sigma'\eta_{14}]\mu_{15}\}.
\end{equation}
\begin{proof}
In the exact sequence $(24)_5$:
$$
\pi^4_{25}\rarrow{\Delta}R^4_{24}\rarrow{i_*}R^5_{24}
\rarrow{p_*}\pi^4_{24}\rarrow{\Delta}R^4_{23}, 
$$
we know that $\Ker\ \{\Delta:\pi^4_{24}\to R^4_{23}\}= \Z_2\{\nu^2_4\kappa_{10},
\nu_4\sigma'\eta_{14}\mu_{15}\}\cong(\Z_2)^2$ and 
$$
\pi^4_{25} = \Z_2\{(E\nu')\eta_7\bar{\mu}_8\}
\oplus\Z_2\{\nu_4\eta_7\bar{\mu}_8\}\oplus\Z_8\{\nu_4\zeta_7\sigma_{18}\}\ 
\ \text{\cite[Theorem A]{Mi1}}.
$$
By relations \eqref{k1}, \eqref{k6} and 
$\nu'\zeta_6=0$ \cite[Proposition 2.4(i)]{Og}, we have
$$
\Delta((E\nu')\eta_7\bar{\mu}_8) = [\eta_2]_4\nu'\eta_6\bar{\mu}_7,\ 
\Delta(\nu_4\eta_7\bar{\mu}_8)\equiv [\iota_3]\nu'\eta_6\bar{\mu}_7
\bmod [\eta_2]_4\nu'\eta_6\bar{\mu}_7
$$
and $\Delta(\nu_4\zeta_7\sigma_{18})\equiv[\iota_3]\nu'\zeta_6=0\ \bmod\ 
[\eta_2]_4\nu'\zeta_6=0$. This completes the proof.
\end{proof}

We recall from \cite[p.~137]{T} that
$$
\bar{\zeta}_5\in\{\zeta_5, 8\iota_{16}, 2\sigma_{16}\}_1\subset\pi^5_{24}.
$$
Since the order of $[\zeta_5]\in R^5_{16}$ is $8$ \cite[Proposition 4.1]{Ka}, we define a lift of $\bar{\zeta}_5$ as
\begin{equation}\label{[bzeta_5]}
[\bar{\zeta}_5]\in\{[\zeta_5], 8\iota_{16}, 2\sigma_{16}\}_1\subset R^6_{24}.
\end{equation}

We show 

\begin{prop} \label{R61}
$R^6_{24} = \Z_8\{[\bar{\zeta}_5]\}
\oplus\Z_2\{[\nu_5]\mu_8\sigma_{17}\}
\oplus\Z_2\{[\nu_4\sigma'\eta_{14}]_6\mu_{15}\}
\oplus\Z_2\{[\nu^2_4]_6\kappa_{10}\}$.
\end{prop}
\begin{proof}
We consider the exact sequence $(24)_6$:
$$
\pi^5_{25}\rarrow{\Delta}R^5_{24}\rarrow{i_*}R^6_{24}
\rarrow{p_*}\pi^5_{24}\rarrow{\Delta}R^5_{23}.
$$
Since $\pi^{5}_{24}=\Z_8\{\bar{\zeta}_5\}\oplus\Z_2\{\nu_5\mu_8\sigma_{17}\}$ \cite[Theorem 12.9]{T} and $\Delta\nu_5=0$ \eqref{Dn5}, we have $\Delta(\pi^5_{24})=0$. We know $\pi^5_{25} = \{\nu^2_5\kappa_{11}, \nu_5\bar{\mu}_8,\nu_5\eta_8\mu_9\sigma_{18}\}$ \cite[p.~46]{MT1}. So, we have $\Delta(\pi^5_{25})=0$.

By \eqref{[bzeta_5]} and \eqref{2i8s2i}, we have
$$
8[\bar{\zeta}_5]\in\{[\zeta_5],8\iota_{16},2\sigma_{16}\}_1\circ 8\iota_{24}= -[\zeta_5]\circ E\{8\iota_{15},2\sigma_{15},8\iota_{22}\}=0.
$$
This implies $\sharp[\bar{\zeta}_5]=8$ and completes the proof.
\end{proof}

Let $r_3:SU(3)\to R_6$ be the canonical inclusion.
By \cite[Corollaries 5.3, 5.4, Theorem 5.5]{A} and \cite{Y},
the fibrations $G_2\rarrow{p_G}G_2/SU(3)=S^6$ and $R_7 \rarrow{p_7}S^6$
give a commutative diagram
\begin{equation} \label{Y}
\begin{CD}
SU(3) @>{i_G}>> G_2 @>{p_G}>> S^6 \\
@V{r_3}VV @V{h}VV @V{=}VV \\
R_6 @>{i_7}>> R_7 @>{p_7}>> S^6.
\end{CD}
\end{equation}

By
\cite[Table 2, Proposition 4.1]{Ka}, \cite[Theorem
2]{KM1} and the relation $\bar{\nu}_7\nu_{15}=\eta_7\sigma_8\nu_{15}$ 
(\eqref{t1}, \eqref{t4}, \eqref{t6}), 
\begin{gather*}
R^7_8=\{[\nu_5]_7,[\eta_6]\eta_7\}\cong(\Z_2)^2,\\
R^7_{16}=\{[\zeta_5]_7,[\nu_4\sigma'\eta_{14}]_7\eta_{15},[\nu^2_4]_7\nu^2_{10},[\nu_5]_7\bar{\nu}_8,[\nu_5]_7\varepsilon_8,[\eta_6]\eta_7\varepsilon_8,[\eta_6]\mu_7\}\cong(\Z_2)^7,\\
R^7_{18}=\{[2P(\iota_{13})]\sigma_{11},[\nu_4\zeta_7]_7,[\nu_5]_7\sigma_8\nu_{15}\}\cong\Z_{16}\oplus\Z_8\oplus\Z_2.
\end{gather*}

Since $\pm[\iota_8,\iota_8] = 2\sigma_8 - E\sigma'$ and the Whitehead product is trivial in $\pi_{15}(R_5)$, we have 
\begin{equation}\label{2[nu5]s}
2[\nu_5]\sigma_8 =[\nu_5](E\sigma').
\end{equation}

\begin{lem} \label{89}
$[\nu_4\sigma'\mu_{14}]\eta_{23}= [\nu_4\sigma'\eta_{14}]\mu_{15}$,
$[\nu_4(E\zeta')]\eta_{23}= [\nu_4\eta_7\bar{\varepsilon}_8]\eta_{23}=0$ \ and \ $[\nu_5]\eta_8\bar{\varepsilon}_9=0$. 
\end{lem}
\begin{proof}
Since $\nu_4\sigma'\mu_{14}\eta_{23} = \nu_4\sigma'\eta_{14}\mu_{15}$,
we have $[\nu_4\sigma'\mu_{14}]\eta_{23} - [\nu_4\sigma'\eta_{14}]\mu_{15}
\in {i_5}_*R^4_{24} = 0$ \eqref{i5R424}. This leads to the first relation.

As $\zeta'\eta_{22} = 0$ by \cite[Proposition 2.13(5)]{Og},
$[\nu_4(E\zeta')]\eta_{23}\in {i_5}_*R^4_{24} = 0$. Since
$\eta_7\bar{\varepsilon}_8 = \nu_7\sigma_{10}\nu^2_{17}$ from Lemma \ref{zesig} (3), we have
$\nu_4\eta_7\bar{\varepsilon}_8\eta_{23} = 0$. This implies the relation
$[\nu_4\eta_7\bar{\varepsilon}_8]\eta_{23} = 0$. By \cite[the proof of (12.15)]{T}, $\nu_5\eta_8\bar{\varepsilon}_9 = 0$.
By \eqref{2[nu5]s} and the fact that $\eta_8\bar{\varepsilon}_9=(E\sigma')\nu^3_{15}$, 
we have $[\nu_5]\eta_8\bar{\varepsilon}_9
=[\nu_5](E\sigma')\nu^3_{15}=0$. 
This leads to the third and completes the proof.
\end{proof}

By Lemma \ref{zesig}(4) and the group structure of $R^5_{23}$, we have
\begin{equation}\label{zesig1}
[\zeta_5]\sigma_{16}\in\{[\nu_5]\sigma_8,8\iota_{15},2\sigma_{15}\}_1 
\bmod 2[\zeta_5]\sigma_{16}, [\nu_5]\bar{\varepsilon}_8,
[\nu_4\sigma'\mu_{14}]_6, [\nu_4(E\zeta')]_6, 
[\nu_4\eta_7\bar{\varepsilon}_8]_6.
\end{equation}

By \cite[Theorem 6.1]{Mi3} and \cite[Theorem 1.3]{MMOh}, we know
\[
\pi_{23}(G_2:2)=\{\langle P(E\theta)+\nu_6\kappa_9\rangle, \langle\eta_6\mu_7\rangle\sigma_{16}\}\cong\Z_4\oplus\Z_2,
\]
where  
$2\langle P(E\theta)+\nu_6\kappa_9\rangle
={i_G}_*[\nu_5\bar{\varepsilon}_8]$.

We will use the diagram \eqref{Y}.  We define 
\[
[P(E\theta)+\nu_6\kappa_9]=h_*\langle P(E\theta)+\nu_6\kappa_9\rangle. 
\]
From the fact that $p_6r_3=p_U$ and \eqref{R623}, we have 
$$
{r_3}_*[\nu_5\bar{\varepsilon}_8]\equiv[\nu_5]\bar{\varepsilon}_8\ \bmod\ 
[\nu_4\sigma'\mu_{14}]_6, [\nu_4(E\zeta')]_6, [\nu_4\eta_7\bar{\varepsilon}_8]_6.
$$
Hence, we obtain
\[\begin{split}
2[P(E\theta)+\nu_6\kappa_9]&=h_*{i_G}_*[\nu_5\bar{\varepsilon}_8]
={i_7}_*{r_3}_*[\nu_5\bar{\varepsilon}_8]\\
&\equiv[\nu_5]_7\bar{\varepsilon}_8
\ \bmod\ 
[\nu_4\sigma'\mu_{14}]_7, [\nu_4(E\zeta')]_7, [\nu_4\eta_7\bar{\varepsilon}_8]_7.
\end{split}\]
That is,
\begin{equation}\label{mmoh}
2[P(E\theta)+\nu_6\kappa_9]\equiv[\nu_5]_7\bar{\varepsilon}_8 \ 
\bmod \ [\nu_4\sigma'\mu_{14}]_7, [\nu_4(E\zeta')]_7, [\nu_4\eta_7\bar{\varepsilon}_8]_7.
\end{equation}

We know $\pi^7_{23}=\{\sigma'\mu_{14},E\zeta',\mu_7\sigma_{16},\eta_7\bar{\varepsilon}_8\}\cong(\Z_2)^4$ \cite[Theorem 12.6]{T}.
So, by Lemma \ref{R^7_k}, we have $R^7_{23}\cong\Z_4\oplus(\Z_2)^5$. 
We show 
\begin{prop} \label{R723}
$R^7_{23} = \Z_2\{[\zeta_5]_7\sigma_{16}\}
\oplus\Z_2\{[\nu_4\sigma'\mu_{14}]_7\}
\oplus\Z_2\{[\nu_4(E\zeta')]_7\}
\oplus\Z_2\{[\nu_4\eta_7\bar{\varepsilon}_8]_7\}
\oplus\Z_2\{[\eta_6]\mu_7\sigma_{16}]\}
\oplus \Z_4\{[P(E\theta)+\nu_6\kappa_9]\}$,
where $2[P(E\theta)+\nu_6\kappa_9]
\equiv[\nu_5]_7\bar{\varepsilon}_8\ \bmod\ 
[\nu_4\sigma'\mu_{14}]_7, [\nu_4(E\zeta')]_7, [\nu_4\eta_7\bar{\varepsilon}_8]_7$.
\end{prop}
\begin{proof}
In the exact sequence $(23)_6$, we know $\Ker\{\Delta:\pi^6_{23}\to R^6_{22}\}
=\{P(E\theta)+\nu_6\kappa_9, \eta_6\mu_7\sigma_{16}\}\cong(\Z_2)^2$ \cite[p. 12-3]{HKM}. We know\\ 
$\pi^6_{24}=\{\zeta_6\sigma_{17}, P(E\theta)\eta_{23}, 
\eta_6\bar{\mu}_7\}\cong\Z_8\oplus(\Z_2)^2$ \cite[Theorem 12.8]{T}. 
We have $\Delta(\zeta_6\sigma_{17}) = 2[\zeta_5]\sigma_{16}$ Lemma \ref{zeta1}.
We have $\Delta(\eta_6\bar{\mu}_7)= 0$ by \eqref{Delta}. Since $P(E\theta)\eta_{23}=
(P(E\theta)+\nu_6\kappa_9)\eta_{23}$, we have $\Delta(P(E\theta)\eta_{23})=0$. 

We consider the natural map up to sign between two exact sequences:

\begin{equation} \label{EHP}
\begin{CD}
\pi^6_{24}@>{\Delta}>>
R^{6}_{23} @>{{i_7}_*}>> R^{7}_{23} @>{{p_7}_*}>> \pi^{6}_{23} \ @>{\Delta}>> R^{6}_{22}\\
@V{E^7}VV @V{J}VV @V{J}VV @V{E^{7}}VV \ @V{J}VV\\
\pi^{13}_{31}@>{P}>>\pi^{6}_{29} @>{E}>> \pi^{7}_{30} @>{H}>> \pi^{13}_{30}
@>{P}>> \pi^{6}_{28},
\end{CD}
\end{equation}
So, by \eqref{R623} and \eqref{mmoh}, we have the desired assertion.
This completes the proof.
\end{proof}

By Lemma \ref{89}, we have 
$[\nu_5]_7\bar{\varepsilon}_8\eta_{23}=0$.
Since $[\nu_4\eta_7\bar{\varepsilon}_8]\eta_{23}\in{i_5}_*R^4_{24}=0$ \eqref{i5R424},
we have $[\nu_4(E\zeta')]_7\eta_{23}=0$. 
Finally, we have 
$[\nu_4\sigma'\mu_{14}]\eta_{23}\equiv[\nu_4\sigma'\eta_{14}]\mu_{15}
 \ \bmod\ {i_5}_*R^4_{24}=0$.
This and  \eqref{mmoh}, we have the following:

If 
$[\nu_4\sigma'\mu_{14}]_7 \eta_{23}= [\nu_4\sigma'\eta_{14}]_7\mu_{15} \neq 0$,
we obtain
\[
2[P(E\theta)+\nu_6\kappa_9]
\equiv[\nu_5]_7\bar{\varepsilon}_8\ \bmod\ 
[\nu_4\eta_7\bar{\varepsilon}_8]_7, [\nu_4(E\zeta')]_7.
\]

The relation in Proposition \ref{R723} which is \eqref{mmoh} will be improved after knowing the group structure of $R^7_{24}$.

Here, we remark the following.
 
\begin{rem}\label{R^7_k}
$R^7_k \cong h_*\pi_k (G_2:2) \oplus [\eta_6]\circ\pi^7_k$ for $k\ge 2$.
\end{rem}
\begin{proof}
Let %be 
$k\ge 2$. For a double covering map $q:Spin(7)\to R_7$, we know
$R^7_k \cong q_*\pi_k(Spin(7):2)$.
We consider the fibration $G_2\xrightarrow{i'} Spin(7) \xrightarrow{p'}S^7$.
We recall from \cite[Proposition 9.1]{Mi3} the following:
\[
\pi_k(Spin(7):2)\cong {i'}_*\pi_k(G_2:2)\oplus \chi\pi^7_k,
\]
where $\chi:\pi^7_k\to \pi_k(Spin(7):2)$ is a homomorphism satisfying ${p'}_*\circ\chi=\rm{id}_{\pi^7_k}$.
Therefore, since $h=q\circ i'$ by the construction of $h$, we obtain
\[
R^7_k \cong {h}_*\pi_k(G_2:2)\oplus (q_*\circ \chi)\pi^7_k.
\]
For $k=7$, by the fact that $R^7_7=\Z\{[\eta_6]\}$ \cite[Table 2]{Ka},
$\pi_7(G:2)=0$ \cite[Theorem 6.1]{Mi3} and $\pi^7_7=\Z\{\iota_7\}$, we have
$[\eta_6]=\pm(q_*\circ \chi)(\iota_7)$ and hence, we have
$[\eta_6]_*= \pm q_*\circ \chi$. Thus, for $k\ge 2$, the relation
$
R^7_k \cong h_*\pi_k (G_2:2) \oplus [\eta_6]\circ\pi^7_k$ holds.
This completes the proof.
\end{proof}

By \eqref{Ogu0} and \eqref{4z5}, we have
\begin{equation}\label{4z5sig}
\eta_5\varepsilon_6\mu_{14}
= 4\zeta_5\sigma_{16}.
\end{equation}

Recall that, if $i_{5 \ast}R_{23}^4=0$, $p_{5 \ast} :
	   R_{23}^5 \to \pi_{23}^4$ is a monomorphism (Lemma \ref{236}). 
	  	   
\begin{lem} \label{rel1}
\begin{enumerate}
\item $[\nu_4\sigma'\eta_{14}]\bar{\nu}_{15}
= [\nu_4\eta_7\bar{\varepsilon}_8] = [\nu^2_4]\sigma_{10}\nu^2_{17}$ and
$[\nu_4\sigma'\eta_{14}]\varepsilon_{15}
= [\nu_4(E\zeta')]$.
\item $[\eta_5\varepsilon_6]\nu^3_{14}
= [\nu_4\eta_7\bar{\varepsilon}_8]_6$,
$[\eta_5\varepsilon_6]\eta_{14}\varepsilon_{15}
= [\nu_4(E\zeta')]_6$, $[\eta_6]\eta_7\bar{\varepsilon}_8
= [\nu_4\eta_7\bar{\varepsilon}_8]_7$, \\
$[\eta_6]E\zeta' = [\nu_4(E\zeta')]_7$.
\item $[\eta_5\varepsilon_6]\mu_{14}+[\nu_4\sigma'\mu_{14}]_6
\equiv 4[\zeta_5]\sigma_{16} \ 
\bmod [\nu_4(E\zeta')]_6, [\nu_4\eta_7\bar{\varepsilon}_8]_6$\ and\ 
%\item 
$[\eta_6]\sigma'\mu_{14}=[\eta_5\varepsilon_6]_7\mu_{14}
\equiv[\nu_4\sigma'\mu_{14}]_7\ \bmod \ 
[\nu_4\eta_7\bar{\varepsilon}_8]_7, [\nu_4(E\zeta')]_7$.
\end{enumerate}
\end{lem}
\begin{proof}
We recall 
${i_5}_*R^4_{23}=0$ \eqref{i5R423}. So, by Lemma \ref{zesig}(3), we obtain
$[\nu_4\sigma'\eta_{14}]\bar{\nu}_{15} = [\nu_4\sigma'\eta_{14}\bar{\nu}_{15}]
 = [\nu_4\eta_7\bar{\varepsilon}_8]$ and 
$[\nu^2_4]\sigma_{10}\nu^2_{17} = [\nu^2_4\sigma_{10}\nu^2_{17}]
 = [\nu_4\eta_7\bar{\varepsilon}_8]$. 
Similarly, by \cite[(12.4)]{T}, we have the second relation of (1).

By \cite[Lemma 1.1(v)]{KM1}, we know
\begin{equation}\label{et6s'}
[\eta_6]\sigma'=4[\bar{\nu}_6+\varepsilon_6]+[\nu_5]_7\nu^2_8+[\eta_5\varepsilon_6]_7
\end{equation}
By \cite[Lemma 1.2(iii)]{KM1}, % and \cite[p. 162, Table]{Ke}, 
we know
\begin{equation}\label{km1}
[\eta_5\varepsilon_6]\eta_{14}
=4[\nu_5]\sigma_8 + [\nu_4\sigma'\eta_{14}]_6,\ [\eta_5\varepsilon_6]_7\eta_{14}
=[\nu_4\sigma'\eta_{14}]_7
\end{equation}

By 
\eqref{km1} and \eqref{t3},
we have 
$%\begin{split}
[\eta_5\varepsilon_6]\nu^3_{14}  = [\eta_5\varepsilon_6]\eta_{14}\bar{\nu}_{15}
= 4[\nu_5]\sigma_8\bar{\nu}_{15} + [\nu_4\sigma'\eta_{14}]_6\bar{\nu}_{15}
=[\nu_4\sigma'\eta_{14}]_6\bar{\nu}_{15}$
and
$
[\eta_5\varepsilon_6]\eta_{14}\varepsilon_{15} = [\nu_4\sigma'\eta_{14}]_6\varepsilon_{15}.
$
So, by (1), we obtain
$[\eta_5\varepsilon_6]\nu^3_{14} = [\nu_4\sigma'\eta_{14}]_6\bar{\nu}_{15}
= [\nu_4\eta_7\bar{\varepsilon}_8]_6$ and
$[\eta_5\varepsilon_6]\eta_{14}\varepsilon_{15}
= [\nu_4\sigma'\eta_{14}]_6\varepsilon_{15} = [\nu_4(E\zeta')]_6$. 
By Lemma \ref{zesig}(3) and relations 
$[\eta_6]\nu_7 = b[\nu^2_4]_7$ for $b$ odd (\cite[Lemma 1.1(iv)]{KM1}),
$2[\eta_6]\eta_7\bar{\varepsilon}_8=0$ and (1), we obtain
$[\eta_6]\eta_7\bar{\varepsilon}_8 = [\eta_6]\nu_7\sigma_{10}\nu^2_{17}
 = [\nu^2_4]_7\sigma_{10}\nu^2_{17} = [\nu_4\eta_7\bar{\varepsilon}_8]_7$.
By \eqref{et6s'} and \eqref{t1}, we have
$[\eta_6]\sigma'\eta_{14} = 4[\bar{\nu}_6+\varepsilon_6]\eta_{14}
 + [\nu_5]_7\nu^2_8\eta_{14} + [\eta_5\varepsilon_6]_7\eta_{14}
 =[\eta_5\varepsilon_6]_7\eta_{14}$. 
So, by relations $E\zeta'=\sigma'\eta_{14}\varepsilon_{15}$ (\cite[(12.4)]{T}
 and (1), we obtain 
$[\eta_6]E\zeta'
= [\eta_6]\sigma'\eta_{14}\varepsilon_{15}
= [\eta_5\varepsilon_6]_7\eta_{14}\varepsilon_{15}
= [\nu_4(E\zeta')]_7$. Hence we obtain (2).

By the relation 
${p_6}_*([\eta_5\varepsilon_6]\mu_{14})={p_6}_\ast(4[\zeta_5]\sigma_{16})$
from \eqref{4z5sig}, we have
\[
[\eta_5\varepsilon_6]\mu_{14}\equiv 4[\zeta_5]\sigma_{16}\ \bmod\ {i_6}_*R^5_{23}
= \{[\nu_4\sigma'\mu_{14}]_6, [\nu_4(E\zeta')]_6,
[\nu_4\eta_7\bar{\varepsilon}_8]_6\}.
\]
By \eqref{mR524}, \eqref{et6s'} and \eqref{km1}, we have 
\[\begin{split}
[\eta_6]\sigma'\mu_{14}&=[\eta_5\varepsilon_6]_7\mu_{14}
\in\{[\nu_4\sigma'\eta_{14}],2\iota_{15},8\sigma_{15}\}_1\ni[\nu_4\sigma'\mu_{14}]_7\\
&\qquad
\bmod\ [\nu_4\sigma'\eta_{14}]\circ(E\pi^{14}_{22})+8R^7_{16}\circ\sigma_{16}=\{[\nu_4\eta_7\bar{\varepsilon}_8]_7, [\nu_4(E\zeta')]_7\}.
\end{split}\]
This completes the proof.
\end{proof}

We remark that the following.
\[
[\eta_6]\circ\pi^7_{23}=\{[\eta_6]\mu_7\sigma_{16}, [\nu_4\sigma'\mu_{14}], [\nu_4(E\zeta')]_7, [\nu_4\eta_7\bar{\varepsilon}_8]_7\},
\]
where $[\eta_6](E\zeta')=[\nu_4(E\zeta')]_7$, $[\eta_6]\eta_7\bar{\varepsilon}_8=[\nu_4\eta_7\bar{\varepsilon}_8]_7$ and 
$[\eta_6]\sigma'\mu_{14}\equiv[\nu_4\sigma'\mu_{14}]_7\ \bmod\ 
[\nu_4(E\zeta')]_7, [\nu_4\eta_7\bar{\varepsilon}_8]_7$.

By relations \eqref{kaet},
$\eta_5\varepsilon_6\kappa_{14}=4\nu_5\bar\kappa_8$ \cite[I-Proposition 3.1(2)]{Od1}
and $4\bar\kappa_7=\nu^2_7\kappa_{13}$ \cite[Lemma 15.4]{MT1}, we have
\begin{equation}\label{ep5bep}
\varepsilon_5\bar{\varepsilon}_{13}=\varepsilon_5\eta_{13}\kappa_{14}
=\eta_5\varepsilon_6\kappa_{14}=4\nu_5\bar\kappa_8=\nu^3_5\kappa_{14}.
\end{equation}
So, by $\bar{\nu}_6\eta_{14}=\nu^3_6$ \eqref{t3}, we have 
\begin{equation*}\label{bn6+epbep}
\bar{\nu}_6\bar{\varepsilon}_{14}=\nu^3_6\kappa_{15}=\varepsilon_6\bar{\varepsilon}_{14}.
\end{equation*}
By the relation $J[\nu_5]=\bar{\nu}_6+\varepsilon_6$ \cite[Lemma 1.1(ii)]{KM1},
we have 
\[
J([\nu_5]\bar{\varepsilon}_8)=J[\nu_5]\circ \bar{\varepsilon}_{14}
 =(\bar{\nu}_6+\varepsilon_6)\bar{\varepsilon}_{14}.
\]
So, %by the relation \eqref{bn6+epbep}, 
we obtain
\begin{equation}\label{Jbar}
J([\nu_5]\bar{\varepsilon}_8)=0.
\end{equation}

By the fact that $\pi^7_{27}=\{\bar{\kappa}_7\}\cong\Z_8$ \cite[p. 48]{MT1} 
and the proof of \cite[(12.16)]{T}, we have 
\begin{equation}\label{Hbk7}
H(\bar{\kappa}_7)\equiv\kappa_{13}\ \bmod\ 8\sigma^2_{13}.
\end{equation}

We show
\begin{lem}\label{JP}
$J[P(E\theta)+\nu_6\kappa_9]=\bar{\kappa}_7\nu_{27}-\nu_7\bar{\kappa}_{10}$.
\end{lem}
\begin{proof}
Since $HJ[\nu_4\sigma'\mu_{14}]=E^5(\nu_4\sigma'\mu_{14})=0$, we have 
$J[\nu_4\sigma'\mu_{14}]\equiv 0\ \bmod\ E\pi^4_{27}=\{\nu_5\bar{\kappa}_8\}$ 
\cite[p. 19]{MMO}. 
$\nu_5\bar{\kappa}_8$ does not in the J-image, we have 
$
J[\nu_4\sigma'\mu_{14}]=0.
$
By the same argument, we have 
$J[\nu_4(E\zeta')]=0$ and $J[\nu_4\eta_7\bar{\varepsilon}_8]=0$.
Hence, by \eqref{mmoh} and \eqref{Jbar},
we obtain 
$$
2J[P(E\theta)+\nu_6\kappa_9]=0.
$$

We recall from \cite[Theorem 1.1(a), (3.2), (4.9)]{MMO} that 
\[
2\bar{\rho}_{11}=E^4\bar{\rho}',\ 
E\pi^6_{29} = \{2\bar{\rho}',\nu_7\bar{\kappa}_{10}, \phi_7\}\cong\Z_4\oplus\Z_8\oplus\Z_2
\]
and
\[
\pi^s_{23}=\{\bar{\rho}, \nu\bar{\kappa}, \phi\}\cong\Z_{16}\oplus\Z_8\oplus\Z_2.
\]
We have 
\[\begin{split}
HJ[P(E\theta)+\nu_6\kappa_9] &
= E^7(P(E\theta)+\nu_6\kappa_9) = 
\kappa_{13}\nu_{27} = H(\bar{\kappa}_7\nu_{27})
\end{split}
\]
%from \cite[(4.10)]{MMO} and 
by \eqref{Hbk7}. So, by \cite[Theorem 1.1(a), (3.2)]{MMO},
we obtain 
\[
J[P(E\theta)+\nu_6\kappa_9]\equiv\bar{\kappa}_7\nu_{27} \bmod E\pi^6_{29} = \{2\bar{\rho}',
\nu_7\bar{\kappa}_{10}, \phi_7\}.
\]
We set
\[
J[P(E\theta)+\nu_6\kappa_9]
=\bar{\kappa}_7\nu_{27}
+a\nu_7\bar{\kappa}_{10}
+2b\bar{\rho}' + c\phi_7,
\]
where $a\in\{0,1,\dots,7\}$, $b\in\{0,1,2,3\}$ and $c\in\{0,1\}$.
We know %$2J[P(E\theta)+\nu_6\kappa_9]=0$,  
$[P(E\theta)+\nu_6\kappa_9]_\infty=0$ and $\bar{\kappa}\nu= \nu\bar{\kappa}$.
This implies $a=7$, $b=0$ and $c=0$. 
Hence,  
we obtain this lemma.
\end{proof}

Since $R^8_{23}\cong R^7_{23}\oplus\pi^7_{23}$ and $\pi^7_{23}=\{\sigma'\mu_{14}, E\zeta', \eta_7\bar{\varepsilon}_8,
\mu_7\sigma_{16}\}\cong(\Z_2)^4$ \cite{T}, we have the following.
\begin{prop} \label{R8}
$R^8_{23} = \Z_2\{[\zeta_5]_8\sigma_{16}\}
\oplus\Z_2\{[\nu_4\sigma'\mu_{14}]_8\}
\oplus\Z_2\{[\nu_4(E\zeta')]_8\}
\oplus\Z_2\{[\nu_4\eta_7\bar{\varepsilon}_8]_8\}
\oplus\Z_2\{[\eta_6]_8\mu_7\sigma_{16}\}
\oplus \Z_4\{[P(E\theta)+\nu_6\kappa_9]\}
\oplus\Z_2\{[\iota_7]\sigma'\mu_{14}\}\oplus\Z_2\{[\iota_7](E\zeta')\}\\
\oplus\Z_2\{[\iota_7]\eta_7\bar{\varepsilon}_8\}
\oplus\Z_2\{[\iota_7]\mu_7\sigma_{16}\}$.
\end{prop}

By \cite[(7.25)]{T}, we have
\begin{equation}\label{n6m}
\nu_6\mu_9 = 8P(\sigma_{13}).
\end{equation}

We consider $J: R^9_{14}\to\pi^9_{23}$.
We know $R^9_{14}=\{[\bar{\nu}_6+\varepsilon_6]_9, [\nu_5]_9\nu^2_8\}
\cong\Z_8\oplus\Z_2$ \cite[Table 2]{Ka} and $\pi^9_{23}=\{\sigma^2_9,\kappa_9\}
\cong\Z_{16}\oplus\Z_4$ \cite[Theorem 10.3]{T}. So, $J$ is a monomorphism.
By \eqref{Ji7} and \eqref{Jbn6e6}, we have 
$$
J([\iota_7]_9\sigma')=2\sigma^2_9=J([\bar{\nu}_6+\varepsilon_6]_9).
$$
Hence, we have 
$$
[\iota_7]_9\sigma'=[\bar{\nu}_6+\varepsilon_6]_9.
$$
This implies 
$$
[\iota_7]_8\sigma'-[\bar{\nu}_6+\varepsilon_6]_8\in\Delta\pi^8_{15}.
$$
We know $\pi^8_{15}=\{\sigma_8, E\sigma'\}\cong\Z\oplus\Z_8$ \cite[Prposition 5.15]{T}. So, by \cite[Lemma 1.2(ii)]{KM1}, we have 
%\begin{equation}\label{Dsg81}
$$
\Delta\sigma_8\equiv[\iota_7]\sigma'-[\bar{\nu}_6+\varepsilon_6]_8\ 
\bmod\ \Delta(E\sigma').
$$
%\end{equation}
We have $\Delta(E\sigma')=2[\iota_7]\sigma'-[\eta_6]_8\sigma'$. 
So, by \eqref{et6s'}, we obtain
\begin{equation}\label{Dsg82}
\Delta\sigma_8\equiv x[\iota_7]\sigma'+y[\bar{\nu}_6+\varepsilon_6]_8\ 
\bmod\ [\nu_5]_7\nu^2_8+[\eta_5\varepsilon_6]_7\ \mbox{for odd}\ x, y.
\end{equation}

By \cite[Lemma 2.3(ii)]{KM1}, we also know
\begin{equation} \label{tr2}
[\bar{\nu}_6+\varepsilon_6]\eta_{14}
\equiv [\eta_6](\bar{\nu}_7+\varepsilon_7) + [\nu_5]_7\sigma_8 \bmod
[\nu_4\sigma'\eta_{14}]_7.
\end{equation}
 
By \cite[Proposition 2.13(8)]{Og}, we know
\begin{equation}\label{bn6m}
\bar{\nu}_6\mu_{14}=0
\end{equation}

Next we show the following.
\begin{lem} \label{nusiep}
\begin{enumerate}
\item $[\nu_5]\sigma_8\bar{\nu}_{15}
\equiv [\nu_5]\sigma_8\varepsilon_{15}\equiv [\nu_5]\bar{\varepsilon}_8
\bmod [\nu_4\sigma'\mu_{14}]_6, [\nu_4(E\zeta')]_6,
[\nu_4\eta_7\bar{\varepsilon}_8]_6$.
\item $[\bar{\nu}_6+\varepsilon_6]\nu^3_{14}
\equiv [\bar{\nu}_6+\varepsilon_6]\eta_{14}\varepsilon_{15}
\equiv %[\nu_4\eta_7\bar{\varepsilon}_8]_7 + 
[\nu_5]_7\bar{\varepsilon}_8\ 
%\hspace{20mm} 
\bmod \
\ [\nu_4\sigma'\mu_{14}]_7, [\nu_4(E\zeta')]_7,
[\nu_4\eta_7\bar{\varepsilon}_8]_7$.
\end{enumerate}
\end{lem}
\begin{proof}
By Lemma \ref{zesig}(4), $[\nu_5]\sigma_8\bar{\nu}_{15}
- [\nu_5]\bar{\varepsilon}_8\in {i_6}_*R^5_{23}$ and
$[\nu_5]\sigma_8\varepsilon_{15}- [\nu_5]\bar{\varepsilon}_8\in {i_6}_*R^5_{23}$.
This leads to (1).

By relations $\nu^3_{14}=\eta_{14}\bar{\nu}_{15}$ \eqref{t3}
and \eqref{tr2}, we have
\[\begin{split}
[\bar{\nu}_6+\varepsilon_6]\nu^3_{14}
=[\bar{\nu}_6+\varepsilon_6]\eta_{14}\bar{\nu}_{15}
\equiv [\eta_6](\bar{\nu}_7+\varepsilon_7)\bar{\nu}_{15}
+[\nu_5]_7\sigma_8\bar{\nu}_{15} \ \bmod \ [\nu_4\sigma'\eta_{14}]_7\bar{\nu}_{15}.
\end{split}\]
By Lemma \ref{zesig}(2)-(3), \eqref{et6s'}, \eqref{km1} and
Lemma \ref{rel1}(1), we have
\[\begin{split}
[\eta_6](\bar{\nu}_7+\varepsilon_7)\bar{\nu}_{15}
&=[\eta_6]\eta_7\bar{\varepsilon}_8
=[\eta_6]\sigma'\nu^3_{14}=[\eta_5\varepsilon_6]_7\nu^3_{14}\\
&=[\eta_5\varepsilon_6]_7\eta_{14}\bar{\nu}_{15}
=[\nu_4\sigma'\eta_{14}]_7\bar{\nu}_{15}=[\nu_4\eta_7\bar{\varepsilon}_8]_7.
\end{split}\]
So, by (1), we have
$$
[\bar{\nu}_6+\varepsilon_6]\nu^3_{14}\equiv[\nu_5]_7\bar{\varepsilon}_8 \ 
\bmod \
\ [\nu_4\sigma'\mu_{14}]_7, [\nu_4(E\zeta')]_7,
[\nu_4\eta_7\bar{\varepsilon}_8]_7.
$$

The rest is obtained by the parallel arguments. 
This completes the proof.
\end{proof}

We show the following.

\begin{lem} \label{bnuepmu}
$[\bar{\nu}_6+\varepsilon_6]\mu_{14}\equiv
[\eta_6]\mu_7\sigma_{16} + [\zeta_5]_7\sigma_{16} \\
\bmod \ [\nu_5]_7\bar{\varepsilon}_8, [\nu_4\sigma'\mu_{14}]_7,
[\nu_4(E\zeta')]_7, [\nu_4\eta_7\bar{\varepsilon}_8]_7$.
\end{lem}
\begin{proof}

By \eqref{Ogu0} and \eqref{bn6m}, we have $(\bar{\nu}_6+\varepsilon_6)\mu_{14}=\eta_6\mu_7\sigma_{16}$. So, we have 
$$
[\bar{\nu}_6+\varepsilon_6]\mu_{14}\equiv[\eta_6]\mu_7\sigma_{16}\ \bmod\ 
{i_7}_*R^6_{23}.
$$
Hence, by \eqref{R623}, we obtain 
$$
[\bar{\nu}_6+\varepsilon_6]\mu_{14}\equiv[\eta_6]\mu_7\sigma_{16}\ \bmod\ 
[\zeta_5]_7\sigma_{16}, [\nu_5]_7\bar{\varepsilon}_8, [\nu_4\sigma'\mu_{14}]_7, [\nu_4\eta_7\bar{\varepsilon}_8]_7, [\nu_4(E\zeta')]_7.
$$

By \eqref{tr2}, \eqref{mu4}, \eqref{zesig1} and \eqref{mR524},
we have
\[\begin{split}
[\bar{\nu}_6+\varepsilon_6]\mu_{14}
&\in\{[\bar{\nu}_6+\varepsilon_6]\eta_{14}, 8\iota_{15}, 2\sigma_{15}\}\\
&=\{[\eta_6](\bar{\nu}_7+\varepsilon_7)+[\nu_5]_7\sigma_8
+x[\nu_4\sigma'\eta_{14}]_7, 8\iota_{15}, 2\sigma_{15}\}\\
&\subset\{[\eta_6](\bar{\nu}_7+\varepsilon_7), 8\iota_{15}, 2\sigma_{15}\}
+ \{[\nu_5]_7\sigma_8, 8\iota_{15}, 2\sigma_{15}\}\\
&\qquad\qquad+\{x[\nu_4\sigma'\eta_{14}]_7, 8\iota_{15}, 2\sigma_{15}\}\\ 
&\ni[\eta_6](E\zeta'+\mu_7\sigma_{16}) + [\zeta_5]_7\sigma_{16}
+ x[\nu_4\sigma'\mu_{14}]_7\  (x\in\{0, 1\}).
\end{split}\]
This completes the proof.
\end{proof}

By \cite[p. 77]{T}, we have 
\begin{equation}\label{nu45}
\nu^4_5=P(\bar{\nu}_{11}).
\end{equation}

Now we show the following.

\begin{prop} \label{R923}
$R^9_{23} = \Z_{4}\{[P(E\theta)+\nu_6\kappa_9]_9\}\oplus\Z_2\{[\iota_7]_9\mu_7\sigma_{16}\}
\oplus\Z_2\{[\zeta_5]_9\sigma_{16}\}$, where 
$2[P(E\theta)+\nu_6\kappa_9]_9=[\nu_5]_9\bar{\varepsilon}_8$.

\end{prop}
\begin{proof}
In the exact sequence $(23)_8$:
$$
\pi^8_{24}\rarrow{\Delta}R^8_{23}\rarrow{i_*}R^9_{23}\rarrow{p_*}\pi^8_{23}
\rarrow{\Delta}R^8_{22},
$$
$\Delta: \pi^8_{23}\to R^8_{22}$ is a monomorphism and
$$
\pi^8_{24} = \{\sigma_8\nu^3_{15}, \sigma_8\eta_{15}\varepsilon_{16},
\sigma_8\mu_{15}, (E\sigma')\mu_{15}, E^2\zeta', \mu_8\sigma_{17},
\eta_8\bar{\varepsilon}_9\}\cong(\Z_2)^7.
$$
By Lemma \ref{rel1}(3) and the relation $\Delta\iota_8 = 2[\iota_7]
- [\eta_6]_8$ \eqref{k3}, we obtain the following:
$$
\Delta((E\sigma')\mu_{15}) = [\eta_6]_8\sigma'\mu_{14}
\equiv [\nu_4\sigma'\mu_{14}]_8\ \bmod\  [\nu_4(E\zeta')]_8,
[\nu_4\eta_7\bar{\varepsilon}_8]_8;
$$
$$
\Delta(E^2\zeta') = [\eta_6]_8E\zeta' = [\nu_4(E\zeta')]_8;
\Delta(\mu_8\sigma_{17}) = [\eta_6]_8\mu_7\sigma_{16};
$$
$$
\Delta(\eta_8\bar{\varepsilon}_9) = [\eta_6]_8\eta_7\bar{\varepsilon}_8
= [\nu_4\eta_7\bar{\varepsilon}_8]_8.
$$
By %Lemma \ref{rel1} and 
\eqref{Dsg82}, we have
$$
\Delta(\sigma_8\eta_{15}\varepsilon_{16})
\equiv [\iota_7]E\zeta'
+ [\bar{\nu}_6+\varepsilon_6]_8\eta_{14}\varepsilon_{15} \bmod 
[\eta_5\varepsilon_6]_7\eta_{14}\varepsilon_{15}.
$$
So, by Lemmas \ref{rel1}(1) and \ref{nusiep}(2),
$$
\Delta(\sigma_8\eta_{15}\varepsilon_{16})
\equiv [\iota_7]E\zeta' + [\nu_4\eta_7\bar{\varepsilon}_8]_8
+ [\nu_5]_8\bar{\varepsilon}_8 \bmod [\nu_4\sigma'\mu_{14}]_8,
[\nu_4(E\zeta')]_8, [\nu_4\eta_7\bar{\varepsilon}_8]_8.
$$
By the fact that $\sigma'\nu^3_{14} = \eta_7\bar{\varepsilon}_8$ (Lemma \ref{zesig}(3)) and
$\nu^5_8 = 0$ \eqref{nu45}, %\cite[Proposition 2.6(3)]{Og}, 
we have
$$
\Delta(\sigma_8\nu^3_{15})\equiv [\iota_7]\eta_7\bar{\varepsilon}_8
+ [\bar{\nu}_6+\varepsilon_6]_8\nu^3_{14} \bmod
[\eta_5\varepsilon_6]_8\nu^3_{14}.
$$
By Lemma \ref{nusiep}(2),
$$
[\bar{\nu}_6+\varepsilon_6]_8\nu^3_{14}
= [\bar{\nu}_6+\varepsilon_6]_8\eta_{14}\bar{\nu}_{15}
\equiv [\bar{\nu}_6+\varepsilon_6]_8\eta_{14}\varepsilon_{15} \bmod
[\nu_4\sigma'\mu_{14}]_8, [\nu_4(E\zeta')]_8, 
[\nu_4\eta_7\bar{\varepsilon}_8]_8.
$$
By Lemma \ref{rel1}.(2), $[\eta_5\varepsilon_6]_8\nu^3_{14}
= [\nu_4\eta_7\bar{\varepsilon}_8]_8$. Hence we obtain
$$
\Delta(\sigma_8\nu^3_{15})\equiv [\iota_7]\eta_7\bar{\varepsilon}_8
+ [\nu_4\eta_7\bar{\varepsilon}_8]_8 + [\nu_5]_8\bar{\varepsilon}_8 \bmod
[\nu_4\sigma'\mu_{14}]_8, [\nu_4(E\zeta')]_8,
[\nu_4\eta_7\bar{\varepsilon}_8]_8.
$$

Next, by \eqref{Dsg82},  we have
$$
\Delta(\sigma_8\mu_{15})\equiv [\iota_7]\sigma'\mu_{14}
+ [\bar{\nu}_6+\varepsilon_6]_8\mu_{14} \bmod
[\eta_5\varepsilon_6]_8\mu_{14}.
$$
By Lemma \ref{bnuepmu},
$$
[\bar{\nu}_6+\varepsilon_6]_8\mu_{14}
\equiv [\eta_6]_8(E\zeta'+\mu_7\sigma_{16}) + [\zeta_5]_8\sigma_{16}
\ \bmod [\nu_5]_8\bar{\varepsilon}_8, [\nu_4\sigma'\mu_{14}]_8,
[\nu_4(E\zeta')]_8, [\nu_4\eta_7\bar{\varepsilon}_8]_8.
$$
By Lemma \ref{rel1}(3), 
$$
[\eta_5\varepsilon_6]_8\mu_{14}\equiv 0 \ \bmod \
[\nu_4\sigma'\mu_{14}]_8, 
[\nu_4(E\zeta')]_8, [\nu_4\eta_7\bar{\varepsilon}_8]_8.
$$
Hence we obtain
$$
\Delta(\sigma_8\mu_{15})\equiv [\iota_7]\sigma'\mu_{14}
+ [\zeta_5]_8\sigma_{16} + [\eta_6]_8\mu_7\sigma_{16} \bmod
[\nu_5]_8\bar{\varepsilon}_8, [\nu_4\sigma'\mu_{14}]_8,
[\nu_4(E\zeta')]_8, [\nu_4\eta_7\bar{\varepsilon}_8]_8.
$$
Thus we obtain the group $R^9_{23}$, completing the proof.
\end{proof}

In the exact sequence $(23)_9$:
$$
\pi^9_{24}\rarrow{\Delta}R^9_{23}\rarrow{i_*}R^{10}_{23}
\rarrow{p_*}\pi^9_{23}\rarrow{\Delta}R^9_{22},
$$
we know
$$
\pi^9_{24} = \Z_{16}\{\rho'\}\oplus\Z_2\{\sigma_9\bar{\nu}_{16}\}
\oplus\Z_2\{\sigma_9\varepsilon_{16}\}\oplus\Z_2\{\bar{\varepsilon}_9\}.
$$
By \eqref{Desig9}, Lemmas \ref{zesig} (2) and \ref{rel2} (2), we have 
\[\begin{split}
\Delta(\sigma_9\bar{\nu}_{16})&
 = [\iota_7]_9(\bar{\nu}_7+\varepsilon_7)\bar{\nu}_{15}
 = [\iota_7]_9\eta_7\bar{\varepsilon}_8 =[\nu_5]_9\bar{\varepsilon}_8,
\\
\Delta(\sigma_9\varepsilon_{16})&
 = [\iota_7]_9(\bar{\nu}_7+\varepsilon_7)\varepsilon_{15}
 = [\iota_7]_9\eta_7\bar{\varepsilon}_8 =[\nu_5]_9\bar{\varepsilon}_8.
\end{split}\]
That is,
\begin{equation}\label{Desig9e}
\Delta(\sigma_9\bar{\nu}_{16})=\Delta(\sigma_9\varepsilon_{16})=[\nu_5]_9\bar{\varepsilon}_8.
\end{equation}

Summarizing the above results, we obtain the following.

\begin{lem} \label{rel2}
\begin{enumerate}
\item $[\eta_5\varepsilon_6]_9\mu_{14}
=[\eta_5\varepsilon_6]_9\nu^3_{14}
=[\eta_5\varepsilon_6]_9\eta_{14}\varepsilon_{15}
=[\eta_6]_9\mu_7\sigma_{16} = [\nu_4\sigma'\mu_{14}]_9
=[\nu_4\eta_7\bar{\varepsilon}_8]_9 = [\nu_4(E\zeta')]_9=0$.
\item $[\bar{\nu}_6+\varepsilon_6]_9\nu^3_{14}
=[\bar{\nu}_6+\varepsilon_6]_9\eta_{14}\varepsilon_{15}
=[\nu_5]_9\bar{\varepsilon}_8 = [\nu_5]_9\sigma_8\bar{\nu}_{15}
=[\nu_5]_9\sigma_8\varepsilon_{15} = [\iota_7]_9\eta_7\bar{\varepsilon}_8
=[\iota_7]_9E\zeta'$.
\item $[\iota_7]_9\sigma'\mu_{14}
=[\bar{\nu}_6+\varepsilon_6]_9\mu_{14}
\equiv [\zeta_5]_9\sigma_{16} \bmod [\nu_5]_9\bar{\varepsilon}_8$.
\item $[\bar{\nu}_6+\varepsilon_6]_{10}\nu^3_{14}
= [\bar{\nu}_6+\varepsilon_6]_{10}\eta_{14}\varepsilon_{15}
= [\nu_5]_{10}\bar{\varepsilon}_8 = [\nu_5]_{10}\sigma_8\bar{\nu}_{15}
= [\nu_5]_{10}\sigma_8\varepsilon_{15}
= [\iota_7]_{10}\eta_7\bar{\varepsilon}_8
= [\iota_7]_{10}E\zeta' = 0$.
\end{enumerate}
\end{lem}

\section{Determination of $\pi_{23}(R_n:2) (n\geq 10)$}
\hspace{2ex}

By \cite[Lemma 10.7]{T}, we have 
\begin{equation}\label{epsig1}
\varepsilon_3\sigma_{11}=0,\ \bar{\nu}_6\sigma_{14}=0.
\end{equation}

By \cite[Proposition 4.1]{Ka} and its proof, we know
\begin{equation}\label{R1016}
R^{10}_{16}=\{\Delta\sigma_{10}, [\iota_7]_{10}\mu_7, [\iota_7]_{10}\nu^3_7\}
\cong \Z_{16}\oplus (\Z_2)^2.
\end{equation}

We recall $\pi^{9}_{23}=\{\sigma^2_9,\kappa_9\}\cong\Z_{16}\oplus\Z_4$
\cite[Theorem 10.3]{T}. Since $2\Delta\kappa_9=2\Delta(\iota_9)\circ\kappa_8=0$, we have
$\Delta(\bar{\nu}_9\nu^2_{17}) = 2\Delta\kappa_9 = 0$ in $R^9_{22}$.
So, in $R^{10}_{23}$, there exists  a lift $[\bar{\nu}_9\nu^2_{17}]$ of 
$\bar{\nu}_9\nu^2_{17}$. 

By relations \eqref{Desig9} and \eqref{epsig1}, we have $\Delta(\sigma^2_9)
= [\iota_7]_9(\bar{\nu}_7+\varepsilon_7)\sigma_{15} = 0$.
So, in $R^{10}_{23}$, there exist a lift
$[\sigma^2_9]$ of $\sigma^2_9$. 
By \eqref{M-m}, we have
\[\begin{split}
{p_{10}}_\ast \{[\iota_7]_{10},\bar{\nu}_7+\varepsilon_7, \sigma_{15}\}
&= - \{p_{10}, [\iota_7]_{10},\bar{\nu}_7+\varepsilon_7\}\circ\sigma_{16}\\
&\supset - \{p_{10}, i_{10}, \Delta\iota_9\}\circ\sigma^2_9\ni \sigma^2_9 \bmod
2\sigma^2_9.
\end{split}\]
So we define
\begin{equation}\label{[sig^2_9]}
[\sigma^2_9]\in\{[\iota_7]_{10},\bar{\nu}_7+\varepsilon_7,\sigma_{15}\}.
\end{equation}

Since ${p_{10}}_*\Delta(\sigma^2_{10})=2\sigma^2_9$
by  \eqref{DelS2}, we have 
\begin{equation}\label{Dsgm210}
\Delta(\sigma^2_{10}) - 2[\sigma^2_9]\in{i_{10}}_*R^9_{23}
= \{[P(E\theta)+\nu_6\kappa_9]_{10}, [\iota_7]_{10}\mu_7\sigma_{16}\}.
\end{equation}

We show the following.
\begin{lem} \label{Derh}
$8\Delta(\sigma^2_{10})=16[\sigma^2_{10}] = [\iota_7]_{10}\mu_7\sigma_{16}\neq 0$ and $[\zeta_5]_{10}\sigma_{16}=0$.
\end{lem}
\begin{proof}
By Lemmas \ref{m3sg}(3) and \ref{rel2}(4), we have 
$$
16[\sigma^2_9]\in -[\iota_7]_{10}\circ\{\bar{\nu}_7+\varepsilon_7,\sigma_{15},16\iota_{22}\}\ni [\iota_7]_{10}\mu_7\sigma_{16}+[\iota_7]_{10}E\zeta'=[\iota_7]_{10}\mu_7\sigma_{16}
$$
$$
\bmod\ [\iota_7]_{10}(\bar{\nu}_7+\varepsilon_7)\circ\pi^{15}_{23} 
+ 16[\iota_7]_{10}\pi^7_{23} = 0.
$$
We will examine the image of a map $J\Delta:\pi_{24}^9\to\pi_{32}^9$, where 
$\pi_{24}^9=\{\rho',\sigma_9\bar{\nu}_{16},\sigma_9\varepsilon_{16},\bar{\varepsilon}_9\}$ \cite[Theorem 7.1]{T}.
By the relation \eqref{JD} and the fact that 
$2\rho_{13} = E^4\rho'$ \cite[Lemma 10.9]{T} and
$2[\iota_9, \iota_9]=0$, we have
$J\Delta(\rho')=[\iota_9, \rho'] = [\iota_9, \iota_9]\circ E^8\rho' = 0$.
By the relations \eqref{Desig9e} and \eqref{Jbar}, we have 
$J\Delta(\sigma_9\bar{\nu}_{16})=J\Delta(\sigma_9\varepsilon_{16})=J([\nu_5]_9\bar{\varepsilon}_8)
%=(-1)^3E^3(J[\nu_5]\bar{\varepsilon}_8)
=0$.
By \eqref{k4} and Lemma \ref{rel2} (2), we have 
\[
\Delta\bar{\varepsilon}_9 = [\nu_5]_9\bar{\varepsilon}_8
+ [\iota_7]_9\eta_7\bar{\varepsilon}_8 
= 2[\nu_5]_9\bar{\varepsilon}_8 = 0.
\]
%By (\ref{Jbar}), $J([\nu_5]_9\bar{\varepsilon}_8)=0.$
Hence we have $J\Delta\pi^9_{24} = 0$. 
By \eqref{Ji7} %$J[\iota_7]=\sigma_8$
 and \cite[Theorem 1.1(a)]{MMO}, we have
\begin{equation}\label{s_9smu}
J([\iota_7]_9\mu_7\sigma_{16}) = \sigma_9\mu_{16}\sigma_{25}
= \sigma^2_9\mu_{23}\neq 0.
\end{equation}
This implies $[\iota_7]_9\mu_7\sigma_{16}\not\in\Delta\pi^9_{24}$. 

By \eqref{Dsgm210}, Proposition \ref{R923} and Lemma \ref{rel2} (4), 
we have $2(\Delta(\sigma^2_{10})-2[\sigma^2_9])=0$. This leads to the first 
assertion. 

By \cite[(4.7), p.~28]{Ka}, we know
\begin{equation}\label{kazt10}
[\zeta_5]_{10}\equiv [\iota_7]_{10}\mu_7 + 8\Delta\sigma_{10}
\bmod [\iota_7]_{10}\nu^3_7.
\end{equation}
So, we have 
$[\zeta_5]_{10}\sigma_{16} = 
[\iota_7]_{10}\mu_7\sigma_{16} + 8\Delta(\sigma^2_{10})$.
This and the first lead to the second assertion 
$[\zeta_5]_{10}\sigma_{16}=0$. This completes the proof. 
\end{proof}

Next, we show 
\begin{equation}\label{8sgm8et}
%[8\sigma_8]_{10}\eta_{15}\sigma_{16}=8\Delta(\sigma^2_{10}).
[8\sigma_8]_{10}\eta_{15} \equiv  8\Delta(\sigma_{10}) 
\bmod [\iota_7]_{10}\nu^3_7.
\end{equation}
\begin{proof}
By the relation $HJ[8\sigma_8]=8\sigma_{17}=H(\rho')$ from \cite[(10.12)]{T},
we have $J[8\sigma_8]\equiv \rho' \bmod E\pi^8_{23}$.
We know $E\pi^8_{23}=\{E^2\rho'', \sigma_9\bar\nu_{16}, 
\sigma_9\varepsilon_{16}, \bar\varepsilon_9\}$ and 
$E^2\rho''=2\rho'$ \cite[p.~107]{T}. Hence, we have
\[
J[8\sigma_8] \equiv \rho' 
\bmod 2\rho', \sigma_9\bar\nu_{16}, 
\sigma_9\varepsilon_{16}, \bar\varepsilon_9.
\]
By composing $\eta_{24}$ to this relation on the right and using
Lemma \ref{zesig} (3), we have
\[
J([8\sigma_8]\eta_{15}) \equiv \rho' \eta_{24}
\bmod \sigma_9\bar\nu_{16}\eta_{24}, \sigma_9\varepsilon_{16}\eta_{24}.
\]
By \eqref{t3},  
$\sigma_{10}\nu^3_{17}
=\sigma_{10}\eta_{17}\varepsilon_{18}=0$ and
$E(\rho' \eta_{24})=8[\iota_{10},\sigma_{10}]$ \cite[Proposition 2.8(5), (6)]{Og}, 
we have
\[
J([8\sigma_8]_{10}\eta_{15}) = E(\rho' \eta_{24})=8P(\sigma_{21})
=J(8\Delta(\sigma_{10})).
\]
We know 
$\pi^{10}_{26} = \{P(\sigma_{21}), \sigma_{10}\mu_{17}\}
 \cong \Z_{16} \oplus \Z_2$
\cite[Theorem 12.16]{T}.
So, by \eqref{R1016} and \eqref{Ji7},
$J: R^{10}_{16} \to \pi^{10}_{26}$
is a split epimorphism and $\Ker J =\{[\iota_7]_{10}\nu^3_{7}\}\cong\Z_2$.
This implies the desired relation. 
\end{proof}

By using \cite[Lemma 5.1]{MiM} and  \eqref{Ji7},
we have
$$ 
J[\sigma^2_9]\in
J\{[\iota_7]_{10},\bar{\nu}_7+\varepsilon_7,\sigma_{15}\}
\subset
\{J[\iota_7]_{10},\bar{\nu}_{17}+\varepsilon_{17},\sigma_{25}\}_{10}
=\{\sigma_{10},\bar{\nu}_{17}+\varepsilon_{17},\sigma_{25}\}_{10}.
$$
By \cite[(4.27)]{MMO} and its proof, $\psi_{10}$ is taken as
$$
\psi_{10}\in\{\sigma_{10}, \bar{\nu}_{17}+\varepsilon_{17}, \sigma_{25}\}_4 .
$$
So we can set 
\begin{equation}\label{psi10}
J[\sigma^2_9] = \psi_{10}.
\end{equation}

By \cite[Theorem 0.1(ii)-(iii), pp. 15-16]{HKM}, we have 
\begin{equation}\label{HKmtp}
R^9_{21} = \Z_4\{[\iota_7]_9\kappa_7\}, \ 
R^{10}_{21} = \Z_2\{[\iota_7]_{10}\kappa_7\}\ \mbox{and} \ 
R^{11}_{21} = \{[\iota_7]_{11}\kappa_7, [\eta^2_{10}]\mu_{12}\}\cong(\Z_2)^2.
\end{equation}

Since $\bar{\nu}_9\nu^2_{17} = [\iota_9, \nu^2_9]$ in \cite[p.~102]{T} and
$\nu^2_{17} = \{\eta_{17},\nu_{18}, \eta_{21}\}_1$ \cite[Lemma 5.12, its proof]{T}, we have
$\bar{\nu}_9\nu^2_{17}\in [\iota_9,\iota_9]\circ\{\eta_{17}, \nu_{18}, \eta_{21}\}_1$.
By \eqref{HKmtp}, we have
$[[\iota_9,\eta_9]]\circ\nu_{18} \in R^{10}_{21}=  \Z_2\{[\iota_7]_{10}\kappa_7\}$.
Since $J([[\iota_9,\eta_9]]\circ\nu_{18})=\zeta_{10}\sigma_{19}\nu_{26}=0$ \cite[p.~343]{KM1}
and $J([\iota_7]_{10}\kappa_7)=\sigma_{10}\kappa_{17}\neq 0$ \cite[Theorem A]{Mi1},
we obtain $[[\iota_9,\eta_9]]\circ\nu_{18}=0$. Hence, the Toda bracket 
$\{[[\iota_9,\eta_9]], \nu_{18}, \eta_{21}\}_1$ is well-defined.
We set
\begin{equation}\label{[bn9n^2]}
[\bar{\nu}_9\nu^2_{17}]\in\{[[\iota_9,\eta_9]], \nu_{18}, \eta_{21}\}_1.
\end{equation}

By \cite[Proposition 2.4(2)]{Og} and \cite[(10.7)]{T}, we also know 
$$
\nu_6\zeta_9=\zeta_6\nu_{17}=2(\sigma''\sigma_{13}).
$$
So, by \eqref{2s''}, we have $\nu_7\zeta_{18}=4(\sigma'\sigma_{14})$, 
$\nu_9\zeta_{12}=8\sigma^2_{9}$.
Hence, by the fact that $\sharp(\sigma^2_{14})=8$ \cite[Theorem 10.3]{T}, we have
\begin{equation}\label{n9z}
\nu_7\zeta_{18}=\zeta_7\nu_{18}=4(\sigma'\sigma_{14}),\ 
\nu_9\zeta_{12}=\zeta_{9}\nu_{20}=8\sigma^2_{9} \quad\text{and}\quad \nu_{14}\zeta_{17}=0.
\end{equation}

We show 
\begin{lem}\label{4om14}
$\{\zeta_{14}, \nu_{25}, \eta_{28}\}_1 \equiv 4\omega_{14} \bmod
 \sigma_{14}\mu_{21}$ and\\
$\{\zeta_{n+14}, \nu_{n+25}, \eta_{n+28}\}_1 \equiv 0 \bmod
 \sigma_{n+14}\mu_{n+21}$ for $n\geq 1$. 
\end{lem}
\begin{proof}
Since $\zeta_{13}\nu_{24}=8\sigma^2_{13}=P(\eta^2_{27})$ (\eqref{n9z}, \cite[(10.10)]{T}), we obtain $H\{\zeta_{14}, \nu_{25}, \eta_{28}\}_1
 = -P^{-1}(8\sigma^2_{13}) \circ
 \eta_{29}=\{\eta^3_{27}\}=\{4\nu_{27}\}$. Together with the fact
 $H(\omega_{14})=\nu_{27}$ and \cite[Proposition 2.6]{T}, it implies
 that $\{\zeta_{14}, \nu_{25}, \eta_{28}\}_1 \equiv 4\omega_{14} \bmod
 E\pi^{13}_{29}=\{\sigma_{14}\mu_{21}\}$. 
Since $4\omega_{15}=0$ from \cite[Theorem 12.16]{T}, we have the second relation.
 This completes the proof. 
\end{proof}

Now we show the following.
\begin{prop} \label{R10}
$R^{10}_{23} = \Z_{32}\{[\sigma^2_9]\}\oplus\Z_2\{[\bar{\nu}_9\nu^2_{17}]\}
\oplus\Z_2\{[P(E\theta)+\nu_6\kappa_9]_{10}\}$
with relations %$[\nu_5]_{10}\bar{\varepsilon}_8=0$,
$\Delta\rho'\equiv [\zeta_5]_9\sigma_{16} \bmod
[\nu_5]_9\bar{\varepsilon}_8$
and $16[\sigma^2_9] = 8\Delta(\sigma^2_{10})
= [\iota_7]_{10}\mu_7\sigma_{16}$.
\end{prop}
\begin{proof}
The order of $[\sigma^2_9]$ and the last relation follow from the first relation of Lemma \ref{Derh}. We know $2[P(E\theta)+\nu_6\kappa_9]_{10}
=[\nu_5]_{10}\bar{\varepsilon}_8=0$ by Lemma \ref{rel2}. 

By the second of Lemma \ref{Derh}, we know $[\zeta_5]_{10}\sigma_{16} = 0$. So, we have $[\zeta_5]_9\sigma_{16}\in\Delta\pi_{24}^9$. 
By the fact that $\pi_{24}^9=\{\rho',\sigma_9\bar{\nu}_{16},\sigma_9\varepsilon_{16},\bar{\varepsilon}_9\}$,
\eqref{Desig9e} and  $\Delta\bar{\varepsilon}_9 =0$,
we have $\Delta\pi^9_{24}=\{\Delta\rho',[\nu_5]_9\bar{\varepsilon}_8\}$. 
This implies the first relation. 

By \eqref{[bn9n^2]}, we have
\[\begin{split}
2[\bar{\nu}_9\nu^2_{17}]&\in\{[[\iota_9,\eta_9]], \nu_{18}, \eta_{21}\}
\circ 2\iota_{23} = - [[\iota_9,\eta_9]]\circ\{\nu_{18}, \eta_{21}, 
2\iota_{22}\}\\
&\subset[[\iota_9,\eta_9]]\circ\pi^{18}_{23} = 0.
\end{split}\]
So the order of $[\bar{\nu}_9\nu^2_{17}]$ is $2$. This implies the group
$R^{10}_{23}$.
This completes the proof.
\end{proof}

By \cite[p. 22]{MMO}, we know
\begin{equation}\label{sigsigmu}
P(\rho_{19})=\sigma^2_9\mu_{23}.
\end{equation}

Next we show the following.

\begin{prop} \label{R11}
$\Delta(\sigma^2_{10}) = 2x[\sigma^2_9]$ for odd $x$, $\Delta\kappa_{10}
\equiv [\bar{\nu}_9\nu^2_{17}] + [P(E\theta)+\nu_6\kappa_9]_{10}
\bmod\ 16[\sigma^2_9]$
and $R^{11}_{23} = \Z_2\{[P(E\theta)+\nu_6\kappa_9]_{11}\}
\oplus\Z_2\{[\sigma^2_9]_{11}\}$.
\end{prop}
\begin{proof}
In the exact sequence $(23)_{10}$:
$$
\pi^{10}_{24}\rarrow{\Delta}R^{10}_{23}\rarrow{i_*}R^{11}_{23}
\rarrow{p_*}\pi^{10}_{23}\rarrow{\Delta}R^{10}_{22},
$$
$\Delta: \pi^{10}_{23}\to R^{10}_{22}$ is a monomorphism and $\pi^{10}_{24}
= \Z_{16}\{\sigma^2_{10}\}\oplus\Z_2\{\kappa_{10}\}$.

We recall $\Delta(\sigma^2_{10}) - 2[\sigma^2_9]\in{i_{10}}_*R^9_{23}
= \{[P(E\theta)+\nu_6\kappa_9]_{10}, [\iota_7]_{10}\mu_7\sigma_{16}\}$
\eqref{Dsgm210}.
We know $J\Delta(\sigma^2_{10}) = [\iota_{10}, \sigma^2_{10}] = 2\psi_{10}$
\cite[(4.27)]{MMO}.
By Lemma \ref{JP} and \cite[(4.16)]{MMO}, we know
\begin{equation}\label{JPeth10}
J[P(E\theta)+\nu_6\kappa_9]_{10}
=\bar{\kappa}_{10}\nu_{30} - \nu_{10}\bar{\kappa}_{13}
=P(\kappa_{21})=J\Delta\kappa_{10}.
\end{equation}
Hence we take $\Delta(\sigma^2_{10})\equiv 2[\sigma^2_9] \bmod
[\iota_7]_{10}\mu_7\sigma_{16} = 16[\sigma^2_9]$ (Lemma \ref{Derh}). 
This leads to the first relation.

Since ${p_{10}}_*\Delta\kappa_{10}=2\kappa_9=\bar{\nu}_9\nu^2_{17}$ 
by \eqref{DelS2}, we have 
\begin{equation}\label{Dk10}
\Delta\kappa_{10}-[\bar{\nu}_9\nu^2_{17}]\in{i_{10}}_*R^9_{23}
=\{[P(E\theta)+\nu_6\kappa_9]_{10}, [\iota_7]_{10}\mu_7\sigma_{16}\}.
\end{equation} 
By relations \eqref{[bn9n^2]},
$J[[\iota_9,\eta_9]] = \pm \sigma_{10}\zeta_{17}$ \cite[p.~343]{KM1} and
$\zeta_{17}\nu_{28}=\nu_{17}\zeta_{20}=0$ from \eqref{n9z}, we have
\[\begin{split}
J[\bar{\nu}_9\nu^2_{17}]&\in J\{[[\iota_9,\eta_9]], \nu_{18}, \eta_{21}\}
\subset\{J[[\iota_9,\eta_9]], \nu_{28}, \eta_{31}\}\\
&=\{\pm\sigma_{10}\zeta_{17}, \nu_{28}, \eta_{31}\}
\supset\pm\sigma_{10}\circ\{\zeta_{17}, \nu_{28}, \eta_{31}\}.
\end{split}\]
By Lemma \ref{4om14}, we have
\[
\sigma_{10}\circ\{\zeta_{17}, \nu_{28}, \eta_{31}\}\ni 0
 \bmod \sigma^2_{10}\mu_{24} = 0.
\]
Hence, by the fact that
$\pi^{28}_{33}=0$ \cite[Proposition 5.9]{T} and 
$\pi^{10}_{32}=\{\sigma_{10}\rho_{17},\varepsilon_{10}\kappa_{18},\nu_{10}\bar\sigma_{13}\}$
\cite[Theorem B]{Mi1}, we have
\[
J[\bar{\nu}_9\nu^2_{17}]\equiv 0
\bmod\sigma_{10}\zeta_{17}\circ\pi^{28}_{33}+\pi^{10}_{32}\circ\eta_{32}
=\{\sigma_{10}\rho_{17}\eta_{32}, \varepsilon_{10}\kappa_{18}\eta_{32},
 \nu_{10}\bar\sigma_{13}\eta_{32}\}.
\] 
We know $\sigma_{10}\rho_{17}\eta_{32}=\sigma^2_{10}\mu_{24} = 0$, 
$\varepsilon_{10}\kappa_{18}\eta_{32}=4\nu_{10}\bar\kappa_{13}$ by \eqref{ep5bep} and
$\nu_{10}\bar\sigma_{13}\eta_{32}=\nu_{10}\eta_{13}\bar\sigma_{14}=0$ by \cite[(2.7)]{MMO}.
Since $4\nu_{10}\bar\kappa_{13}$ is not J-image,
we have $J[\bar{\nu}_9\nu^2_{17}]=0$.
Hence, applying $J$ to \eqref{Dk10}, we have
\[
J\Delta\kappa_{10}\equiv 0
\ \bmod\ J[P(E\theta)+\nu_6\kappa_9]_{10}, 
J([\iota_7]_{10}\mu_7\sigma_{16}).
\]
Furthermore, by \eqref{JPeth10}, 
$J([\iota_7]_{10}\mu_7\sigma_{16})=\sigma^2_{10}\mu_{24}=0$ from \eqref{s_9smu},
we have
\[
\Delta\kappa_{10} \equiv [\bar{\nu}_9\nu^2_{17}] +[P(E\theta)+\nu_6\kappa_9]_{10}
\ \bmod\ [\iota_7]_{10}\mu_7\sigma_{16}.
\]
Thus, by Proposition \ref{R10}, we obtain the second relation and the group $R^{11}_{23}$. 
This completes the proof.
\end{proof}

We show the following.

\begin{lem} \label{fin}
$[\eta_9\varepsilon_{10}]\nu_{18} = 0$\ and\ 
$[\bar{\nu}_9\nu^2_{17}]_{11} = [P(E\theta)+\nu_6\kappa_9]_{11}
=[[\iota_{11},\nu_{10}]]\eta_{22}$.
\end{lem}
\begin{proof}
Since $[\eta_9\varepsilon_{10}]_{11}\in R^{11}_{18} 
= \Z_8\{[\varepsilon_{10}]\}$ \cite[Theorem 2(iii)]{KM1}, we have 
$[\eta_9\varepsilon_{10}]_{11}= 2y[\varepsilon_{10}]$ for $y\in\{0, 1, 2, 3\}$. 
Assume that $[\eta_9\varepsilon_{10}]\nu_{18}\neq 0$. Then, by \eqref{HKmtp}, we have $[\eta_9\varepsilon_{10}]\nu_{18}=[\iota_7]_{10}\kappa_7$. So, by the group structure of $R^{11}_{21}$, we have 
$[\eta_9\varepsilon_{10}]_{11}\nu_{18}=[\iota_7]_{11}\kappa_7\neq 0$.
On the other hand, we have $[\eta_9\varepsilon_{10}]_{11}\nu_{18}
=2y[\varepsilon_{10}]\nu_{18}\in 2R^{11}_{21}=0$. 
This implies a contradiction and hence 
we obtain  $[\eta_9\varepsilon_{10}]\nu_{18} = 0$.

By Proposition \ref{R11}, we have 
$[\bar{\nu}_9\nu^2_{17}]_{11}=[P(E\theta)+\nu_6\kappa_9]_{11}$.
By the relation  
$\Delta[\iota_{10},\iota_{10}]
=[[\iota_9,\iota_9]\eta_{17}] + 2[\eta_9\varepsilon_{10}]$
\cite[p.~343]{KM1}, we have
$$
[[\iota_9,\eta_9]]_{11} = -2[\eta_9\varepsilon_{10}]_{11}.
$$
So, by the relation \eqref{[bn9n^2]} and 
the fact that $2R^{11}_{23} = 0$ (Proposition \ref{R11}), 
we see that
\[
[\bar{\nu}_9\nu^2_{17}]_{11}
\in\{-2[\eta_9\varepsilon_{10}]_{11}, \nu_{18},\eta_{21}\}
\supset
-2\{[\eta_9\varepsilon_{10}]_{11}, \nu_{18},\eta_{21}\}
= 0
\]
with the indeterminacy 
$-2[\eta_9\varepsilon_{10}]_{11}\circ\pi^{18}_{23}
 +R^{11}_{22}\circ\eta_{22}=R^{11}_{22}\circ\eta_{22}$.
We recall that $R^{11}_{22} = \{[[\iota_{10},\nu_{10}]],
[8\sigma_8]_{11}\sigma_{15}\}$ \cite[Theorem 0.2]{HKM}. 
Then, by the relation $[8\sigma_8]_{11}\sigma_{15}\eta_{22}=0$ 
\eqref{8sgm8et}, we obtain $R^{11}_{22}\circ\eta_{22}
= \{[[\iota_{10},\nu_{10}]]\eta_{22}\}$. This implies $[\bar{\nu}_9\nu^2_{17}]_{11}=[[\iota_{10},\nu_{10}]]\eta_{22}$ and 
completes the proof.
\end{proof}

By relations \eqref{psi10} and 
$\psi_n\neq 0$ for $10\leq n\leq 21$ \cite[Theorem 1.1(a)]{MMO}, 
we have 
$$
\sharp[\sigma^2_9]_n=2\ (11\leq n\leq 21).
$$

By the relation ${p_{12}}_\ast\Delta\theta = \theta'$ by \cite[Lemma 6.2(ii)]{HKM},
we take $[\theta']=\Delta\theta\in R^{12}_{23}$.
We show the following.

\begin{prop}\label{R12}
$R^{12}_{23} = \Z_2\{[\theta']\}
\oplus\Z_2\{[\bar{\nu}_9\nu^2_{17}]_{12}\}
\oplus\Z_2\{[\sigma^2_9]_{12}\}$, where $[\theta']=\Delta\theta$.
\end{prop}
\begin{proof}
In the exact sequence $(23)_{11}$:
$$
\pi^{11}_{24}\rarrow{\Delta}R^{11}_{23}\rarrow{i_*}R^{12}_{23}
\rarrow{p_*}\pi^{11}_{23}\rarrow{\Delta} R^{11}_{22},
$$
we know $\pi^{11}_{23}=\Z_2\{\theta'\}$ and 
$\pi^{11}_{24}= \Z_2\{\theta'\eta_{23}\}\oplus\Z_2\{\sigma_{11}\nu^2_{18}\}$
\cite[Theorems 7.6 and 7.7]{T}.
We recall the relation
 $\eta_{11}\theta=\sigma_{11}\nu^2_{18}+\theta'\eta_{23}$ by
 \cite[Proposition 2.2(8)]{Og}. So, we can take
 $\pi^{11}_{24}=\{\theta'\eta_{23}, \eta_{11}\theta\}$. 
We have $2[\theta']=\Delta(2\theta)= 0$ by \cite[Lemma 7.5]{T}. 
We have  $\Delta(\theta'\eta_{23}) =\Delta\theta'\circ\eta_{22}= 0$ and 
$\Delta(\eta_{11}\theta)=0$ by \eqref{Delta} and \eqref{Ke1}.
So Proposition \ref{R11} implies the conclusion.
\end{proof}

We know Mahowald's result $\sharp\Delta([\iota_{12},\iota_{12}])=11!/8$ 
\cite[Theorem 1.3]{HKM}. In the $2$-primary components, we have 
$\sharp\Delta([\iota_{12},\iota_{12}])=32$. 
We show the following.

\begin{prop} \label{R13}
$R^{13}_{23} = \Z\{[32[\iota_{12}, \iota_{12}]]\}
\oplus\Z_2\{[\sigma^2_9]_{13}\}$, where \\
$\Delta(E\theta') = [\bar{\nu}_9\nu^2_{17}]_{12}
= [P(E\theta)+\nu_6\kappa_9]_{12}
= [[\iota_{10},\nu_{10}]]_{12}\eta_{22}$.
\end{prop}
\begin{proof}
In the exact sequence $(23)_{12}$:
$$
\pi^{12}_{24}\rarrow{\Delta}R^{12}_{23}\rarrow{i_*}R^{13}_{23}
\rarrow{p_*}\pi^{12}_{23}\rarrow{\Delta}R^{12}_{22},
$$
$\Ker\ \{\Delta: \pi^{12}_{23}\to R^{12}_{22}\} = \{32[\iota_{12}, 
\iota_{12}]\}$. %\cite[Lemma 63.(ii), Proposition 6.6(i)]{HKM}. 
Notice that\ $\sharp P(\zeta_{25})=8$ \cite[p. 320]{Mi1} implies $\sharp\Delta(\zeta_{12})=8$. 
We know $\pi^{12}_{24} = \Z_2\{\theta\}\oplus\Z_2\{E\theta'\}$ \cite[Theorem 7.6]{T}. $\Delta\theta = [\theta']$ by Proposition \ref{R12} and $J[\sigma^2_9]_{13} = \psi_{13}$ by \eqref{psi10}.
By relations $E\theta' = [\iota_{12},\iota_{12}]\circ\eta_{23}$ \cite[(7.30)]{T},
\cite[Lemma 6.3(ii)]{HKM} and \eqref{8sgm8et}, we have
$$
\Delta(E\theta') = \Delta([\iota_{12}, \iota_{12}])\circ\eta_{22}
\equiv [[\iota_{10},\nu_{10}]]_{12}\eta_{22} \bmod
[8\sigma_8]_{12}\sigma_{15}\eta_{22} = 0.
$$
So, by Lemma \ref{fin}, we obtain
$$
\Delta(E\theta') = [[\iota_{10},\nu_{10}]]_{12}\eta_{22}
= [P(E\theta)+\nu_6\kappa_9]_{12} = [\bar{\nu}_9\nu^2_{17}]_{12}.
$$
This implies the group $R^{13}_{23}$, completing the proof.
\end{proof}

In the exact sequence $(23)_{13}$:
$$
\pi^{13}_{24}\rarrow{\Delta}R^{13}_{23}\rarrow{i_*}R^{14}_{23}
\rarrow{p_*}\pi^{13}_{23}\rarrow{\Delta}R^{13}_{22},
$$
$\Delta: \pi^{13}_{23}\to R^{13}_{22}$ is a monomorphism \cite[Lemma 6.7]{HKM},
$\pi^{13}_{24}=\{\zeta_{13}\}$ \cite[Theorem 7.4]{T}
and
$\Delta\zeta_{13} = 0$ by Lemma \ref{zeta1}. So we obtain the following.

\begin{prop} \label{R14}
$R^{14}_{23} = \Z\{[32[\iota_{12}, \iota_{12}]]_{14}\}
\oplus\Z_2\{[\sigma^2_9]_{14}\}$.
\end{prop}

By the facts $R^\infty_{23}\cong \Z$ \cite{Bo} and \cite[Table 1]{HoM},
we have Table \ref{R23}.
{\begin{table}[H]
\centering
\begin{tabular}{|c|c||c|c||c|c|}\hline
$R^{13}_{23}$ & $\Z \oplus \Z_2$ 
 & $R^{14}_{23}$ & $\Z \oplus \Z_2$
 & $R^{15}_{23}$ & $\Z \oplus (\Z_2)^3$  \\ \hline
$R^{16}_{23}$ & $\Z \oplus (\Z_2)^5$ 
 & $R^{17}_{23}$ & $\Z \oplus (\Z_2)^3$
 & $R^{18}_{23}$ & $\Z \oplus (\Z_2)^2$\\ \hline
\end{tabular}
\caption{$R^{n}_{23}$ for $13\leq n\leq 18$}\label{R23}
\end{table}
On the other hand, we have Table \ref{R23K} from \cite{Ke}.
\begin{table}[H]
\centering
\begin{tabular}{|c|c||c|c||c|c|}\hline
$R^{19}_{23}$ & $\Z \oplus \Z_2$ 
 & $R^{20}_{23}$ & $\Z \oplus \Z_2$
 & $R^{21}_{23}$ & $\Z \oplus \Z_2$  \\ \hline
$R^{22}_{23}$ & $\Z$ 
 & $R^{23}_{23}$ & $\Z$
 & $R^{24}_{23}$ & $\Z\oplus\Z$\\ \hline
 \end{tabular}
\caption{$R^{n}_{23}$ for $19\le n\le 24$}\label{R23K}
\end{table}}

By Tables \ref{R23} and \ref{R23K}, we have the group structure of $R^n_{23}$ for $n \ge 13$.
So, hereafter our main work is to research the generators and the relations in 
$R^n_{23}$ for $n \ge 13$.

We define 
\[
 [\eta_{14}\varepsilon_{15}] \in \{[\eta^2_{14}], 2\iota_{16}, \nu^2_{16}\}_{11}.
\]
By  Lemma \ref{ord} and Example \ref{shetep}, $\sharp[\eta_{14}\varepsilon_{15}]=2$. 
We show the following.
\begin{prop} \label{R15}
%\begin{enumerate}\item 
$R^{15}_{23} = \Z\{[32[\iota_{12}, \iota_{12}]]_{15}\}
\oplus\Z_2\{[\sigma^2_9]_{15}\}\oplus\Z_2\{[\eta_{14}\varepsilon_{15}]\}
\oplus\Z_2\{[\eta_{14}\sigma_{15}]\eta_{22}\}$.
\end{prop}
\begin{proof}
In the exact sequence $(23)_{14}$:
$$
\pi^{14}_{24}\rarrow{\Delta}R^{14}_{23}\rarrow{i_*}R^{15}_{23}
\rarrow{p_*}\pi^{14}_{23}\rarrow{\Delta}R^{14}_{22},
$$
we know 
$\pi^{14}_{24}=\Z_2\{\eta_{14}\mu_{15}\}$ \cite[Theorem 7.3]{T}, 
$\Ker\ \{\Delta: \pi^{14}_{23}\to R^{14}_{22}\}
= \{\eta_{14}\varepsilon_{15}, \eta^2_{14}\sigma_{16}\}$ and
$\Delta(\eta_{14}\mu_{15}) =\Delta\mu_{14}\circ \eta_{22}
= 8[\eta^2_{13}\sigma_{15}]\eta_{22}=0$ by \cite[the proof of Lemma 6.8]{HKM}. 
%by Lemma \ref{etab}. 
From the fact that
 $\Delta(\bar\nu_{14})=\Delta(\varepsilon_{14})$~\cite[p.17]{HKM}, We
 know that the element
$[\eta_{14}\sigma_{15}]\in R^{15}_{22}$. By Example \ref{shetep}(1),
we have $2[\eta_{14}\varepsilon_{15}] = 0$. 
This completes the proof.
\end{proof}

We recall \cite[(6.1)]{T}:
$$
\varepsilon_{15}\in\{\eta_{15},2\iota_{16},\nu^2_{16}\}_1 
= \{\eta_{15},2\nu_{16},\nu_{19}\}_1.
$$

By the fact that  $\sharp [\eta_{15}]=2$ \eqref{Ke1},
we can define 
%We set 
$$
[\varepsilon_{15}]\in\{[\eta_{15}],2\iota_{16},\nu^2_{16}\}_1.
$$
Since $R^{16}_{20} = 0$ \cite[p. 161, Table]{Ke}, we have $R^{16}_{20}\circ\nu^2_{20}=R^{16}_{17}\circ\nu^2_{17}=0$. Hence, we have
$$
\{[\eta_{15}],2\iota_{16},\nu^2_{16}\}_1 = \{[\eta_{15}],2\nu_{16},\nu_{19}\}_1.
$$

We show the following.

\begin{prop} \label{R16}
%\begin{enumerate}\item 
$R^{16}_{23} = \Z\{[32[\iota_{12}, \iota_{12}]]_{16}\}
\oplus\Z_2\{[\sigma^2_9]_{16}\}\oplus\Z_2\{[\eta_{14}\varepsilon_{15}]_{16}\}\\
\oplus\Z_2\{[\eta_{14}\sigma_{15}]_{16}\eta_{22}\}
\oplus\Z_2\{[\eta_{15}]\sigma_{16}\}\oplus\Z_2\{[\varepsilon_{15}]\}$.
\end{prop}
\begin{proof}
We have
$$
2[\varepsilon_{15}]\in
-\{2\iota_{R_{16}},[\eta_{15}], 2\iota_{16}\}\circ\nu^2_{17}
\subset R^{16}_{17}\circ\nu^2_{17}=0.
$$
We know $\pi^{15}_{24} = \{\eta_{15}\varepsilon_{16}, \eta_{15}\bar{\nu}_{16},
\mu_{15}\}$ by \eqref{t3} and \cite[Theorem 7.2]{T}. 
By \eqref{Delta} and \eqref{Ke1}, %Lemma \ref{rel2},
$$
\Delta(\eta_{15}\varepsilon_{16})=0 \ \mbox{and} \ 
\Delta(\eta_{15}\bar{\nu}_{16})=0.
$$
 
By the fact that $R^{15}_{16}\cong(\Z_2)^2$ \cite[p. 161, Table]{Ke}, we have
$$
\Delta(\mu_{15})\in\Delta\{\eta_{15},2\iota_{16},8\sigma_{16}\}_1
\subset\{0,2\iota_{15},8\sigma_{15}\}=R^{15}_{16}\circ 8\sigma_{16}=0.
$$
That is, $\Delta(\mu_{15})=0$.
So, by use of the exact sequence $(23)_{15}$, we have a split exact sequence:
$$
0\rarrow{}R^{15}_{23}
\rarrow{i_*}R^{16}_{23}\rarrow{p_*}\pi^{15}_{23}
\rarrow{}0.
$$
This completes the proof.
\end{proof}

We recall that $J[\eta^2_{14}]\equiv\eta^{*\prime}\ \bmod\ \mu_{15}\sigma_{24}$ \cite[Lemma 2.2(i)]{HKM} and $J[\eta_{14}\sigma_{15}]
\equiv x\sigma^{*\prime}\ \bmod\ \omega_{15}\nu^2_{31}$ for $x$ odd \cite[p. 26]{HKM}. 
Since $\eta^2_{14}\sigma_{16}=(\eta_{14}\sigma_{15})\eta_{22}$, we have
$[\eta_{14}\sigma_{15}]\eta_{22} \equiv [\eta^2_{14}]\sigma_{16}\
\bmod {i_{15}}_*R^{14}_{23}$. So, we can set 
$[\eta_{14}\sigma_{15}]\eta_{22} = [\eta^2_{14}]\sigma_{16} + y[\sigma^2_9]_{15}$\ for\ $y\in\{0, 1\}$. 
By Taking the $J$-image of this relation, we have 
$\sigma^{*\prime}\eta_{37} = \eta^{*\prime}\sigma_{31} + y\psi_{15}$.

By the relation $\pm P(\sigma_{33})\equiv 2\sigma^*_{16} - E\sigma^{*\prime}
\bmod\ \rho_{16}\sigma_{31}$, we have
$(E\sigma^{*\prime})\eta_{38}\equiv P(\sigma_{33}\eta_{40}) \bmod\  \rho_{16}\sigma_{31}\eta_{38}$.
Since $\rho_{16}\sigma_{31}\eta_{38}=\rho_{16}\eta_{31}\sigma_{32}
=\mu_{16}\sigma^2_{25}=0$, we obtain 
$(E\sigma^{*\prime})\eta_{38} = P(\sigma_{33}\eta_{40})$.
On the other hand, we know 
$
P(\eta^{*\prime}\sigma_{31} + y\psi_{15}) 
= P(\eta_{33}\sigma_{34}) + y\psi_{16}.
$
by \cite[Lemma 2.10]{Og} and \eqref{sigsigmu}.
This implies $y=0$, and hence we obtain
\begin{equation}\label{sgm*'et}
[\eta_{14}\sigma_{15}]\eta_{22} = [\eta^2_{14}]\sigma_{16}\ \mbox{and} \
\sigma^{*\prime}\eta_{37} = \eta^{*\prime}\sigma_{31}.
\end{equation}

We show the following.
\begin{prop} \label{R17}
$R^{17}_{23} = \Z\{[32[\iota_{12}, \iota_{12}]]_{17}\}
\oplus\Z_2\{[\sigma^2_9]_{17}\}\oplus\Z_2\{[\eta_{15}]_{17}\sigma_{16}\}
\oplus\Z_2\{[\varepsilon_{15}]_{17}\}$, where $\Delta(\eta_{16}\sigma_{17})=[\eta_{14}\sigma_{15}]_{16}\eta_{22}$ and\  
$\Delta\varepsilon_{16}= [\eta_{14}\varepsilon_{15}]_{16}$.
\end{prop}
\begin{proof}
In the exact sequence $(23)_{16}$:
$$
\pi^{16}_{24}\rarrow{\Delta}R^{16}_{23}\rarrow{i_*}R^{17}_{23}
\rarrow{p_*}\pi^{16}_{23}\rarrow{\Delta}R^{16}_{22},
$$
The second $\Delta$ is a monomorphism \cite[the proof of Lemma 6.4]{T} and 
$\pi^{16}_{24} = \Z_2\{\eta_{16}\sigma_{17}\}
\oplus\Z_2\{\varepsilon_{16}\}$ \cite[Theorem 7.1,  Lemma 6.4]{T}. 
By \eqref{sgm*'et}, we have
 $\Delta(\eta_{16}\sigma_{17})
= [\eta^2_{14}]_{16}\sigma_{16}= [\eta_{14}\sigma_{15}]_{16}\eta_{22}$. 
We have
\[
\Delta\varepsilon_{16}\in\Delta\{\eta_{16}, 2\iota_{17}, \nu^2_{17}\}_{12}
\subset -\{[\eta^2_{14}]_{16}, 2\iota_{16}, \nu^2_{16}\}_{11} \\
\ni[\eta_{14}\varepsilon_{15}]_{16}
\]
\[
 \bmod\ [\eta^2_{14}]_{16}\circ E^{11}\pi^5_{12}\subset [\eta^2_{14}]_{16}\circ 8\pi^{16}_{23}  = 0.
\]
This implies $\Delta(\varepsilon_{16}) 
= [\eta_{14}\varepsilon_{15}]_{16}$. 
and completes the proof.
\end{proof}

Next we show
\begin{prop} \label{R18}
$R^{18}_{23} = \Z\{[32[\iota_{12}, \iota_{12}]]_{18}\}
\oplus\Z_2\{[\sigma^2_9]_{18}\}\oplus\Z_2\{[\varepsilon_{15}]_{18}\}$, where 
$[\varepsilon_{15}]_{18}=[\bar{\nu}_{15}]_{18}$. 
\end{prop}
\begin{proof}
Since $P(\nu^2_{35}) =\omega_{17}\nu^2_{33} \neq 0$ \cite[p.~323]{Mi1},
$\Delta:\pi^{17}_{33}\to R^{17}_{22}$ is a monomorphism.
So, we have ${i_{18}}_*R^{17}_{23}=R^{18}_{23}$.
By relations 
$\Delta(\iota_{17}) = [\eta_{15}]_{17}$ \eqref{Ke2} and
$[\eta_{15}]\sigma_{16}=[\bar{\nu}_{15}]+[\varepsilon_{15}]$, 
we obtain 
$
\Delta(\sigma_{17})=[\eta_{15}]_{17}\sigma_{16}=[\bar{\nu}_{15}]_{17}+[\varepsilon_{15}]_{17}.
$
This completes the proof. 
\end{proof}

Next we show
\begin{prop}\label{R1923}
$R^{19}_{23} = \Z\{[32[\iota_{12}, \iota_{12}]]_{19}\}
\oplus\Z_2\{[\sigma^2_9]_{19}\}$, %\oplus\Z_2\{[\varepsilon_{15}]_{18}\}
where $\Delta(\nu^2_{18}) = [\varepsilon_{15}]_{18}$. 
\end{prop}
\begin{proof} 
It suffices to show the relation. 
By the relation $[\eta_{15}]_{17} = \Delta(\iota_{17})$ \eqref{Ke2}, Lemma \ref{[Ea]} and the fact that $R^{17}_{17}\circ\nu_{17}
\subset R^{17}_{20} = 0$ \cite[p. 161, Table]{Ke}, we have 
\[\begin{split}
[\varepsilon_{15}]_{18} &= {i_{16,18}}_*\{[\eta_{15}],2\iota_{16},\nu^2_{16}\}
\subset {i_{18}}_*\{\Delta(\iota_{17}),2\iota_{16},\nu^2_{16}\}\\
&\supset -\{i_{18},\Delta(\iota_{17}),2\iota_{16}\}\circ\nu^2_{17}\\
&\ni \Delta(\iota_{18})\circ\nu^2_{17} = \Delta(\nu^2_{18})
\bmod {i_{18}}_*R^{17}_{17}\circ\nu^2_{17} = 0.
\end{split}\]
This implies  $\Delta(\nu^2_{18}) = [\varepsilon_{15}]_{18}$.
\end{proof}

Since $\pi^{19}_{23+i} = 0$ for $i = 0$ and $1$, 
we have $R^{19}_{23}\cong R^{20}_{23}$. Since $\pi^{20}_{24} = 0$ 
and $\sharp P(\nu_{41}) = 8$ \cite[Lemma 8.3, Theorem B]{Mi1}, 
we have $R^{20}_{23}\cong R^{21}_{23}$. 
 
We consider the exact sequence $(22)_{23}$:
$$
\pi^{21}_{24}\rarrow{\Delta}R^{21}_{23}
\rarrow{{i}_*}R^{22}_{23}\rarrow{{p}_*}\pi^{21}_{23}
$$
By \cite[Table]{Ke}, $R^{22}_{23}\cong\Z$. We know 
$P(\eta^2_{43}) = 4\sigma^*_{21} \ne 0$ \cite[Lemma 8.3, Theorem B]{Mi1}, we get that ${i_{22}}_*$ is a split epimorphism. Hence, we obtain the following.

\begin{prop} \label{R2223}
$\Delta(\nu_{21}) = [\sigma^2_9]_{21}$ and
$R^{22}_{23} = \Z\{[32[\iota_{12}, \iota_{12}]]_{22}\}$.
\end{prop}

Obviously we obtain
$$
R^n_{23} = \Z\{[32[\iota_{12}, \iota_{12}]]_n\} (n = 23, n\geq 25)
$$
and
$$
R^{24}_{23} = \Z\{[32[\iota_{12}, 
\iota_{12}]]_{22}\}\oplus\Z\{\Delta\iota_{24}\}.
$$

We recall the relation  $P(\nu_{43})=\sigma_{21}\omega_{28}+\omega_{21}\sigma_{37}$ 
\cite[III-Proposition 2.1(3)]{Od1}. So, by the fact that $P(\nu_{43})=\psi_{21}$ and 
that $E^7: \pi^{14}_{37}\to\pi^{21}_{44}$ is isomorphic onto \cite[Theorem 1.1(a)]{MMO}, we obtain
%\begin{equation}\label{om14sg}
$\omega_{14}\sigma_{30}+\sigma_{14}\omega_{21}=\psi_{14}.$

Finally we can generalize of the equality $\Delta(\nu^2_{18}) 
= [\varepsilon_{15}]_{18}$. We define 
$$
[\varepsilon_{8n-1}]\in\{[\eta_{8n-1}],2\iota_{8n},\nu^2_{8n}\}_1\ 
\mbox{for} \ n\geq 2.
$$
Then, by the quite parallel argument to Proposition 4.15, we obtain 
$$
\Delta(\nu^2_{8n+2}) = [\varepsilon_{8n-1}]_{8n+2}\ \mbox{for}\ n\geq 2.
$$
This implies the relation $P(\nu^2_{16n+5}) 
= E^2J[\varepsilon_{8n-1}]$ and it is regarded as an example of \cite[Proposition 11.11(ii)]{T}.

By \cite{A0} and \cite{Bo}, we have 
$J[32[\iota_{12},\iota_{12}]]_\infty
=\bar{\rho}$. So, by \cite[Theorem 1.1(b)]{MMO}, we have 
\begin{equation}\label{J32}
J[32[\iota_{12},\iota_{12}]]\equiv\bar{\rho}_{13}\ \bmod\ \psi_{13}.
\end{equation}

\section{Determination of $R^n_{24}\ (7\leq n\leq 9)$}
\hspace{2ex}

By the same argument as the proof of \cite[I-Proposition 3.4(3)]{Od1}, 
we can take 
\begin{equation}\label{defbsg6}
\bar{\sigma}_6\in\{\bar{\nu}_6+\varepsilon_6,\sigma_{14},\nu_{21}\}_1\ \bmod\ 2\bar{\sigma}_6.
\end{equation}

We need the following \cite[Table 2]{Ka}:
\[
R^6_7=\Z\{[\eta^2_5]\}, \ R^5_{10}=\Z_8\{[\nu^2_4]\}
\text{ and } {i_6}_* R^5_{10}=\Z_8\{[\nu^2_4]_6\}.
\]
We recall the relations 
$J[\eta^2_5]=\pm\sigma''$ \cite[Lemma 2.1(i)]{HKM} and $J[\nu^2_4]=\nu_5\sigma_8$ \cite[Lemma 1.1(ii)]{KM1}. 
By \cite[Proposition 2.2(1)]{Og}, we know 
$\sigma''\nu_{13} = \pm 2\nu_6\sigma_9$.
So we obtain 
\begin{equation}\label{[eta25nu}
[\eta^2_5]\nu_7=\pm 2[\nu^2_4]_6.
\end{equation}

By the fact that
\ $[\nu^2_4]_6\sigma^2_{10}\in i_*R^5_{24}$ and $R^5_{24}\cong(\Z_2)^2$, we
have 
\begin{equation}\label{2n24sg2}
2[\nu^2_4]_6\sigma^2_{10} = 0.
\end{equation}

By \cite[Proposition 2.3(2)]{Og}, we have 
\begin{equation}\label{k7n}
\kappa_7\nu_{21}=\nu_7\kappa_{10}.
\end{equation}
We show the following.

\begin{prop} \label{724}
$R^7_{24}=\{[\bar{\zeta}_5]_7,[\nu_5]_7\mu_8\sigma_{17},[\nu_4\sigma'\eta_{14}]_7\mu_{15},[\nu^2_4]_7\kappa_{10}, [P(E\theta)+\nu_6\kappa_9]\eta_{23}$,\\ 
$[\eta_6]\bar{\mu}_7, [\eta_6]\eta_7\mu_8\sigma_{17}\}\cong(\Z_2)^7$, where $\Delta\bar{\zeta}_6 = 2[\bar{\zeta}_5]$. 
\end{prop}
\begin{proof}
In the exact sequence $(24)_6$:
$$
\pi^6_{25}\rarrow{\Delta}R^6_{24}\rarrow{i_*}R^7_{24}
\rarrow{p_*}\pi^6_{24}\rarrow{\Delta}R^6_{23},
$$
We know $\Ker\ \{\Delta:\pi^6_{24}\to R^6_{23}\}= \{(P(E\theta)+\nu_6\kappa_9)\eta_{23},
\eta_6\bar{\mu}_7, \eta^2_6\mu_8\sigma_{17}\}$ and 
$\pi^6_{25} = \Z_8\{\bar{\zeta}_6\}\oplus\Z_{32}\{\bar{\sigma}_6\}$ \cite[Theorem 12.9]{T}. 
By using the relation \eqref{mt1}, Lemma \ref{zeta1} and \eqref{[bzeta_5]}, we have
\[\begin{split}
\Delta\bar{\zeta}_6\in\Delta\{\zeta_6, 8\iota_{17}, 2\sigma_{17}\}_2
&\subset\{2[\zeta_5], 8\iota_{16}, 2\sigma_{16}\}_1\\
&\ni 2[\bar{\zeta}_5]
\bmod 2[\zeta_5]\circ\pi^{16}_{24} + R^6_{17}\circ 2\sigma_{17}.
\end{split}\]
As $\pi^{16}_{24}\cong(\Z_2)^2$ \cite{T} and $R^6_{17}
= \Z_8\{[\nu^2_4]_6\sigma_{10}\}\oplus(\Z_2)^3$ \cite{KM1},
the indeterminacy is equal to $\{2[\nu^2_4]_6\sigma^2_{10}\}=0$ 
\eqref{2n24sg2}.
Hence the indeterminacy is trivial.
This implies $\Delta\bar{\zeta}_6 = 2[\bar{\zeta}_5]$.

We show $\Delta\bar{\sigma}_6=0$. By \eqref{defbsg6}, it suffices to 
prove $\Delta\{\bar{\nu}_6+\varepsilon_6,\sigma_{14},\nu_{21}\}_1=0$. 
Since $\Delta(\bar{\nu}_6+\varepsilon_6)=0$, we have 
$$
\Delta\{\bar{\nu}_6+\varepsilon_6,\sigma_{14},\nu_{21}\}_1
\subset\{0,\sigma_{13},\nu_{20}\}=R^6_{21}\circ\nu_{21}.
$$
By \cite[Theorem 0.1(i)]{HKM}, we know $R^6_{21}=\{[\nu_4\sigma'\sigma_{21}]_6, [\eta^2_5]\kappa_7\}\cong\Z_8\oplus\Z_2$.
We have $[\nu_4\sigma'\sigma_{21}]\nu_{21}\in{i_5}_*R^4_{24}=0$ \eqref{i5R424}. This implies $[\nu_4\sigma'\sigma_{21}]_6\nu_{21}=0$.
By \eqref{[eta25nu} and the relation $\kappa_7\nu_{21}=\nu_7\kappa_{10}$ 
\eqref{k7n}, we have $ [\eta^2_5]\kappa_7\nu_{21}
= [\eta^2_5]\nu_7\kappa_{10}=2[\nu^2_4]_6\kappa_{10}=0$.
Hence, by Proposition \ref{R61}, we obtain 
the group structure of $R^7_{24}$. 
\end{proof}

Since $[\nu_4\sigma'\mu_{14}]_7 \eta_{23}= [\nu_4\sigma'\eta_{14}]_7\mu_{15} \neq 0$,
we obtain
\[
2[P(E\theta)+\nu_6\kappa_9]
\equiv[\nu_5]_7\bar{\varepsilon}_8\ \bmod\ 
[\nu_4\eta_7\bar{\varepsilon}_8]_7, [\nu_4(E\zeta')]_7.
\]

Since $R^8_{24}\cong R^7_{24}\oplus\pi^7_{24}$ and $\pi^7_{24}
= \{\sigma'\eta_{14}\mu_{15}, \nu_7\kappa_{10}, \bar{\mu}_7,
\eta_7\mu_8\sigma_{17}\}\cong(\Z_2)^4$, we have the following.

\begin{prop} \label{824}
$R^8_{24} = \Z_2\{[\bar{\zeta}_5]_8\}
\oplus\Z_2\{[\nu_5]_8\mu_8\sigma_{17}\}
\oplus\Z_2\{[\nu_4\sigma'\eta_{14}]_8\mu_{15}\}
\oplus\Z_2\{[\nu^2_4]_8\kappa_{10}\}
\oplus\Z_2\{[\eta_6]_8\bar{\mu}_7\}
\oplus\Z_2\{[\eta_6]_8\eta_7\mu_8\sigma_{17}\}
\oplus\Z_2\{[P(E\theta)+\nu_6\kappa_9]_8\eta_{23}\}
\oplus\Z_2\{[\iota_7]\sigma'\eta_{14}\mu_{15}\}
\oplus\Z_2\{[\iota_7]\nu_7\kappa_{10}\}
\oplus\Z_2\{[\iota_7]\bar{\mu}_7\}
\oplus\Z_2\{[\iota_7]\eta_7\mu_8\sigma_{17}\}$.
\end{prop}

Next we show 

\begin{prop} \label{924}
$R^9_{24} = \Z_2\{[\bar{\zeta}_5]_9\}
\oplus\Z_2\{[\nu_5]_9\mu_8\sigma_{17}\}
\oplus\Z_2\{[P(E\theta)+\nu_6\kappa_9]_9\eta_{23}\}
\oplus\Z_2\{[\iota_7]_9\nu_7\kappa_{10}\}
\oplus\Z_2\{[\iota_7]_9\bar{\mu}_7\}
\oplus\Z_2\{[\iota_7]_9\eta_7\mu_8\sigma_{17}\}$
with relations 
$[\nu_4\sigma'\eta_{14}]_9\mu_{15}
=[\nu^2_4]_9\kappa_{10}
=0$ and
$[\bar{\nu}_6+\varepsilon_6]_9\eta_{14}\mu_{15}
=[\iota_7]_9\sigma'\eta_{14}\mu_{15}=[\nu_5]_9\mu_8\sigma_{17}$.
\end{prop}
\begin{proof}
In the exact sequence $(24)_8$:
\[
\pi^8_{25}\rarrow{\Delta}R^8_{24}\rarrow{i_*}R^9_{24}\rarrow{p_*}\pi^8_{24}
\rarrow{\Delta}R^8_{23},
\]
we know that $\Delta: \pi^8_{24}\to R^8_{23}$ is a monomorphism and that
\[
\pi^8_{25} \!=\! \{\sigma_8\eta_{15}\mu_{16}, (E\sigma')\eta_{15}\mu_{16},
\nu_8\kappa_{11}, \bar{\mu}_8, \eta_8\mu_9\sigma_{18}\}\!\cong\!(\Z_2)^5
\text{ \cite[Theorem 12.7]{T}}.
\]

By \eqref{Dsg82}, we obtain the following: for odd $x$, $y$, % and $c \in \{0, \cdots , 7\}$, 
\[\begin{split}
\Delta(\sigma_8\eta_{15}\mu_{16})
&\equiv x[\iota_7]\sigma'\eta_{14}\mu_{15} + y[\bar{\nu}_6+\varepsilon_6]_8\eta_{14}\mu_{15} \ \bmod\ 
([\nu_5]_8\nu^2_8 + [\eta_5\varepsilon_6]_8)\circ\eta_{14}\mu_{15}\\
&\equiv [\iota_7]\sigma'\eta_{14}\mu_{15}
+ [\bar{\nu}_6+\varepsilon_6]_8\eta_{14}\mu_{15}\ \bmod\ 
[\eta_5\varepsilon_6]_8\eta_{14}\mu_{15}.
\end{split}\]
By Lemma \ref{bnuepmu} and the relation 
$\zeta'\eta_{22}=0$ \cite[Proposition 2.13(5)]{Og}, we have
\[\begin{split}
[\bar{\nu}_6+\varepsilon_6]_8\eta_{14}\mu_{15}
&\equiv [\eta_6]_8(E\zeta' + \mu_7\sigma_{16})\eta_{23}
+ [\zeta_5]_8\sigma_{16}\eta_{23}\\
&=[\eta_6]_8\eta_7\mu_8\sigma_{17} + [\nu_5]_8\mu_8\sigma_{17}\\
&\bmod [\nu_5]_8\bar\varepsilon_8\eta_{23},
[\nu_4\sigma'\mu_{14}]_8\eta_{23}, [\nu_4(E\zeta')]_8\eta_{23},
[\nu_4\eta_7\bar{\varepsilon}_8]_8\eta_{23}.
\end{split}\]
So, by Lemma \ref{89}, we have
\[
[\bar{\nu}_6+\varepsilon_6]_8\eta_{14}\mu_{15}
\equiv [\eta_6]_8\eta_7\mu_8\sigma_{17} + [\nu_5]_8\mu_8\sigma_{17}
\bmod [\nu^2_4]_8\kappa_{10}, [\nu_4\sigma'\eta_{14}]_8\mu_{15}.
\]
By \cite[Lemmas 1.1(v), 1.2(iii)]{KM1} and Lemma \ref{89}, we obtain
\[
[\eta_6]_8\sigma'\eta_{14}\mu_{15} = [\eta_5\varepsilon_6]_8\eta_{14}\mu_{15}
= [\nu_4\sigma'\eta_{14}]_8\mu_{15}.
\]
Hence, we have
\begin{equation}\label{Dsg8etm}
%\[\begin{split}
\Delta(\sigma_8\eta_{15}\mu_{16})
\equiv [\iota_7]\sigma'\eta_{14}\mu_{15}
%&
+ [\eta_6]_8\eta_7\mu_8\sigma_{17} + [\nu_5]_8\mu_8\sigma_{17}\ 
%&\qquad
\bmod [\nu^2_4]_8\kappa_{10}, [\nu_4\sigma'\eta_{14}]_8\mu_{15}.
%\end{split}\]
\end{equation}
We recall the relation $\Delta\iota_8 = 2[\iota_7] - [\eta_6]_8$ \eqref{k3}. 
Since $[\eta_6]\nu_7 = x[\nu^2_4]_7$ for $x$ odd \cite[Lemma 1.1(iv)]{KM1},
we obtain
$\Delta(\nu_8\kappa_{11}) = [\eta_6]_8\nu_7\kappa_{10}=[\nu^2_4]_8\kappa_{10}$. 

We also get
\begin{align*}
&\Delta((E\sigma')\eta_{15}\mu_{16})
= [\eta_6]_8\sigma'\eta_{14}\mu_{15}= [\nu_4\sigma'\eta_{14}]_8\mu_{15};\\
&\Delta(\nu_8\kappa_{11}) = [\nu^2_4]_8\kappa_{10};\ 
\Delta\bar{\mu}_8 = [\eta_6]_8\bar{\mu}_7;\ 
\Delta(\eta_8\mu_9\sigma_{18}) = [\eta_6]_8\eta_7\mu_8\sigma_{17}.
\end{align*}
Thus we obtain the group $R^9_{24}$ and
the equation
\[
 [\nu_4\sigma'\eta_{14}]_9\mu_{15}=[\nu^2_4]_9\kappa_{10}
=[\eta_6]_9\bar{\mu}_7 = [\eta_6]_9\eta_7\mu_8\sigma_{17}=0.
\]
Furthermore, we obtain 
$[\bar{\nu}_6+\varepsilon_6]_9\eta_{14}\mu_{15}=[\nu_5]_9\mu_8\sigma_{17}$
by the relation
\[
[\bar{\nu}_6+\varepsilon_6]_9\eta_{14}\mu_{15}
\equiv [\eta_6]_9\eta_7\mu_8\sigma_{17} + [\nu_5]_9\mu_8\sigma_{17}
\bmod [\nu^2_4]_9\kappa_{10}, [\nu_4\sigma'\eta_{14}]_9\mu_{15}.
\]
By \eqref{Dsg8etm}, we also obtain
$[\iota_7]_9\sigma'\eta_{14}\mu_{15}=[\nu_5]_9\mu_8\sigma_{17}$.
This completes the proof.
\end{proof}

By the proof of Propositions \ref{R923}, \ref{924}, 
\cite[Proposition 2.1]{Ka} and \cite[Theorem 5.4]{HKM}, 
we notice the following.
\begin{rem}
${p_9}_*:R^9_n \to \pi^8_{n}$ is trivial for $n\leq 25$, except for $n=15$ and $22$.
The elements $[8\sigma_8]$ and $[8\sigma_8]\sigma_{15}$
are the candidates for the non-trivialities, respectively.
\end{rem}

\section{Determination of $\pi_{24}(R_n:2) \ (n\geq 10)$}
\hspace{2ex}
First we note that $\sigma_9(\bar{\nu}_{16} + \varepsilon_{16})
= \sigma^2_9\eta_{23}$ by \eqref{et9s}. In the exact sequence $(24)_{9}$:
\[%\begin{equation}\label{seq24_9}\tag*{$(24)_{9}$}
%(\ast) \hspace{4ex} 
\pi^9_{25}\rarrow{\Delta}R^9_{24}\rarrow{i_*}R^{10}_{24}
\rarrow{p_*}\pi^9_{24}\rarrow{\Delta}R^9_{23},
\]%\end{equation}
we know $\Ker\ \Delta = \{2\rho', \sigma^2_9\eta_{23}, \bar{\varepsilon}_9\}$
from the proof of Proposition \ref{R10}.
So there exist elements %$[2\rho']$, 
$\Delta(E\rho')$, $[\sigma^2_9]\eta_{23}$, 
$[\bar{\varepsilon}_9] \in R^{10}_{24}$. 

We need the following.
\begin{lem}\label{tn7k}
\begin{enumerate}
\item $\bar{\varepsilon}_5=\{2\nu_5,\nu_8, \bar\nu_{11}\}_1$.
\item $\{\sigma'\eta_{14}+\varepsilon_7,\bar{\nu}_{15},2\iota_{23}\}_1
=\nu_7\kappa_{10}+\{\sigma'\eta_{14}\mu_{15}+\eta_7\mu_8\sigma_{15}\}$.
\end{enumerate}
\end{lem}
\begin{proof}
Since $\pi^7_{19}=0$ \cite[Theorem 7.6]{T}, we have
$\{2\nu_5,\nu_8,\bar{\nu}_{11}\}_1=\{2\nu_5,\nu_8,\bar{\nu}_{11}\}$.
By \cite[Theorem 2(2)]{MaO} and \cite[(3.9)-i)]{T}, we know
$\langle 2\nu, \nu,\bar{\nu} \rangle 
= \langle \bar{\nu},\nu,2\nu\rangle=\bar{\varepsilon}$.
So, by the fact that $E^\infty:\pi^5_{20}\to \pi^S_{15}$ is a monomorphism \cite[Theorems 10.5 and 10.10]{T}, we obtain (1).

Next, the Toda bracket 
$\{\sigma'\eta_{14}+\varepsilon_7,\bar{\nu}_{15},2\iota_{23}\}_1$ is 
defined by $(\sigma'\eta_{14}+\varepsilon_7)\bar\nu_{15}=0$ (\eqref{epep}, Lemma \ref{zesig}(3)).
It has the indeterminacy 
$(\sigma'\eta_{14}+\varepsilon_7)\circ E\pi^{14}_{23}+2\pi^{7}_{24}
=(\sigma'\eta_{14}+\varepsilon_7)\circ \{\nu^3_{15},\mu_{15},\eta_{15}\varepsilon_{16}\}$ \cite[Theorems 7.2 and 12.7]{T}.
We have $(\sigma'\eta_{14}+\varepsilon_7)\nu^3_{15}=0$ 
(\eqref{t1}, \eqref{t6}).
By \eqref{Ogu0}, we have $(\sigma'\eta_{14}+\varepsilon_7)\mu_{15}
=\sigma'\eta_{14}\mu_{15}+\eta_7\mu_8\sigma_{15}$.  
We also have $(\sigma'\eta_{14}+\varepsilon_7)\eta_{15}\varepsilon_{16}=0$ \eqref{et^2e} and
$\varepsilon_7\eta_{15}\varepsilon_{16}=\eta_7\varepsilon^2_8=\eta_7\nu_8\sigma_{11}\nu^2_{18}=0$ \eqref{epep}.
Hence,  the indeterminacy  of $\{\sigma'\eta_{14}+\varepsilon_7,\bar{\nu}_{15},2\iota_{23}\}_1$
is $\{\sigma'\eta_{14}\mu_{15}+\eta_7\mu_8\sigma_{15}\}$.

By \cite[III-Proposition 6.1(2)]{Od1} and the fact that 
$\varepsilon_{10}\circ\pi^{18}_{27} + 2\pi^{10}_{27} = \{\eta_{10}\mu_{11}\sigma_{20}\}$, 
we have 
$$\{\varepsilon_{10},\bar{\nu}_{18},2\iota_{26}\}
=\nu_{10}\kappa_{13}+\{\eta_{10}\mu_{11}\sigma_{20}\}.
$$
This and the fact that  $\Ker\{E^3: \pi^7_{24}\to\pi^{10}_{27}\} = \{\sigma'\eta_{14}\mu_{15}\}$ \cite[Theorems 12.7 and 12.17]{T} 
imply (2). 
This %leads to (2) and 
completes the proof. 
\end{proof}

By \cite[Proposition 2.4(2)]{Og}, we have 
\begin{equation}\label{zt5nu}
\zeta_5\nu_{16}\equiv\nu_5\zeta_8\ \bmod\ \nu_5\bar{\nu}_8\nu_{16}.
\end{equation}

We show 
\begin{lem}\label{n5musg}
\begin{enumerate}
\item
$\nu_5\sigma_8\mu_{15} = \nu_5\mu_8\sigma_{17}$.
\item $[\zeta_5]\sigma_{16}\eta_{23}\equiv
[\nu_5]\sigma_8\mu_{15}\equiv[\nu_5]\mu_8\sigma_{17}\ \bmod \ 
%[\nu^2_4]_6\kappa_{10}, 
[\nu_4\sigma'\eta_{14}]_6\mu_{15}$ and $[\zeta_5]_9\sigma_{16}\eta_{23}=[\nu_5]_9\sigma_8\mu_{15}=[\nu_5]_9\mu_8\sigma_{17}$.
\item
$[\zeta_5]\bar\nu_{16}=0$ and 
$[\zeta_5]\varepsilon_{16} = [\zeta_5]\eta_{16}\sigma_{17}$.
\end{enumerate}
\end{lem}
\begin{proof}
By $\nu_5\sigma_8\eta_{15} = \nu_5\varepsilon_8$ \cite[p.~152]{T} and
Lemma \ref{m3sg}(1), we have 
\[\begin{split}
\nu_5\sigma_8\mu_{15}
&\in \nu_5\sigma_8\circ\{\eta_{15},8\iota_{16},2\sigma_{16}\}_1
\subset\{\nu_5\varepsilon_8,8\iota_{16},2\sigma_{16}\}_1\\
&\supset\nu_5\circ\{\varepsilon_8,8\iota_{16},2\sigma_{16}\}_1
\ni\nu_5\mu_8\sigma_{17}\\
&\bmod \nu_5\varepsilon_8\circ E\pi^{15}_{23} + \pi^5_{17}\circ 2\sigma_{17}
+ \{\nu_5\eta_8\bar{\varepsilon}_8\}=0.
\end{split}\]
This leads to (1).

By \eqref{zesig1} and the fact that $\mu_{15}\in \{8\iota_{15},2\sigma_{15},\eta_{22}\}$, we have 
$$
[\zeta_5]\sigma_{16}\eta_{23}\in [\nu_5]\sigma_8\circ \{8\iota_{15},2\sigma_{15},\eta_{22}\}\ni [\nu_5]\sigma_8\eta_{15}
$$
$$ 
\bmod\ [\nu_5]\bar{\varepsilon}_8\eta_{23},
[\nu_4\sigma'\mu_{14}]_6\eta_{23},
[\nu_4(E\zeta')]_6\eta_{23}, 
[\nu_4\eta_7\bar{\varepsilon}_8]_6\eta_{23}.
$$
We have $[\nu_5]\bar{\varepsilon}_8\eta_{23} = [\nu_5](E\sigma')\nu^3_{15} = 0$ \eqref{2[nu5]s}. 
By Lemma \ref{89}, we have 
$[\nu_4\sigma'\mu_{14}]_6\eta_{23} 
= [\nu_4\sigma'\eta_{14}]_6\mu_{15}$, 
$[\nu_4(E\zeta')]\eta_{23} = 0$ and 
$[\nu_4\eta_7\bar{\varepsilon}_8]_6\eta_{23} = 0$. 
This leads to the first of (2).  

By (1) and \eqref{524}, we obtain 
$[\nu_5]\sigma_8\mu_{15}\equiv[\nu_5]\mu_8\sigma_{17}\ \bmod \ 
[\nu^2_4]_6\kappa_{10}, 
[\nu_4\sigma'\eta_{14}]_6\mu_{15}$
By \eqref{Jn5}, we have $J([\nu_5]\sigma_8\mu_{15}) = J([\nu_5]\mu_8\sigma_{17}) = 0$. By the fact that $J[\nu^2_4] = \nu_5\sigma_8$, 
we know $J([\nu^2_4]_6\kappa_{10})=\nu_6\sigma_9\kappa_{16}\ne 0$ 
\cite[Theorem 1.1(a)]{MMO}. This leads to the second of (2). 

The third of (2) follows from the fact 
that $[\nu_4\sigma'\eta_{14}]_9=0$.

By \cite[Proposition 2.2(6)]{Og}, we know $\zeta_6\eta_{17}=8P(\sigma_{13})$. 
We consider the matrix Toda bracket.
\[
\mtoda{\zeta_6}{\eta_{17}}{{\nu}_{18}}
{[\iota_6,\iota_6]}{8\sigma_{11}}_1
\]
Since $H\mtoda{\zeta_6}{\eta_{17}}{{\nu}_{18}}
{[\iota_6,\iota_6]}{8\sigma_{11}}_1\subset\{2\iota_{11},8\sigma_{11},\nu_{18}\}_1=\zeta_{11}$, this is regarded as $\zeta'$. We have 
$E\zeta'=\{\zeta_7,\eta_{18},\nu_{19}\}_1$ and $\{\zeta_8,\eta_{19},\nu_{20}\}_1=0$. 

By \eqref{zt5nu} and \cite[Theorem 2.1]{KM2}, we have 
$$
[\zeta_5]\nu_{16}\equiv \pm[\nu_5]\zeta_8\ \bmod\ [\nu_5]\bar{\nu}_8\nu_{16}.
$$
So, by the fact that $\bar{\nu}_{16}=\{\nu_{16},\eta_{19},\nu_{20}\}_1$ and $\bar{\nu}^2_8=0$ (Lemma \ref{zesig}), for $a\in\{0, 1\}$, we have
$$
[\zeta_5]\bar{\nu}_{16}=[\zeta_5]\circ\{\nu_{16},\eta_{19},\nu_{20}\}_1
\subset \{[\zeta_5]\nu_{16},\eta_{19},\nu_{20}\}_1
$$
$$
=\{\pm[\nu_5]\zeta_8+a[\nu_5]\bar{\nu}_8\nu_{16},\eta_{19},\nu_{20}\}_1
\subset\{\pm[\nu_5]\zeta_8,\eta_{19},\nu_{20}\}_1
+\{a[\nu_5]\bar{\nu}_8\nu_{16},\eta_{19},\nu_{20}\}_1
$$
$$
\supset\pm[\nu_5]\circ\{\zeta_8,\eta_{19},\nu_{20}\}_1
+a[\nu_5]\bar{\nu}_8\circ\{\nu_{16},\eta_{19},\nu_{20}\}_1
\ni 0\ \bmod\ R^6_{21}\circ\nu_{21}.
$$
By \cite[Theorem 0.1(i)]{HKM}, we have $R^6_{21}\circ\nu_{21}=\{[\nu_4\sigma'\sigma_{14}]_6\nu_{21}, [\eta^2_5]\kappa_7\nu_{21}\}$.
By \eqref{i5R424} and \eqref{[eta25nu}, we have $[\nu_4\sigma'\sigma_{14}]\nu_{21}\in{i_5}_*R^4_{24}=0$ and 
$[\eta^2_5]\kappa_7\nu_{21}= [\eta^2_5]\nu_7\kappa_{10}=0$.
This leads to the first of (3).

The second of (3) follows from the fact that $\bar{\nu}_{16}=\eta_{16}\sigma_{17}+\varepsilon_{16}$. 
This completes the proof. 
\end{proof}

Here, we need
\begin{lem}\label{ind[bep9}
\begin{enumerate}
\item
$[\iota_7]_{10}\sigma'\eta_{14}\mu_{15} = [\iota_7]_{10}\eta_7\mu_8\sigma_{17}=0$.
\item
$\Delta(\sigma_{10}\bar{\nu}_{17}) = 0$.
\end{enumerate}
\end{lem}
\begin{proof}
We recall from \eqref{Desig9} and \eqref{bn6m}:
$$
\Delta\sigma_9 = [\iota_7]_9(\bar{\nu}_7+\varepsilon_7), \ 
\bar{\nu}_6\mu_{14}=0.
$$
So, by \eqref{Ogu0}, we have
$$
\Delta(\sigma_9\mu_{16}) = [\iota_7]_9(\bar{\nu}_7+\varepsilon_7)\mu_{15} = [\iota_7]_9\eta_7\mu_8\sigma_{17}.
$$
This leads to the relation $[\iota_7]_{10}\eta_7\mu_8\sigma_{17} = 0$.

By \eqref{[nu5]sg8}, Lemma \ref{n5musg}(2) and the relation 
$[\nu_5]_{10} = [\iota_7]_{10}\eta_7$, we have $[\iota_7]_{10}\sigma'\eta_{14}\mu_{15} 
= [\nu_5]_{10}\sigma_8\mu_{15} = [\nu_5]_{10}\mu_8\sigma_{17} = 
[\iota_7]_{10}\eta_7\mu_8\sigma_{17} = 0$. This leads to (1).

Since  $\sigma_{10}\bar{\nu}_{17}=[\iota_{10},\nu^2_{10}]$
\cite[(10.20)]{T} and $\Delta[\iota_{10},\nu_{10}]=0$ \cite[(4.4)]{HKM},
we have
\[
\Delta(\sigma_{10}\bar{\nu}_{17})=\Delta([\iota_{10},\nu_{10}]\circ\nu_{22})
=\Delta[\iota_{10},\nu_{10}]\circ \nu_{21}=0.
\]
This leads (2) and completes the proof. 
\end{proof}

By \cite[(3.9)]{MMO}, we know
$$
H(\tilde{\varepsilon}_{10}) = \bar{\varepsilon}_{19} 
\ \mbox{and} \  2\tilde{\varepsilon}_{10}=\sigma_{10}\nu_{17}\kappa_{20}.
$$
This implies the relation
$$
J[\bar{\varepsilon}_9]\equiv\tilde{\varepsilon}_{10}\ \bmod \ 
E\pi^9_{33}.
$$

By \eqref{n6e}, we have $\nu_{12}\bar\nu_{15}=0$. 
We also have 
$(\varepsilon_7 + \sigma'\eta_{14})\bar\nu_{15}=0$ 
since
$\varepsilon_3\bar\nu_{11}=\eta_3\bar\varepsilon_4$ \eqref{epep} and
$\sigma'\eta_{14}\bar\nu_{15}=\eta_7\bar\varepsilon_8$ (Lemma \ref{zesig}(3)). 
Therefore, by \eqref{Dnu^2}, the following matrix Toda bracket (\cite{Mi1}, \cite{MaO}, \cite{MiM1}) is well defined:
\[
A = 
\mtoda{\Delta(\nu_{10})}{\nu_{12}}{\bar\nu_{15}}
{[\iota_7]_{10}}{\varepsilon_7 + \sigma'\eta_{14}}.
\]
%Let $A$ be  a representative of this matrix Toda bracket.
Then, by \eqref{pD} and \cite[Lemma 4.1(2)]{MiM1}, we have 
\[
p_{10 \ast}(A) \in
\mtoda{2\nu_{9}}{\nu_{12}}{\bar\nu_{15}}{0}{\varepsilon_7 + \sigma'\eta_{14}}
=\{2\nu_9, \nu_{12}, \bar\nu_{15}\}.
\]
We see that 
$\bar\varepsilon_9 \in \{2\nu_9, \nu_{12}, \bar\nu_{15}\}
\bmod 2\nu_9 \circ \pi^{12}_{24}+\pi^9_{16}\circ\bar\nu_{16}=0$ by Lemma \ref{tn7k}(1). 
Therefore, $A$ is a lift of $\bar\varepsilon_9$. 
So, we take
\begin{equation}\label{mTbep}
[\bar{\varepsilon}_9]=\mtoda{\Delta(\nu_{10})}{\nu_{12}}{\bar\nu_{15}}{[\iota_7]_{10}}
{\varepsilon_7 + \sigma'\eta_{14}}.
\end{equation}

We will determine the indeterminacy of this matrix Toda bracket.
It is $\Delta(\nu_{10})\circ\pi^{12}_{24} + [\iota_7]_{10}\circ\pi^7_{24} + R^{10}_{16}\circ\bar{\nu}_{16}$.
We know $\pi^{12}_{24} = \{\theta, E\theta'\}\cong(\Z_2)^2$ 
\cite[Theorem 7.6]{T}. Since $E\theta = [\iota_{13},\iota_{13}]$ and 
$E\theta' = [\iota_{12},\eta_{12}]$, we have $\nu_{10}(E\theta)
= [\iota_{10},\nu_{10}]\nu_{22}$ and $\nu_{12}(E^2\theta') = 0$.
By \cite[(4.4)]{HKM}, we have $\Delta(\nu_{10})\theta = 0$. This 
implies $\Delta(\nu_{10})\circ\pi^{12}_{24} = 0$. We know 
$\pi^7_{24} = \{\sigma'\eta_{14}\mu_{15}, \nu_7\kappa_{10}, \bar{\mu}_7, 
\eta_7\mu_8\sigma_{17}\cong(\Z_2)^4$ \cite[Theorem 12.7]{T}. So, by 
Lemma \ref{ind[bep9}, we have 
$$
[\iota_7]_{10}\circ\pi^7_{24} = \{[\iota_7]_{10}\nu_7\kappa_{10}, 
[\iota_7]_{10}\bar{\mu}_7\}.
$$
By the fact that $R^{10}_{16} = \{\Delta(\sigma_{10}), [\iota_7]_{10}\nu^3_7, [\iota_7]\mu_7\}$ \eqref{R1016}, the relation $\nu_{13}\bar{\nu}_{16} = 0$, $\mu_7\bar{\nu}_{16} = 0$ and 
$\sigma_{10}\bar{\nu}_{17} = [\iota_{10},\nu^2_{10}]$, 
we have $R^{10}_{16}\circ\bar{\nu}_{16} = 0$. 
This shows that the indeterminacy of $[\bar{\varepsilon}_9]$ is 
$\{[\iota_7]_{10}\nu_7\kappa_{10}, [\iota_7]_{10}\bar{\mu}_7\}$.

Now we show the following.

\begin{prop} \label{10}
$R^{10}_{24} = \Z_{16}\{\Delta(E\rho')\}\oplus\Z_4\{[\bar{\varepsilon}_9]\}
\oplus\Z_2\{[\sigma^2_9]\eta_{23}\}\oplus\Z_2\{[\iota_7]_{10}\bar{\mu}_7\}
\oplus\Z_2\{[P(E\theta)+\nu_6\kappa_9]_{10}\eta_{23}\}$, where
$[P(E\theta)+\nu_6\kappa_9]_{10}\eta_{23}
= [\bar{\nu}_9\nu^2_{17}]\eta_{23}$, 
$[\iota_7]_{10}\eta_7\mu_8\sigma_{17}
=[\nu_5]_{10}\mu_8\sigma_{17}= 0$, 
$8\Delta(E\rho') = [\bar{\zeta}_5]_{10}+[\iota_7]_{10}\bar{\mu}_7$ and $2[\bar{\varepsilon}_9]=[\iota_7]_{10}\nu_7\kappa_{10}$.
%\bmod[P(\Sigma\theta)+\nu_6\kappa_9]_{10}\eta_{23}$.
\end{prop}
\begin{proof}
We recall $\Delta\kappa_{10} = [\bar{\nu}_9\nu^2_{17}]
+ [P(E\theta)+\nu_6\kappa_9]_{10}\ \bmod\ 16[\sigma^2_9]$ in Proposition \ref{R11}. 
By \eqref{Det10}, we have
$\Delta\kappa_{10}\circ\eta_{23}
= \Delta(\eta_{10}\kappa_{11}) = 2[\iota_7]_{10}\nu_7\kappa_{10} = 0$. So we have the first relation 
$[P(E\theta)+\nu_6\kappa_9]_{10}\eta_{23}
= [\bar{\nu}_9\nu^2_{17}]\eta_{23}$.
The second relation is obtained by Lemma \ref{ind[bep9}(1).

We consider the exact sequence $(24)_9$:
$$
\pi^9_{25}\rarrow{\Delta}R^9_{24}\rarrow{i_*}R^{10}_{24}\rarrow{p_*}
\pi^9_{24}\rarrow{\Delta}R^9_{23}\rarrow{}\cdots.
$$
Since $J\Delta(E\rho') = 2P(\rho_{21})$ and 
$\sharp (P(\rho_{21})) = 32$ \cite[(5.15)]{MMO}, 
we have $\sharp(\Delta(E\rho')) = 16$. 

We know $\pi^9_{25} = \{\sigma_9\nu^3_{16},
\sigma_9\mu_{16}, \sigma_9\eta_{16}\varepsilon_{17}, \mu_9\sigma_{18}\}
\cong(\Z_2)^4$ \cite[Theorem 12.6]{T}. We recall
\begin{gather*}
\Delta\sigma_9 = [\iota_7]_9(\bar{\nu}_7+\varepsilon_7)
\quad\text{\eqref{Desig9}}, \\
\kappa_7\nu_{21}=\nu_7\kappa_{10}
\quad\text{and}\quad
\bar{\nu}_6\mu_{14}=0
\quad\text{\eqref{k7n}, \eqref{bn6m}}.
\end{gather*}
So, by \eqref{t6} and the fact that
$2\kappa_7=\bar{\nu}_7\nu^2_{15}$, %\cite[(15.5)]{MT1} 
and \eqref{Ogu0}, 
we have
\begin{gather*}
\Delta(\sigma_9\nu^3_{16}) = [\iota_7]_9(\bar{\nu}_7+\varepsilon_7)\nu^3_{15}
= [\iota_7]_9\bar{\nu}_7\nu^3_{15} =  2[\iota_7]_9\nu_7\kappa_{10} = 0.
\end{gather*}

By Lemma \ref{zesig} (2), (3) and  \eqref{t1}, 
we have
\[
\Delta(\sigma_9\eta_{16}\varepsilon_{17})
= \Delta(\sigma_9\varepsilon_{16}\eta_{24})
= [\iota_7]_9(\bar{\nu}_7+\varepsilon_7)\varepsilon_{15}\eta_{23}
= [\iota_7]_9\sigma'\nu^3_{14}\eta_{23} %= [\iota_7]_9\eta^2_7\bar{\varepsilon}_9
= 0.
\]

In the proof of Lemma \ref{ind[bep9}(1), we have 
\[
\Delta(\sigma_9\mu_{16}) = [\iota_7]_9\eta_7\mu_8\sigma_{17}.
\]

By \eqref{k4}, we have $\Delta(\mu_9\sigma_{18})
= [\iota_7]_9\eta_7\mu_8\sigma_{17} + [\nu_5]_9\mu_8\sigma_{17}$, and hence, we obtain 
\begin{equation}\label{i10R^924}
{i_{10}}_* R^9_{24}=
\{[\bar{\zeta}_5]_{10}, [P(E\theta)+\nu_6\kappa_9]_{10}\eta_{23}, 
[\iota_7]_{10}\nu_7\kappa_{10}, [\iota_7]_{10}\bar{\mu}_7\}
\cong (\Z_2)^4.
\end{equation}

We will show the relation $8\Delta(E\rho') = [\bar{\zeta}_5]_{10}+[\iota_7]_{10}\bar{\mu}_7$. 

Since $[\zeta_5]_{11}\equiv [\iota_7]_{11}\mu_7\ \bmod\ [\iota_7]_{11}\nu^3_7$ \eqref{kazt10}, we have 
$$
[\bar{\zeta}_5]_{11}\in\{[\zeta_5]_{11},8\iota_{16},2\sigma_{16}\}_1
\equiv\{[\iota_7]_{11}\mu_7,8\iota_{16},2\sigma_{16}\}_1\ 
\bmod\ \{[\iota]_{11}\nu^3_7,8\iota_{16},2\sigma_{16}\}_1.
$$

By relations $\zeta_{13}=\{\nu_{13},16\iota_{16}, \sigma_{16}\}$ \cite[(9.1)]{T},
\eqref{n9z} and \eqref{n6e}, we have 
\[
\{[\iota_7]_{11}\nu^3_7, 8\iota_{16}, 2\sigma_{16}\}_1
\supset
[\iota_7]_{11}\nu^2_7\circ\{\nu_{13}, 8\iota_{16}, 2\sigma_{16}\}_1\ni 
[\iota_7]_{11}\nu^2_7\zeta_{13}=0
\]
with the indeterminacy 
$[\iota_7]_{11}\nu^3_7\circ E\pi^{15}_{23}+R^{11}_{17}\circ 2\sigma_{17}
=0$.

We have 
$$
\{[\iota_7]_{11}\mu_7,8\iota_{16},2\sigma_{16}\}_1
\supset
[\iota_7]_{11}\circ\{\mu_7,8\iota_{16},2\sigma_{16}\}_1
\ni[\iota_7]_{11}\bar{\mu}_7
$$
$$
\bmod\ [\iota_7]_{11}\mu_7\circ E\pi^{15}_{23} 
+ R^{11}_{17}\circ 2\sigma_{17} = 0.
$$
Notice that $[\iota_7]_{11}\mu_7\circ E\pi^{15}_{23} = \{[\iota_7]_{11}\eta_7\mu_8\sigma_{17}\} = 0$. This implies 
$$
[\bar{\zeta}_7]_{11} = [\iota_7]_{11}\mu_7.
$$
So, by the fact that $\Delta(\bar{\varepsilon}_{10})=0$, we have 
$[\bar{\zeta}_5]_{10}+[\iota_7]_{10}\bar{\mu}_7\in
\Delta(\pi^{10}_{25}) = \{\Delta(E\rho')\}$.
This and \eqref{i10R^924} implies  
$$
8\Delta(E\rho') = [\bar{\zeta}_5]_{10}+[\iota_7]_{10}\bar{\mu}_7.
$$

By \eqref{mTbep} and \cite[p. 406, the first formula]{MaO},%\cite[Lemma 4.2]{MiM1},
we have 
\[
2[\bar{\varepsilon}_9]
 \in -\Delta(\nu_{10})\circ\{\nu_{12}, \bar\nu_{15}, 2\iota_{23}\}
 +[\iota_7]_{10}\circ\{\varepsilon_7+\sigma'\eta_{14},\bar\nu_{15}, 2\iota_{23}\}. 
\]
 Here, $\{\nu_{12}, \bar\nu_{15},2\iota_{23}\}
\supset E^3\{\nu_9, \bar\nu_{12}, 2\iota_{20}\}=0 \bmod
\nu_{12}\circ\pi^{15}_{24}+2\pi^{12}_{24}=0$ by $\pi^9_{21}=0$ and
$E^3\{\varepsilon_7+\sigma'\eta_{14},\bar\nu_{15}, 2\iota_{23}\}
\subset \{\varepsilon_{10}, \bar\nu_{18}, 2\iota_{26}\} \ni
\nu_{10}\kappa_{13}$ by \cite[III-Proposition 6.1(2)]{Od1}
and  $\varepsilon_{10}\circ\pi^{18}_{27}+2\pi^{10}_{27}=0$ .  
Therefore, we have $\{\varepsilon_7+\sigma'\eta_{14},
 \bar\nu_{15},2\iota_{23}\} \ni \nu_7\kappa_{10} \bmod\
\sigma'\eta_{14}\mu_{15}$\ since\ $\Ker\{E^3:\pi^7_{24} \to \pi^{10}_{27}\}=\{\sigma'\eta_{14}\mu_{15}\}$ \cite[Theorems 12.7 and 12.17, p.~146 and (12.21)]{T}. Hence, by Lemma \ref{ind[bep9}(1), 
we have 
\[
 2[\bar{\varepsilon}_9] \equiv [\iota_7]_{10}\nu_7\kappa_{10} \bmod [\iota_7]_{10}\sigma'\eta_{14}\mu_{15} = 0.
\]
This completes the proof.
\end{proof}

Next we show the following.

\begin{prop} \label{R1124}
$R^{11}_{24} = \Z_4\{[\bar{\varepsilon}_9]_{11}\}
\oplus\Z_2\{[\sigma^2_9]_{11}\eta_{23}\}
\oplus\Z_2\{[\iota_7]_{11}\bar{\mu}_7\}
\oplus\Z_2\{[P(E\theta)+\nu_6\kappa_9]_{11}\eta_{23}\}$, where $[\bar{\zeta}_5]_{11}=[\iota_7]_{11}\bar{\mu}_7$.
\end{prop}
\begin{proof}
In the exact sequence $(24)_{10}$:
$$
\pi^{10}_{25}\rarrow{\Delta}R^{10}_{24}\rarrow{i_*}R^{11}_{24}
\rarrow{p_*}\pi^{10}_{24}\rarrow{\Delta}R^{10}_{23},
$$
We know that $\Delta: \pi^{10}_{24}\to R^{10}_{23}$ is a monomorphism and $\pi^{10}_{25}
= \Z_{16}\{E\rho'\}\oplus\Z_2\{\sigma_{10}\bar{\nu}_{17}\}
\oplus\Z_2\{\bar{\varepsilon}_{10}\}$. 

By Lemma \ref{ind[bep9}(2), we have
$\Delta(\sigma_{10}\bar{\nu}_{17})=0$. By relations 
\eqref{kaet} and \eqref{Det10}, we have  
$\Delta\bar{\varepsilon}_{10}=\Delta(\eta_{10})\kappa_{10}=2[\iota_7]_{10}\nu_7\kappa_{10}=0$.
This completes the proof.
\end{proof}

We show
\begin{lem}\label{[bep9]11eta}
\begin{enumerate}
\item
$[\bar{\varepsilon}_9]\eta_{24} \equiv 0\ \bmod\ 
[\iota_7]_{10}\zeta_7\sigma_{18}, 
[\iota_7]_{10}\eta_7\bar\mu_8$.
\item
$[\bar\varepsilon_9]_{11}\eta_{24}\equiv 0\ \bmod\ [\iota_7]_{11}\eta_7\bar\mu_8$.
\end{enumerate}
\end{lem}
\begin{proof}
By the definition of $[\bar{\varepsilon}_9]$ \eqref{mTbep}, \eqref{t3}, \eqref{t6} and \cite[Lemma 4.1(2)]{MiM1}, 
we have 
\[\begin{split}
[\bar{\varepsilon}_9]\eta_{24}
&\in\mtoda{\Delta(\nu_{10})}{\nu_{12}}{\bar\nu_{15}}{{[\iota_7]_{10}}}{\varepsilon_7 
+ \sigma'\eta_{14}}\circ\eta_{24}
\subset\mtoda{\Delta(\nu_{10})}{\nu_{12}}{\nu^3_{15}}
{{[\iota_7]_{10}}}{\varepsilon_7 + \sigma'\eta_{14}}\\
&\subset\mtoda{\Delta(\nu_{10})}{\nu^3_{12}}{\nu_{21}}{{[\iota_7]_{10}}}{0}
=\{\Delta(\nu_{10}),\nu^3_{12},\nu_{21}\}
+ [\iota_7]_{10}\circ \pi^7_{25}.
\end{split}
\]
By \eqref{bn9n}, we have 
\[\begin{split}
\{\Delta(\nu_{10}),\nu^3_{12},\nu_{21}\}
&=\{\Delta(\nu_{10}),\eta_{12}\bar{\nu}_{13},\nu_{21}\}\\
&\supset\{\Delta(\nu_{10}),\eta_{12},0\}\ni 0
\ \bmod\ \Delta(\nu_{10})\circ\pi^{12}_{25} + R^{10}_{22}\circ\nu_{22}.
\end{split}\]
Since $\pi^{12}_{25}=\{\theta\eta_{24}, (E\theta')\eta_{24}\}\cong(\Z_2)^2$,  
$E\pi^{12}_{25}=\pi^{13}_{26}=\{(E\theta)\eta_{25}\}\cong\Z_2$ and $E\theta=P(\eta_{27})$ \cite[Theorem 7.7, (7.30)]{T}, we have 
$\Delta(\nu_{10})\circ\pi^{12}_{25}\subset\Delta(\nu_{10}\circ P(\eta_{27}))
=0$. So, we have 
$$
\{\Delta(\nu_{10}),\nu^3_{12},\nu_{21}\}=R^{10}_{22}\circ\nu_{22}.
$$
By the fact that $R^{10}_{22}=\{[8\sigma_8]\sigma_{15},[\iota_7]_{10}\rho'',[\iota_7]_{10}\bar\varepsilon_7\}$ \cite[Theorem 0.2(iii)]{HKM}, \eqref{kaet}, \eqref{k7n} 
and \eqref{n11s}, we obtain
\[
[\bar{\varepsilon}_9]\eta_{24}\in R^{10}_{22}\circ\nu_{22} + [\iota_7]_{10}\circ \pi^7_{25}
 = [\iota_7]_{10}\circ \pi^7_{25}
\]
By the fact that $\pi^7_{25} = \{\zeta_7\sigma_{18}, 
\eta_7\bar{\mu}_8\}\cong\Z_8\oplus\Z_2$ \cite[Theorem 12.8]{T}, we obtain 
\[
[\iota_7]_{10}\circ\pi^7_{25} = \{[\iota_7]_{10}\zeta_7\sigma_{18}, 
[\iota_7]_{10}\eta_7\bar{\mu}_8\}.
\]
This leads to (1).

(1) implies (2) if we show $[\iota_7]_{11}\zeta_7\sigma_{18} = 0$. 
By \cite[Proposition 2.17(6)]{Og}, we have 
$$
\rho''\nu_{22}=\sigma'\zeta_{14}.
$$
We know $\pi^{11}_{26} = \Z_{16}\{E^2\rho'\} \oplus
\Z_2\{\bar\varepsilon_{11}\}$ \cite[Theorem 10.5]{T}. By \cite[Lemma 6]{Os}, 
we have
\[
\notag \Delta(E^2\rho')= [\iota_7]_{11}\nu_7(E\rho') =0
\quad\text{and}\quad
\notag \Delta(\bar\varepsilon_{11})=\Delta(\eta_{11})\kappa_{11} = 0.
\]
So it suffices to show $[\iota_7]_{12}\rho''\nu_{22}=0$.
By \cite[Lemma 6.3]{HKM} and the fact that $\sharp P(\nu_{25})=4$ \cite[Theorem 10.3]{T}, we have 
$$
[\iota_7]_{12}\rho''\nu_{22}
=4[P(\nu_{21})]_{12}\nu_{22}=4\Delta(P(\nu_{25})=0.
$$
This leads to (2) and completes the proof.
\end{proof}

We show the following.

\begin{prop} \label{R1224}
$R^{12}_{24} = \Z_2\{[\theta']\eta_{23}\}\oplus\Z_2\{[\eta_{11}]\theta\}
\oplus\Z_2\{[\bar{\varepsilon}_9]_{12}\}
\oplus\Z_2\{[\sigma^2_9]_{12}\eta_{23}\}
\oplus\Z_2\{[\iota_7]_{12}\bar{\mu}_7\}
\oplus\Z_2\{[P(E\theta)+\nu_6\kappa_9]_{12}\eta_{23}\}$.
\end{prop}
\begin{proof}
In the exact sequence $(24)_{11}$:
$$
\pi^{11}_{25}\rarrow{\Delta}R^{11}_{24}\rarrow{i_*}R^{12}_{24}
\rarrow{p_*}\pi^{11}_{24}\rarrow{\Delta}R^{11}_{23},
$$
we know $\pi^{11}_{25} = \{\sigma^2_{11}, \kappa_{11}\}\cong\Z_{16}
\oplus\Z_2$ \cite[Theorem 10.3]{T} 
and 
\[
\Ker\ \{\Delta: \pi^{11}_{24}\to R^{11}_{23}\}
= \pi^{11}_{24} = \{\theta'\eta_{23}, \eta_{11}\theta\}\cong(\Z_2)^2
\] 
from the proof of Proposition \ref{R12}.
By \eqref{k4.5},
\eqref{n7sigsig} and  Proposition \ref{10},
we have 
$\Delta(\sigma^2_{11}) = [\iota_7]_{11}\nu_7\sigma^2_{10} = 0$ and
$
\Delta\kappa_{11}=[\iota_7]_{11}\nu_7\kappa_{10}
=2[\bar{\varepsilon}_9]_{11}.
$
So Proposition \ref{R1124} implies the conclusion.
\end{proof}

Next we show the following.

\begin{prop} \label{13}
$R^{13}_{24} = \Z_2\{[\eta_{11}]_{13}\theta\}
\oplus\Z_2\{[\bar{\varepsilon}_9]_{13}\}
\oplus\Z_2\{[\sigma^2_9]_{13}\eta_{23}\}
\oplus\Z_2\{[\iota_7]_{13}\bar{\mu}_7\}$,\ where 
$\Delta(E\theta)=[\eta_{11}]_{13}\theta$. 
\end{prop}
\begin{proof}
In the exact sequence $(24)_{12}$:
$$
\pi^{12}_{25}\rarrow{\Delta}R^{12}_{24}\rarrow{i_*}R^{13}_{24}
\rarrow{p_*}\pi^{12}_{24}\rarrow{\Delta}R^{12}_{23},
$$
$\Delta: \pi^{12}_{24}\to R^{12}_{23}$ is a monomorphism 
by the proof of Lemma \ref{R13} and
$$
\pi^{12}_{25} = \Z_2\{\theta\eta_{24}\}\oplus\Z_2\{(E\theta')\eta_{24}\}.
$$
As $\Delta\theta = [\theta']$ and $\Delta(E\theta')
= [P(E\theta)+\nu_6\kappa_9]_{12}$ by Lemma \ref{R13}, we have
\[
\Delta(\theta\eta_{24}) = [\theta']\eta_{23}
\quad\text{and}\quad \Delta((E\theta')\eta_{24})
= [P(E\theta)+\nu_6\kappa_9]_{12}\eta_{23}.
\]

By the fact that $\Delta\iota_{13} = [\eta_{11}]_{13}$
(\cite[Table 3]{Ka}, \eqref{Ke2}),  we have
$\Delta(E\theta) = [\eta_{11}]_{13}\theta$. 
This completes the proof.
\end{proof}

Next we show the following.

\begin{prop} \label{1424}
$R^{14}_{24} = \Z_8\{[\zeta_{13}]\}
\oplus\Z_2\{[\bar{\varepsilon}_9]_{14}\}
\oplus\Z_2\{[\sigma^2_9]_{14}\eta_{23}\}
\oplus\Z_2\{[\iota_7]_{14}\bar{\mu}_7\}$, where $J[\zeta_{13}]
\equiv \pm\zeta^* \bmod E\pi^{13}_{37}$.
\end{prop}
\begin{proof}
In the exact sequence $(24)_{13}$:
$$
\pi^{13}_{25}\rarrow{\Delta}R^{13}_{24}\rarrow{i_*}R^{14}_{24}
\rarrow{p_*}\pi^{13}_{24}\rarrow{\Delta}R^{13}_{23},
$$
we know that
$\Delta: \pi^{13}_{24}\to R^{13}_{23}$ is trivial by Lemma \ref{zeta1}
and $\pi^{13}_{24} = \Z_8\{\zeta_{13}\}$ \cite[Theorem 7.4]{T}. 
We also know $\pi^{13}_{25} = \Z_2\{E\theta\}$ \cite[Theorem 7.6]{T}.
By Lemma \ref{zeta1}, we have 
$\Delta\zeta_{14} = \pm 2[\zeta_{13}]$ and 
$\sharp[\zeta_{13}]=8$. 
This implies the group $R^{14}_{24}$. By \cite[(3.10)]{MMO}, $H(\zeta^*)
= \pm\zeta_{27}$. So we have $J[\zeta_{13}]
\equiv\pm\zeta^* \bmod E\pi^{13}_{37}$. This completes the proof.
\end{proof}

By the facts $R^\infty_{24}\cong \Z_2$ \cite{Bo} and \cite[Table 1]{HoM},
we have Table \ref{R24}.
\begin{table}[H]
\centering
\begin{tabular}{|c|c||c|c||c|c|}\hline
$R^{14}_{24}$ & $\Z_2 \oplus \Z_8 \oplus (\Z_2)^2$ 
 & $R^{15}_{24}$ & $\Z_2 \oplus (\Z_2)^4$
 & $R^{16}_{24}$ & $\Z_2 \oplus (\Z_2)^7$  \\ \hline
$R^{17}_{24}$ & $\Z_2 \oplus (\Z_2)^4$ 
 & $R^{18}_{24}$ & $\Z_2 \oplus \Z_{16} \oplus \Z_2$
 & $R^{19}_{24}$ & $\Z_2 \oplus \Z_2$\\ \hline
\end{tabular}
\caption{$R^{n}_{24}$ for $14\leq n\leq 19$}\label{R24}
\end{table}
By \cite{Ke}, we also have Table \ref{R24K}.
\begin{table}[H]
\centering
\begin{tabular}{|c|c||c|c||c|c|}\hline
$R^{20}_{24}$ & $\Z_2$ 
 & $R^{21}_{24}$ & $\Z_2$
 & $R^{22}_{24}$ & $\Z_4 \oplus \Z_2$  \\ \hline
$R^{23}_{24}$ & $(\Z_2)^2$ 
 & $R^{24}_{24}$ & $(\Z_2)^3$
 & $R^{25}_{24}$ & $(\Z_2)^2$\\ \hline
\end{tabular}
\caption{$R^{n}_{24}$ for $20\leq n\leq 25$}\label{R24K}
\end{table}

By Tables \ref{R24} and \ref{R24K}, we have the group structure of $R^n_{24}$ for $n \ge 14$.
So, hereafter our main work is to research the generators and the relations in 
$R^n_{24}$ for $n \ge 14$.

Now we show the following.

\begin{prop} \label{1524}
$R^{15}_{24}=\Z_2\{[\eta_{14}\mu_{15}]\}
\oplus\Z_2\{[\zeta_{13}]_{15}\}
\oplus\Z_2\{[\bar{\varepsilon}_9]_{15}\}
\oplus\Z_2\{[\sigma^2_9]_{15}\eta_{23}\}
\oplus\Z_2\{[\iota_7]_{15}\bar{\mu}_7\}$, where $J[\eta_{14}\mu_{15}]
\equiv \mu^{*\prime} \bmod E\pi^{14}_{38}$.
\end{prop}
\begin{proof}
In the exact sequence $(24)_{14}$:
$$
\pi^{14}_{25}\rarrow{\Delta}R^{14}_{24}\rarrow{i_*}R^{15}_{24}
\rarrow{p_*}\pi^{14}_{24}\rarrow{\Delta}R^{14}_{23},
$$
$\Delta: \pi^{14}_{24}\to R^{14}_{23}$ is trivial by Proposition \ref{R15}
and $\pi^{14}_{25}= \Z_8\{\zeta_{14}\}$ \cite[Theorem 7.4]{T}. 
By Lemma \ref{zeta1}, we have
$\Delta\zeta_{14}=\pm 2[\zeta_{13}]$. 

By Example \ref{shetep}, $\sharp[\eta_{14}\mu_{15}]=2$. 
Since $HJ[\eta_{14}\mu_{15}] = \eta_{29}\mu_{30} = H(\mu^{*\prime})$ \cite[(3.11)]{MMO},
we have the relation. This completes the proof.
\end{proof}

Next we show the following.

\begin{prop} \label{1624}
$R^{16}_{24} = \Z_2\{[\zeta_{13}]_{16}\}
\oplus\Z_2\{[\sigma^2_9]_{16}\eta_{23}\}
\oplus\Z_2\{[\iota_7]_{16}\bar{\mu}_7\}
\oplus\Z_2\{[\bar{\varepsilon}_9]_{16}\}
\oplus\Z_2\{[\eta_{14}\mu_{15}]_{16}\}
\oplus\Z_2\{[\mu_{15}]\}\oplus\Z_2\{[\eta_{15}]\varepsilon_{16}\}
\oplus\Z_2\{[\eta_{15}]\bar{\nu}_{16}\}$, where $J[\mu_{15}]\equiv \mu^*_{16}
\bmod E\pi^{15}_{39}$.
\end{prop}
\begin{proof}
In the exact sequence $(24)_{15}$:
$$
\pi^{15}_{25}\rarrow{\Delta}R^{15}_{24}\rarrow{i_*}R^{16}_{24}
\rarrow{p_*}\pi^{15}_{24}\rarrow{\Delta}R^{14}_{23},
$$
we know $\Delta: \pi^{15}_{24}\to R^{14}_{23}$ is trivial. Since
$[\eta_{15}]$ is of order $2$ \cite[Proposition 4.1]{Ka}, 
the Toda bracket
$\{[\eta_{15}], 2\iota_{16}, 8\sigma_{16}\}_1$ is well-defined.
We define 
$$
[\mu_{15}]\in\{[\eta_{15}], 2\iota_{16}, 8\sigma_{16}\}_1.
$$
Then, by \eqref{2i8s2i}, we have
\begin{align*}
2[\mu_{15}]\in\{[\eta_{15}], 2\iota_{16}, 8\sigma_{16}\}_1\circ 2\iota_{24}
=[\eta_{15}]\circ E\{ 2\iota_{15}, 8\sigma_{15}, 2\iota_{22}\}=0.
\end{align*}
Hence, we obtain $\sharp [\mu_{15}] = 2$. 

By the relation
$
HJ[\mu_{15}]=\mu_{31}=H(\mu^\ast)$
\cite[(3.11)]{MMO}, we have the relation.
This completes the proof.
\end{proof}

We know $P(\nu^3_{33})\equiv\tilde{\varepsilon}_{16} \bmod\ E^7\pi^9_{33}$ 
\cite[p. 39]{MMO}. 
By \cite[Theorem 1.1.(b)]{MMO}, we have 
$E^{n-9}\pi^9_{33}=\{\delta_n,\bar{\sigma}'_n, \bar{\mu}_n\sigma_{n+17}, \sigma_n\varepsilon^*_{n+7}\}$ for $n\ge 13$. 
By \cite[Lemma 12.15 ii)]{T} and \cite[I-Propositions 3.4(4), 3.5(9)]{Od1}.
we have $\sigma_{11}\varepsilon^*_{18}=\sigma_{11}\omega_{18}\eta_{34}
=\eta_{11}\sigma_{12}\omega_{19}=\eta_{11}\phi_{12}\equiv\delta_{11} \ \bmod\ 
\bar{\mu}_{11}\sigma_{28}$. Hence, we obtain
\begin{equation}\label{Pn333}
P(\nu^3_{33})\equiv\tilde{\varepsilon}_{16} \bmod\ \delta_{16},\bar{\sigma}'_{16}, \bar{\mu}_{16}\sigma_{33}.
\end{equation}

Here we show

\begin{equation}\label{ztrem1}
\Delta(\eta_{16}\varepsilon_{17}) = [\zeta_{13}]_{16}. 
\end{equation} 
\begin{proof}
By Lemma \ref{zeta1}(2), we have 
$$
[\zeta_{13}]_{16}\equiv\Delta(\eta_{16}\varepsilon_{17}) \ \bmod\ 
{i_{14,16}}_*R^{14}_{16}\circ\varepsilon_{16}.
$$
By the fact that $R^{14}_{16}=\{\Delta\nu_{14}, [\iota_7]_{14}\mu_7\}$
\cite[Proposition  4.1 and (4.11)]{Ka}, we have 
$R^{14}_{16}\circ\varepsilon_{16} = \{[\iota_7]_{14}\varepsilon_7\mu_{15}\}$. By \eqref{Dn13}, we have 
$R^{14}_{16} \circ \varepsilon_{16}=0$ and so, the assertion follows. 
\end{proof}

We show the following.

\begin{prop} \label{17}
$R^{17}_{24}=
\Z_2\{[\mu_{15}]_{17}\}
\oplus\Z_2\{[\eta_{15}]_{17}\varepsilon_{16}\}
\oplus\Z_2\{[\eta_{15}]_{17}\bar{\nu}_{16}\}
\oplus\Z_2\{[\sigma^2_9]_{17}\eta_{23}\}
\oplus\Z_2\{[\iota_7]_{17}\bar{\mu}_7\}$,
where 
$\Delta\mu_{16}\equiv [\eta_{14}\mu_{15}]_{16} \bmod
\Delta(\nu^3_{16}), \Delta(\eta_{16}\varepsilon_{17})$,\
$\Delta(\eta_{16}\varepsilon_{17}) = [\zeta_{13}]_{16}$\ 
and\ $\Delta(\nu^3_{16})\equiv [\bar{\varepsilon}_9]_{16} \bmod
[\iota_7]_{16}\bar{\mu}_7$.
\end{prop}
\begin{proof}
In the exact sequence $(24)_{16}$:
$$
\pi^{16}_{25}\rarrow{\Delta}R^{16}_{24}\rarrow{i_*}R^{17}_{24}
\rarrow{p_*}\pi^{16}_{24}\rarrow{\Delta}R^{16}_{23},
$$
$\Delta:\pi^{16}_{24}\to R^{16}_{23}$ is a monomorphism
by the proof of Proposition \ref{R17}. 
By inspecting the calculation of $\pi^{16}_{40}$~\cite[(5.21)]{MMO} and
 \cite[(3.10), p.~39]{MMO}, we know $P : \pi^{33}_{42} \to \pi^{16}_{40}$
 is a monomorphism. So is $\Delta : \pi^{16}_{25} \to R^{16}_{24}$. 
By \eqref{Ke21},
we have
\begin{align*}
\Delta\mu_{16}&\in\Delta\{\eta_{16}, 2\iota_{17}, 8\sigma_{17}\}_1
\subset\{[\eta^2_{14}]_{16}, 2\iota_{16}, 8\sigma_{16}\}\\
&\ni[\eta_{14}\mu_{15}]_{16}\ \bmod\ [\eta^2_{14}]_{16} \pi^{16}_{24}
+ R^{16}_{17}\circ 8\sigma_{17}.
\end{align*}
Since $R^{16}_{17}\cong(\Z_2)^3$ \cite[Theorem 1(iv)]{KM1}, we have
$R^{16}_{17}\circ 8\sigma_{17} = 0$ and $[\eta^2_{14}]_{16}\pi^{16}_{24}
= \{[\eta^2_{14}]_{16}\bar{\nu}_{16}, [\eta^2_{14}]_{16}\varepsilon_{16}\}$.
So we have
$$
\Delta\mu_{16}\equiv [\eta_{14}\mu_{15}]_{16} \bmod
\Delta(\eta_{16}\bar{\nu}_{17}), \Delta(\eta_{16}\varepsilon_{17}).
$$

By \eqref{ztrem1}, we have 
$\Delta(\eta_{16}\varepsilon_{17}) = [\zeta_{13}]_{16}$.

By the relation $\eta_{33}\bar{\nu}_{34}=\nu^3_{33}$ \eqref{t3} and 
\eqref{Pn333}, we have
\[
J\Delta(\eta_{16}\bar{\nu}_{17})=P(\eta_{33}\bar{\nu}_{34})
=P(\nu^3_{33})\equiv \tilde{\varepsilon}_{16} \bmod \delta_{16},\bar{\sigma}'_{16}, \bar{\mu}_{16}\sigma_{33}.
\]
We recall from \cite[(3.9)]{MMO} that
$H(\tilde{\varepsilon}_{10}) = \bar{\varepsilon}_{19}$
and
$2\tilde{\varepsilon}_{10} = \sigma_{10}\nu_{17}\kappa_{20}$.
By using the EHP sequence and the relation $H(\tilde{\varepsilon}_{10})
 = \bar{\varepsilon}_{19} =HJ[\bar{\varepsilon}_9]$, we have
$
J[\bar{\varepsilon}_9]\equiv \tilde{\varepsilon}_{10} \bmod E\pi^9_{33}.
$
So we obtain
\[
J[\bar{\varepsilon}_9]_{12}\equiv \tilde{\varepsilon}_{12} \bmod
\delta_{12}, \sigma_{12}\bar{\mu}_{19}, \bar{\sigma}'_{12}.
\]

By Lemma \ref{[Ea]}, we have 
$[\eta^2_{14}]\bar{\nu}_{16}\in{i_{15}}_*\{\Delta\iota_{14},\eta^2_{13},
\bar{\nu}_{15}\}$.
We obtain 
\[\begin{split}
{p_{14}}_*\{\Delta\iota_{14},\eta^2_{13},\bar{\nu}_{15}\}
&\subset\{2\iota_{13},\eta^2_{13},\bar{\nu}_{15}\}
\supset\{2\iota_{13},\eta_{13},\nu^3_{14}\}\\
&\supset\{2\iota_{13},\eta_{13},\nu_{14}\}\circ\nu^2_{18}\ni 0 
\ \bmod \ 2\pi^{13}_{24}=\{2\zeta_{13}\}.
\end{split}\]
This implies
$$
\{\Delta\iota_{14},\eta^2_{13},\bar{\nu}_{15}\}\ni 0 \ \bmod \ 
\{2[\zeta_{13}]\} + 
{i_{14}}_*R^{13}_{24}=\{[\bar{\varepsilon}_9]_{14}, 
[\sigma^2_9]_{14}\eta_{23}, [\iota_7]_{14}\bar{\mu}_7\}.
$$
This yields
\[
\Delta(\eta_{16}\bar{\nu}_{17})\equiv [\bar{\varepsilon}_9]_{16} \bmod
[\sigma^2_9]_{16}\eta_{23}, [\iota_7]_{16}\bar{\mu}_7.
\]
By Lemma \ref{[bep9]11eta}(2), we know $[\bar{\varepsilon}_9]_{11}\eta_{24}
\equiv 0\ \bmod\ [\iota_7]_{11}\eta_7\bar{\mu}_8$. 
Since $J([\sigma^2_9]_{16}\eta^2_{23})=\psi_{16}\eta^2_{39}=4\xi_{16}\sigma_{34}\ne 0$ \cite[Lemma 3.5]{MiM1}, we obtain 
\[
\Delta(\eta_{16}\bar{\nu}_{17})\equiv [\bar{\varepsilon}_9]_{16} \bmod
[\iota_7]_{16}\bar{\mu}_7.
\]
This completes the proof.
\end{proof}

By \eqref{ztrem1} and \cite[(3.10)]{MMO}, we have
\[
J[\zeta_{13}]_{16}=J\Delta(\eta_{16}\varepsilon_{17})
=P(\eta_{33}\varepsilon_{34}) = E^2\zeta^*.
\]
Since $E^3: \pi^{13}_{37} \to \pi^{16}_{40}$ is a monomorphism \cite[Theorem 1.1(b)]{MMO},
the relation $J[\zeta_{13}]
\equiv \pm\zeta^* \bmod E\pi^{13}_{37}$ from Proposition \ref{1424} implies
\[
J[\zeta_{13}]= \pm\zeta^*.
\]

We show the following.
\begin{prop} \label{1824}
$R^{18}_{24} = \Z_{16}\{\Delta\sigma_{18}\}
\oplus\Z_2\{[\sigma^2_9]_{18}\eta_{23}\}
\oplus\Z_2\{[\iota_7]_{18}\bar{\mu}_7\}$, where
$8\Delta\sigma_{18} = [\mu_{15}]_{18}$.
\end{prop}
\begin{proof}
In the exact sequence $(24)_{17}$,
we know $\Ker\ \{\Delta: \pi^{17}_{24}\to R^{17}_{23}\} = \{2\sigma_{17}\}$
by the proof of Proposition \ref{R18}. 
By \eqref{pD}, $\Delta\sigma_{18}$ is a lift of $2\sigma_{17}$.
Since $P(\sigma_{37})$ is of order $16$ by \cite[(5.36)]{MMO},
 $\sharp\Delta\sigma_{18}=16$. 
By the relation \eqref{Ke2}, we obtain
\[
\Delta\varepsilon_{17} = [\eta_{15}]_{17}\varepsilon_{16}
\quad\text{and}\quad
\Delta\bar\nu_{17} = [\eta_{15}]_{17}\bar\nu_{16}.
\]
Therefore we have 
$
{i_{18}}_\ast R^{17}_{24}=\{[\mu_{15}]_{18}, [\sigma^2_9]_{18}\eta_{23},
[\iota_7]_{18}\bar{\mu}_7\}\cong(\Z_2)^3.
$

By the definition of $[\mu_{15}]$ and \eqref{Ke2}, %and Corollary \ref{Di2n}, 
we have
\[\begin{split}
[\mu_{15}]_{18}
&\in i_{18}\circ i_{17}\circ \{[\eta_{15}], 2\iota_{16}, 8\sigma_{16}\}_1
\subset i_{18}\circ\{[\eta_{15}]_{17}, 2\iota_{16}, 8\sigma_{16}\}_1\\
&= i_{18}\circ\{\Delta\iota_{17},2\iota_{16},8\sigma_{16}\}_1
= -\{i_{18}, \Delta\iota_{17},2\iota_{16}\} \circ 8\sigma_{17}
\ni 8\Delta\sigma_{18}.
\end{split}\]
The indeterminacy is $i_{18}\circ R^{17}_{17} \circ 8\sigma_{17}=0$
by \cite{Bo}. That is, $8\Delta\sigma_{18} = [\mu_{15}]_{18}$.
This completes the proof.
\end{proof}

We know %$8\Delta(\sigma_{18}) = [\mu_{15}]_{18}$ (Proposition \ref{18}) and 
$8\Delta\sigma_{14} = [\mu_{11}]_{14}$ \cite[p.~111]{KM2}.
We propose 
\begin{conj0}
$8\Delta\sigma_{4n+2}=[\mu_{4n-1}]_{4n+2}$ for $n\geq 5$.
\end{conj0}

We know that the stable $J$-image in $\pi^S_{24}(S^0)$ is isomorphic to
$\Z_2$ generated by $\sigma\bar{\mu} = \eta\bar{\rho}$. So, from the fact that $\pi^S_k(S^0) = 0 \ (k = 4, 5)$, 
by using the exact sequence $(24)_n$ \ $(n = 20, 21)$
and \cite{Ke}, we obtain the following.

\begin{equation} \label{Ker1}
R^n_{24} = \Z_2\{[\iota_7]_n\bar{\mu}_{7}\}\  \mbox{for}\  n = 20 \ \mbox{and} \ 21.
\end{equation}

Since $\Delta(\nu_{21}) = [\sigma^2_9]_{21}$, we have 
$[\sigma^2_9]_{21}\eta_{23} = 0$\ and\ $[\sigma^2_9]_{20}\eta_{23}\in \Delta(\pi^{20}_{25}) = 0$. This implies 
$[\sigma^2_9]_{19}\eta_{23}\in\Delta(\pi^{19}_{25}) 
= \{\Delta(\nu^2_{19})\}$.
By \cite[II-Proposition 2.1(10)]{Od1} and \cite[Theorem 1.1(b)]{MMO}, 
we have $J([\sigma^2_9]_{19}\eta_{23}) = \psi_{19}\eta_{42}\equiv
\bar{\sigma}'_{19} + \delta_{19}\ \bmod\ \bar{\mu}_{19}\sigma_{36}$.
This yields\  $\Delta(\nu^2_{19}) = [\sigma^2_9]_{19}\eta_{23}$. 
So, using the exact sequence $(24)_{18}$, we have the following.

\begin{prop} \label{19}
$R^{19}_{24}=\Z_2\{[\iota_7]_{19}\bar{\mu}_{7}\}
\oplus\Z_2\{[\sigma^2_9]_{19}\eta_{23}\}$, where $\Delta(\nu^2_{19})=[\sigma^2_9]_{19}\eta_{23}$.
\end{prop}

The following result is easily obtained and it overlaps with that of \cite{Ke}.

\begin{prop} \label{Ker2}
\begin{enumerate}
\item $R^{22}_{24} = \Z_4\{\Delta\nu_{22}\}
\oplus\Z_2\{[\iota_7]_{22}\bar{\mu}_7\}$.
\item $R^{23}_{24} = \Z_2\{[\eta^2_{22}]\}
\oplus\Z_2\{[\iota_7]_{23}\bar{\mu}_7\}$.
\item $R^{24}_{24} = \Z_2\{[\eta_{23}]\}\oplus\Z_2\{[\eta^2_{22}]_{24}\}
\oplus\Z_2\{[\iota_7]_{24}\bar{\mu}_7\}$, where $\Delta\eta_{24}=[\eta^2_{22}]_{24}$.
\item $R^{25}_{24}= \Z_2\{[\eta_{23}]_{25}\}
\oplus\Z_2\{[\iota_7]_{24}\bar{\mu}_7\}$,
where $\Delta(\iota_{25}) = [\eta_{23}]_{25}$.
\item $R^n_{24} = \Z_2\{[\iota_7]_n\bar{\mu}_7\}$ for $n\geq 26$.
\end{enumerate}
\end{prop}

By using the parallel argument to \eqref{ztrem1}, we obtain
$$
\Delta(\eta_{24}\varepsilon_{25}) = [\zeta_{21}]_{24}.
$$

\section{Determination of odd components}

We shall state the group structure of the $p$-primary components
$\pi_{23+k}(R_n:p)$ for $p$ an odd prime for $k=0, 1$. We know isomorphisms for all $k$;
\begin{gather}\label{R-Sp:odd}
\pi_k(R_{2n+1}:p) \cong \pi_k(Sp(n):p) \ \cite[(3)]{Ha};\\
\label{R-Sp:even}
\pi_k(R_{2n+2}:p) \cong \pi_k(R_{2n+1}:p)\oplus\pi_k(S^{2n+1}:p) \  \cite[Proposition 18.2]{BS}.
\end{gather}
The result is the following.
\begin{prop}
The non-zero part of the odd prime components of $\pi_k(R_n)$ for $k = 23$ and $24$ is isomorphic
to the following group;
\begin{enumerate}
\item For $k = 23$, $\Z_3 (n = 6);\ \Z_3 (n = 14);$
\item For $k = 24$, $\Z_{33} (n = 6);\ \Z_3 (n = 7);\ \Z_3 (n = 8);\
\Z_{15} (n = 10);\ \Z_3 (n = 12); \Z_{63} (n = 14);\ \Z_{15} (n = 18);\ \Z_3 
(n = 22)$.
\end{enumerate}
\end{prop}
\begin{proof}
We have $\pi_{k}(R_1)=\pi_{k}(R_2)=0$ for $k=23$ and $24$ since $R_1$ is a point and $R_2=S^1$.
Let be $m\ge 1$ and $p$ an odd prime.
We will show (1).
By using the relation \eqref{R-Sp:odd} and the fact that $\pi_{23}(Sp(m):p)=0$ from \cite{MT3} and \cite{Bo}, 
we have $\pi_{23}(R_{2m+1}:p)=0$.
By \cite{T} %[Table of $\pi_k(S^n), I and II]
and \cite{MT1}, %[Theorem 16.8]
we know that the non-zero part of the odd prime components of 
$\pi_{23}(S^{2m+1})$ are $\pi_{23}(S^5:3)\cong \Z_3$ and $\pi_{23}(S^{13}:3) \cong \Z_3$. 
Thus by using the relation \eqref{R-Sp:even}, we have (1).

Next, we will show (2).
By using the relation \eqref{R-Sp:odd}, we have
$\pi_{24}(R_3:p)=\pi_{24}(R_5:p)=0$ from $\pi_{24}(Sp(1):p) \cong \pi_{24}(S^3:p) =0$ 
\cite{Mi1} %[p.~325]
and $\pi_{24}(Sp(2):p) =0$ \cite{F}. %[Theorem 1.1]
We recall that $\pi_{24}(Sp(2)) \cong (\Z_2)^2$ and $\pi_{23}(Sp(2)) \cong (\Z_2)^3$ 
by \cite{F} and \cite{MT1}.
By the following exact sequence
\[
\pi_{24}(Sp(2)) \xrightarrow{i_\ast} \pi_{24}(Sp(3)) \xrightarrow{p_\ast} \pi_{24}(S^{11} ) \xrightarrow{\partial} \pi_{23}(Sp(2)),
\]
we have $\pi_{24}(Sp(3):p) \cong \pi_{24}(S^{11}:p)$.
So, by the fact that $\pi_{24}(S^{11})\cong \Z_6\oplus \Z_2$, 
we have $\pi_{24}(Sp(3):3) \cong \Z_3$ and $\pi_{24}(Sp(3):p) = 0$ for
 $p \ge 5$.
Thus we have 
$\pi_{24}(R_7:3) \cong \pi_{24}(Sp(3):3) \cong \Z_3$ and
 $\pi_{24}(R_7:p)= 0$ for $p \ge 5$.
Similarly, we have $\pi_{24}(R_{2m+1}:p)=0$ for $m\ge 4$ by relations
$\pi_{24}(Sp(4):p) =\pi_{24}(Sp(5):p) =0$ from \cite{MT3} and $\pi_{24}(Sp(m):p) =0 $ for $m \ge 6$
from \cite{Bo}. Thus we have (2) for $n=2m+1$.
By \cite{T}, %[Table of $\pi_k(S^n), I, II and III]
for $m\ge 2$, 
we know that the non-zero part of the odd prime components of $\pi_{24}(S^{2m+1})$ 
are $\Z_{33}(m=2)$, %$\Z_3(m=3)$, 
$\Z_{15}(m=4)$, $\Z_3(m=5)$, $\Z_{63}(m=6)$, $\Z_{15}(m=8)$ and $\Z_{3}(m=10)$.
Hence by using \eqref{R-Sp:even} and (2) for $n=2m+1$, we obtain (2) for $n=2m+2$.
This leads to (2) and completes the proof.
\end{proof}

\textit{Acknowledgement:} 
The third author was supported in part by funding from Fukuoka University (Grant No. 195001)

\noindent
Yoshihiro Hirato\\
Department of General Education, Nagano National College of Technology\\
716 Tokuma, Nagano 381-8550, Japan\\
hirato@nagano-nct.ac.jp\\

\noindent
Jin-ho Lee\\
DATA SCIENCE, PULSE9. Inc.\\  
122,  Mapo-daero, Mapo-gu, Seoul, Korea\\
jhlee@pulse9.net\\

\noindent
Toshiyuki Miyauchi\\
Department of Applied Mathematics, Faculty of Science, Fukuoka University\\
8-19-1 Nanakuma, Jonan-ku, Fukuoka 814-0180, Japan\\
tmiyauchi@fukuoka-u.ac.jp\\

\noindent
Juno Mukai\\
Shinshu University\\
3-1-1 Asahi, Matsumoto, Nagano 390-8621, Japan\\
jmukai@shinshu-u.ac.jp\\

\noindent
Mariko Ohara\\
National Institute of Technology, Oshima College\\
1091-1 Komatsu, Suo-Oshima, Oshima, Yamaguchi 742-2193, Japan\\
primarydecomposition@gmail.com\\

\end{document}